\documentclass[article,11pt,reqno]{amsart}
\usepackage[top=25mm,bottom=25mm,left=25mm,right=25mm]{geometry}
\usepackage{color}
\usepackage{amsmath}
\usepackage{amsthm}
\usepackage{amssymb}
\usepackage{amsfonts}
\usepackage{mathrsfs}
\usepackage{calligra}
\usepackage[mathscr]{eucal}
\usepackage{comment}
\usepackage{graphicx}
\usepackage{tikz}
\usetikzlibrary{arrows,decorations.markings}

\usepackage{nicefrac,bm,upgreek,mathtools,verbatim}

\newtheorem{lemma}{Lemma}[section]
\newtheorem{proposition}[lemma]{Proposition}
\newtheorem{theorem}[lemma]{Theorem}
\newtheorem{corollary}[lemma]{Corollary}
\newtheorem{definition}[lemma]{Definition}
\newtheorem{remark}[lemma]{Remark}

\makeatletter
\def\section{\@startsection{section}{1}%
\z@{1\linespacing\@plus\linespacing}{1\linespacing}%
{\bf\centering}}
\def\subsection{\@startsection{subsection}{0}%
\z@{\linespacing\@plus\linespacing}{\linespacing}%
{\bf}}
\def\subsubsection{\@startsection{subsubsection}{0}%
\z@{\linespacing\@plus\linespacing}{\linespacing}%
{\bf}}
\makeatother

\makeatletter
\@addtoreset{equation}{section}
\makeatother

\DeclareMathOperator{\Dom}{Dom}
\DeclareMathOperator{\Ran}{Ran}
\DeclareMathOperator{\Spec}{Spec}

\DeclareMathOperator{\Per}{Per}

\DeclareMathOperator*{\slim}{s-lim}

\newcommand{\vertiii}[1]{{\left\vert\kern-0.25ex\left\vert\kern-0.25ex\left\vert #1
		\right\vert\kern-0.25ex\right\vert\kern-0.25ex\right\vert}}
\newcommand{\Norm}[2]{\left\Vert #1 \right\Vert_{#2}}
\newcommand{\NNorm}[2]{\vertiii{#1}_{#2}}

\providecommand{\seq}[1]{(#1_n)_{n\in \mathbb{N}}}

\newcommand{\cS}{\mathcal{S}}

\newcommand{\cD}{\mathcal{D}}
\newcommand{\cA}{\mathcal{A}}
\newcommand{\cL}{\mathscr{B}}
\newcommand{\cR}{\mathcal{R}}
\newcommand{\cK}{\mathcal{K}}
\newcommand{\cI}{\mathcal{I}}

\newcommand{\cF}{\mathcal{F}}

\newcommand{\R}{\mathbb{R}}

\newcommand{\N}{\mathbb{N}}
\definecolor{mr}{rgb}{0.1,0.2,0.7}

\begin{document}
\title[Fractional Rollnik class]
{\small Special potentials for relativistic Laplacians I: Fractional Rollnik-class}
\author{Giacomo Ascione, Atsuhide Ishida and J\'ozsef L{\H o}rinczi}

\address{Giacomo Ascione,
Dipartimento di Matematica e Applicazioni ``Renato Caccioppoli" \\
Universit\`a degli Studi di Napoli Federico II, 80126 Napoli, Italy}
\email{giacomo.ascione@unina.it}
\address{Atsuhide Ishida,
Katsushika Division, Institute of Arts and Sciences \\
Tokyo University of Science, Tokyo, 125-8585, Japan \\
and \\
Alfr\'ed R\'enyi Institute of Mathematics \\
Re\'altanoda utca 13-15, 1053 Budapest, Hungary}
\email{aishida@rs.tus.ac.jp}
\address{J\'ozsef L\H orinczi,
Alfr\'ed R\'enyi Institute of Mathematics \\
Re\'altanoda utca 13-15, 1053 Budapest, Hungary}
\email{lorinczi@renyi.hu}

\begin{abstract}
We propose a counterpart of the classical Rollnik-class of potentials for fractional and massive relativistic
Laplacians, and describe this space in terms of appropriate Riesz potentials. These definitions rely
on precise resolvent estimates, which we present in detail. We obtain these classes for the operators with
fractional exponent $\alpha = 1$ in low dimensions, making use of tools of $\Gamma$-convergence. We show that
Coulomb-type potentials are elements of fractional Rollnik-class up to but not including the critical singularity
of the Hardy potential. In a second part of the paper we derive detailed results on the self-adjointness and
spectral properties of relativistic Schr\"odinger operators obtained under perturbations by fractional Rollnik
potentials. We also define an extended fractional Rollnik-class which is the maximal space for the Hilbert-Schmidt
property of the related Birman-Schwinger operators.

\bigskip
\noindent
\emph{Key-words}: fractional and relativistic Laplacians, Riesz potential, resolvent estimates,
$\Gamma$-convergence, Coulomb potentials, subcriticality

\bigskip
\noindent
2010 {\it MS Classification}: Primary 47D08, 60G51; Secondary 47D03, 47G20

\end{abstract}

\maketitle

\baselineskip 0.55 cm
\tableofcontents


\section{Introduction}
The aim of this paper is to define and explore a relativistic counterpart of the classical
Rollnik-class of potentials, which plays a fundamental role in mathematical quantum theory.
We call this fractional Rollnik-class, and also set ourselves the task to describe the
properties of relativistic Schr\"odinger operators obtained as perturbations of the (free)
relativistic operators by this class of potentials. Since both the massive and massless cases
involve singular integral operators, in our approach the Riesz potential will play a role
to a far-reaching extent. In particular, we will make use of it to a detailed description
of Coulomb-type potentials belonging to fractional Rollnik-class and investigate their
boundaries verging on the critical Hardy potential.

In the paper \cite{R56} Rollnik identified the class of potentials
\begin{equation}
\label{class}
{\mathcal R} = \left\{V:\R^3\to\R: \, \|V\|^2_{{\mathcal R}} := \int_{\R^3}\int_{\R^3}
\frac{|V(x)||V(y)|}{|x-y|^{2}}dxdy < \infty \right\},
\end{equation}
which later turned out to play a special role in the classical theory of Schr\"odinger operators on
the Hilbert space $L^2(\R^3)$. He noticed that for an eigenvalue $\lambda < 0$ generated by a negative
potential $V$, the eigenvalue equation $H\varphi:=(-\frac{1}{2}\Delta + V)\varphi = \lambda\varphi$
can equivalently be written as $K_\lambda\psi:=\sqrt{-V}(-\frac{1}{2}\Delta - \lambda)^{-1}
\sqrt{-V}\psi=\psi$, with $\psi = \sqrt{-V}\varphi$, so the original Schr\"odinger operator $H$ has
a eigenvalue $\lambda$ exactly when the related operator $K_\lambda$ has eigenvalue 1.
Thereby he appears to have anticipated by a five years (if we count by the publication dates) what is
now known as the Birman-Schwinger principle.
In fact, Rollnik studied mostly the case $\lambda = 0$ and required his operator $K_0$ (in his notation
$W(0)$) to be compact, that he satisfied by turning its kernel square-integrable, which results in
\eqref{class}.

The integral condition \eqref{class} has appeared in subsequent work by Schwinger \cite{S61}, Scadron,
Weinberg and Wright \cite {SWW}, Coester \cite{C}, and in the series of papers by Grossmann and Wu
\cite{GW}. A later development is Simon's mathematical work in the paper \cite{S71} and the monograph
\cite{Si1}, in which he made a comprehensive analysis of Rollnik potentials in the non-relativistic case.
Some more recent related works include \cite{D01,RS04,W05}.

Our goal here is to construct and analyse a counterpart of Rollnik-class for the massive and massless
relativistic Laplacians
$$
L_{m,\alpha} = (-\Delta + m^{2/\alpha})^{\alpha/2}-m, \quad 0 < \alpha < 2, \; m\geq 0,
$$
on $L^2(\R^d)$, $d \geq 1$, and study relativistic Schr\"odinger operators obtained as perturbations of
$L_{m,\alpha}$ by these potentials. For the massless case $m=0$ the operator $L_{m,\alpha}$ reduces to the
fractional Laplacian with exponent $\frac{\alpha}{2}$, and for $\alpha =1$ and $m > 0$ it is the physical
relativistic operator widely used in mathematical physics. For the purposes of understanding fractional
Rollnik-class it will be useful to keep the full range $\alpha \in (0,2)$ in order to appreciate how our
results (qualitatively) depend on the details of the operators of different fractional exponent. Also,
exponents $\alpha$ different from 1 relate with further models in physics, see \cite{KL}.  Our guiding
principle is to keep to the original idea that the Birman-Schwinger operators associated with the non-local
Schr\"odinger operators  should have the Hilbert-Schmidt property, however, for the rest we will
have to trod a new path due to qualitative differences between the Laplacian and the fractional operators
requiring very different techniques.

In the following we summarise our approach and the main results of this paper.

\vspace{0.2cm}
\noindent
(1) Our first task is to find an appropriate concept of a fractional Rollnik-class holding to this principle
and investigate its structure. The construction will depend on precise estimates of the resolvent kernels of
the operators $L_{m,\alpha}$. Making use of estimate \eqref{forroll} and assuming
\begin{equation}
\label{restri}
\alpha < d < 2\alpha,
\end{equation}
we introduce the fractional Rollnik-class
$$
\mathcal{R}_{d,\alpha}=\left\{V:\mathbb{R}^d\rightarrow\mathbb{R} : \, \|V\|^2_{\mathcal{R}_{d,\alpha}}=
\int_{\mathbb{R}^{d}} \int_{\mathbb{R}^{d}}\frac{|V(x)||V(y)|}{|x-y|^{2(d-\alpha)}}dxdy<\infty \right\}
$$
in the massless case and
$$
\mathcal{R}^m_{\alpha} = \mathcal{R}_{d,\alpha} \cap \mathcal{R}
$$
in the massive case (see Definition \ref{fRollnik}). This definition is consistent with the classical case for
it reproduces $\mathcal R$ for $d=3$ and $\alpha = 2$ (see also Corollary \ref{cor:BBM} below). Observing that
the Fourier transform of the Riesz potential $\cI_\beta$ of functions on Lizorkin space can be computed (Theorem
\ref{thm:LizRi}), we show that the fractional Rollnik norm of $V$ can be expressed in the form
\begin{equation}
\label{rieszplay}
\Norm{V}{\cR_{d,\alpha}} = c(d,\alpha) \Norm{\cI_{\alpha-\frac{d}{2}}|V|}{2}
\end{equation}
where $c(d,\alpha)$ is a constant and the left-hand side above is indeed a norm (Theorem \ref{norms}),
turning $\mathcal{R}_{d,\alpha}$ into a Banach space (Theorem \ref{banach}). Furthermore, as it will be seen,
fractional Rollnik-class contains a variety of $L^p$ subspaces (Propositions \ref{cor7}-\ref{cor9}).

To get a closer idea of fractional Rollnik potentials, we consider several specific cases (Section 3.3). In
particular, we consider potentials having Coulomb-type singularities and by using \eqref{rieszplay} and formula
\eqref{eq:radialexpr} allowing to compute Riesz potentials of radial functions, we show that $V(x) = |x|^{-\beta_1}
\mathbf{1}_{\{|x|<1\}}+|x|^{-\beta_2}\mathbf{1}_{\{|x| \ge 1\}}$ belongs to $\cR_{d,\alpha}$ if and only if
$\beta_1 <\alpha$ and $\beta_2 \in (\alpha,d)$ (see Propositions \ref{prop:Coul1}-\ref{prop:Coul3} and Corollary
\ref{coulombsummary}). Since for the fractional Laplacian $(-\Delta)^{\alpha/2}$ the critical potential is the
Hardy potential $V(x) = |x|^{-\alpha}$, we see that the Coulomb-type fractional Rollnik-class potentials are
subcritical with possible singularities arbitrarily approaching but not attaining the critical exponent. Also,
we show (Theorem \ref{nomore}) that the Hardy exponent marks the boundary to Coulomb-potentials in fractional
Rollnik class, and $V(x)=|x|^{-\alpha}(-\log|x|)^{-\gamma}\mathbf{1}_{\{|x|<1/2\}}$, with $\alpha
< d$, belongs to $\cR_{d,\alpha}$ exactly when $\gamma>\frac{1}{2}$ (Proposition \ref{prop:logV}). Apart from
Coulomb singularities, we also discuss potentials with compact support and show that their fractional Rollnik norm
increases under symmetric rearrangement and it can be bounded by the fractional perimeter of the supporting set
(Section 3.3.1), and asymmetric potentials whose non-radial symmetry deficit we can bound by the $L^1$-norm of
the difference between the potential and a family of rotationally symmetric potentials (Section 3.3.2).

Using another key resolvent estimate (see \eqref{forrollext1}-\eqref{eq:Hda0} and
\eqref{eq:equivrollext1}-\eqref{eq:Hmda} below) we also introduce extended fractional Rollnik-classes
${\overline {\mathcal{R}}}_{d,\alpha}$ and ${\overline {\mathcal{R}}}^m_{d,\alpha}$ (Definition \ref{extrollnik}),
containing $\mathcal{R}_{d,\alpha}$ as a proper subset (Proposition \ref{includes} and Remark \ref{genuine} below).
We will discuss the role of the extended fractional Rollnik-class further below.

The restriction \eqref{restri} has the implication that in case $d \geq 2\alpha$ no non-trivial continuous potential
exists in $\mathcal{R}_{d,\alpha}$ (Theorem \ref{thm:nonL1kernel}). As frequently with fractional Laplacians, the
$\alpha = 1$ case is difficult. In this case there exists no fractional Rollnik-class potential in any dimension
directly, which is an essential difference from the classical situation. However, mathematically this case remains
very interesting and from a physics point of view it
is satisfying that in low dimensions we can make sense of it by further developing Definition \ref{fRollnik}. In
$d=1$ we cover it by the extended fractional Rollnik-class, and for $d=2$ we obtain it as a limiting case. Proposition
\ref{prop:notappr} shows, however, that the required limit is not straightforward. As it turns out, with a sequence
$\seq \alpha$ accumulating at $\alpha = 1$, the corresponding Schr\"odinger operators with
$\mathcal{R}_{d,\alpha_n}$-potentials have a limit in $\Gamma$-convergence sense (see Theorem \ref{alpha1sa}), and
we arrive to define it by $\cR_{d,1}^p:= L^p(\R^d) \cap \liminf_{\alpha \downarrow 1}\cR_{d,\alpha}$ with $p > 1$.

\vspace{0.2cm}
\noindent
(2) 
Our interest next is to study the properties of relativistic Schr\"odinger
operators of the form $H_{m,\alpha} = L_{m,\alpha} + V$ with $V$ in $\mathcal{R}^m_{d,\alpha}$ or ${\overline
{\mathcal{R}}}^m_{d,\alpha}$, for $m\geq 0$. First we consider the self-adjointness of $H_{m,\alpha}$. To obtain this it
suffices to have the operator norm of $K^m_\lambda=|V|^{1/2}(\lambda+L_{m,\alpha})^{-1} |V|^{1/2}$ vanish as $\lambda
\to\infty$ (Lemma \ref{lem:selfadj}). To ensure this, we show that $K^m_\lambda$ is a Hilbert-Schmidt operator at
least for a large enough $\lambda$. In particular, we show that for potentials in $\mathcal{R}_{d,\alpha}$ the operator
$H_{m,\alpha}$ is self-adjoint (Theorem \ref{prop:uniformHS}). Also, we show that if a potential is in the extended
fractional Rollnik-class ${\overline {\mathcal{R}}}_{d,\alpha}$, then the related $K^m_\lambda$ is a Hilbert-Schmidt
operator, and conversely, if $K^m_\lambda$ is a Hilbert-Schmidt operator for a potential, then the potential is in
${\overline {\mathcal{R}}}_{d,\alpha}$ (Theorem \ref{thm:HSKlambda}), thus the extended fractional Rollnik-class is
the maximal space for the Hilbert-Schmidt property to hold. 

Next we study the spectrum of $H_{m,\alpha}$. Since fractional Rollnik potentials are small in an aggregate sense
(without being subject to any specific fall-off rate), a perturbation by such potentials does not change the essential
spectrum of $L_{m,\alpha}$ (see Theorem \ref{the12}). However, contrary to the classical case, we can not make use of
Fourier methods and develop an other argument to show this. In case $H_{m,\alpha}$ has a non-empty discrete spectrum,
we can estimate the number of negative eigenvalues (Theorem \ref{the13}); since this number is controlled above by the
fractional Rollnik norm, the discrete spectrum is finite in both the massless and massive cases. We obtain similar
results for ${\overline {\mathcal{R}}}_{d,\alpha}$-class potentials (Theorems \ref{the13extended}-\ref{the13extended2}),
where now the control is made by the quantity $[V]_{\overline{\cR}_{d,\alpha}}$, which defines the extended class. By
using the extended Birman-Schwinger principle, we can further derive conditions under which no bound states exist and
zero (the spectral edge) is not an embedded eigenvalue. This occurs when then fractional Rollnik norm is small enough
(Theorem \ref{none}).

\enlargethispage{1cm}
\vspace{0.2cm}
\noindent
(3) The main estimates on the resolvent $R^m_\lambda = (\lambda+L_{m,\alpha})^{-1}$, $\lambda > 0$, (uniform in $\lambda$,
pointwise upper and lower for large $\lambda$ and for $\lambda= 0$ both for the massive and massless cases) on which our
constructions are built are stated in the expressions \eqref{forroll}-\eqref{eq:equivrollmassive},
\eqref{forrollext1}-\eqref{eq:Hda0} and \eqref{eq:equivrollext2}-\eqref{eq:Hmda}. Detailed proofs will be presented in
the Appendix. We note that these results have an independent interest, in areas such as analytic and probabilistic
potential theory.

\vspace{0.2cm}
Finally, we point out that in the classical theory of Schr\"odinger operators there are several competing spaces of
potentials suiting different purposes. For instance, Kerman-Sawyer and Stummel-Schechter classes can be used to obtain
improved conditions on self-adjointness, and Kato-class is a natural space for Feynman-Kac representations, which yield
detailed results also on ground state properties. We note that the fractional Rollnik-class introduced in this paper has
an overlap with (relativistic/fractional) Kato-class, however, neither space contains the other (see Proposition \ref{kato}
and its consequences below). Thus the detailed results that we obtain under the Rollnik integral condition not just on
self-adjointness, but also on the spectral properties of such relativistic Schr\"odinger operators  cannot be fully covered
under the conditions of relativistic Kato-class. In forthcoming papers, we will further study the relativistic counterparts
of the above mentioned potential spaces.

\section{Preliminaries}
\subsection{Massive and massless relativistic Laplacians}
Let $\alpha\in (0, 2)$, $m\geq 0$, $d \in \N$,  and consider the family of functions $\Psi_{m,\alpha}(u)=
(u+m^{2/\alpha})^{\alpha/2}-m$ for every $u \ge 0$. Using these functions, we define the operators
\begin{eqnarray*}
L_{m,\alpha} \!\!\! &=& \!\!\! \Psi_{m,\alpha}(-\Delta) =(-\Delta+m^{2/\alpha})^{\alpha/2}-m  \qquad \mbox{if $m>0$}\\
L_{0,\alpha} \!\!\! &=& \!\!\! \Psi_{0,\alpha}(-\Delta) = (-\Delta)^{\alpha/2}  \hspace{3cm} \mbox{if $m=0$}
\end{eqnarray*}
and generically combine the notation into just $L_{m,\alpha}$, $m \geq 0$, whenever suitable. These operators can be defined
in several different ways. We define them through the Fourier multipliers
\begin{equation*}
\widehat{(L_{m,\alpha} f)}(y) = \Psi_{m,\alpha}(|y|^2)\widehat f(y), \quad y \in \R^d, \; f \in
\Dom(L_{m,\alpha}),
\end{equation*}
with domain
\begin{equation*}
\Dom(L_{m,\alpha})=\Big\{f \in L^2(\R^d): \Psi_{m,\alpha}(|\cdot|^2) \widehat f \in L^2(\R^d) \Big\},
\quad m\geq 0.
\end{equation*}
In the $m=0$ case we have more specifically $\Dom(L_{0,\alpha}) = H^\alpha(\R^d)$, i.e., the fractional Sobolev space of
order $0 < \alpha < 2$.

For $f \in C^\infty_{\rm c}(\R^d)$ we have the expression
\begin{equation*}
L_{m,\alpha}f(x) = -\lim_{\varepsilon\downarrow 0} \int_{|y-x|>\varepsilon} \left(f(y)-f(x)\right)\nu_{m,\alpha}(dy)
\end{equation*}
with the L\'evy jump measures as follows. For $m>0$ we have (see, for instance, \cite[Equation $(2.2)$]{CKS12})
\begin{equation}
\label{levym}
\nu_{m,\alpha}(dx) = j_{m,\alpha}(|x|)dx = \frac{2^{\frac{\alpha-d}{2}} m^{\frac{d+\alpha}{2\alpha}} \alpha}
{\pi^{d/2}\Gamma(1-\frac{\alpha}{2})}
\frac{K_{(d+\alpha)/2} (m^{1/\alpha}|x|)}{|x|^{(d+\alpha)/2}} \, dx, \quad x \in \R^d\setminus \{0\},
\end{equation}
and for $m=0$ we have (see, for instance, \cite[Equation (1.22)]{BBKRSV09})
\begin{equation}
\label{levy0}
\nu_{0,\alpha}(dx) = j_{0,\alpha}(|x|)dx = \frac{2^\alpha \Gamma(\frac{d+\alpha}{2})}{\pi^{d/2}|\Gamma(-\frac{\alpha}{2})|}
\frac{dx}{|x|^{d+\alpha}}, \quad x \in \R^d\setminus \{0\}.
\end{equation}
Here
\begin{equation*}
\label{bessel3}
K_\rho (z) = \frac{1}{2} \left(\frac{z}{2}\right)^\rho \int_0^\infty t^{-\rho - 1} e^{-t-\frac{z^2}{4t}} dt, \quad \rho > 0,
\end{equation*}
is the standard modified Bessel function of the second (sometimes also called third) kind. Below we will make use of the
asymptotic formulae (see \cite[Ch. 9, eqs. 9.6.6, 9.7.2]{AS})
\begin{equation}
\label{besselasymp}
K_\rho (z) \sim \sqrt{\frac{\pi}{2z}}\, e^{-z} \quad \mbox{as $z\to\infty$}
\qquad \mbox{and} \qquad
K_\rho(z)\sim \frac{\Gamma(\rho)}{2}\left(\frac{2}{z}\right)^{\rho} \quad \mbox{as $z \downarrow 0$}.
\end{equation}
A consequence of the second expression is that $j_{m,\alpha} \to j_{0,\alpha}$ as $m \downarrow 0$, thus the operator $L_{0,\alpha}$
is the massless limit of $L_{m,\alpha}$.

The operators $L_{m,\alpha}$, $ m\geq 0$, are positive with spectrum $\Spec L_{m,\alpha} = \Spec_{\rm ess} L_{m,\alpha}
= [0,\infty)$, and self-adjoint with core $C^\infty_{\rm c}(\R^d)$, for every $0< \alpha < 2$ and $m \geq 0$. In case $m
> 0$, the operator $L_{m,\alpha}$ is called \emph{relativistic fractional Laplacian with exponent $\frac{\alpha}{2}$ and
rest mass $m$}, while for $m=0$ it is the \emph{fractional Laplacian with exponent $\frac{\alpha}{2}$}, but for simplicity
we will refer to them as the massive and massless relativistic Laplacians, respectively.

Below we will use the following generic notations. $L^p(\R^d)$ norms will be denoted by $\|\cdot\|_p$ and $L^p(E)$, $E
\subsetneq \R^d$, norms by $\|\cdot\|_{L^p(E)}$. A ball centred in $x \in \R^d$ with radius $r > 0$ will be denoted by
$B_r(x)$, and simply $B_r$ when $x=0$. We write $\omega_d$ for the Lebesgue measure of $B_1$ and $\sigma_d$ for the
surface measure of $\partial B_1$, so that $\sigma_d=d\omega_d$. When convenient, we will also use the notations
$a \wedge b = \min\{a,b\}$ and $a \vee b = \max\{a,b\}$.

\subsection{Lizorkin space and Riesz potential}
Consider the negative Laplacian $-\Delta$ and for $\beta \in (0,d)$ let
\begin{equation*}
\label{gammad}
\gamma_d(\beta)=\frac{2^\beta \pi^{d/2}\Gamma\big(\frac{\beta}{2}\big)}{\Gamma\big(\frac{d-\beta}{2}\big)}.
\end{equation*}
The Riesz potential of order $0<\beta<d$ of a measurable function $f$ is defined as
\begin{equation*}
\label{ries}
\mathcal{I}_\beta f(x) = (-\Delta)^{-\beta/2} f(x) =
\frac{1}{\gamma_d(\beta)}\int_{\R^d}\frac{f(y)}{|x-y|^{d-\beta}}dy = (k_\beta \ast f)(x).
\end{equation*}
The requirement $\beta>0$ is necessary in order to guarantee that the Riesz kernel $k_\beta$ is at least
$L^1_{\rm loc}(\R^d)$. It is clear that the integral defining $\mathcal{I}_\beta f$ is absolutely convergent
almost everywhere if and only if $f$ is taken from
\begin{equation}
\label{eq:Xbd}
{\mathfrak X}_{\beta}^d := \Big\{f: \mathbb R^d \to \mathbb C: \mbox{$f$ is measurable and} \, \int_{|x|<1}|f(x)|dx
+\int_{|x| \ge 1}\frac{|f(x)|}{|x|^{d-\beta}}dx<\infty \Big\}.
\end{equation}
If $f$ belongs to the Schwartz space $\mathcal{S}(\R^d)$ of rapidly decreasing functions, its Riesz potential
$\mathcal{I}_\beta f$ is well-defined since $\mathcal{S}(\R^d)\subset {\mathfrak X}_\beta^d$. However, it is
not guaranteed that $\mathcal{I}_\beta f \in L^1(\R^d)+L^2(\R^d)$. To ensure this, we need to consider $f$ in
a suitable subspace
of $\mathcal{S}(\R^d)$. Let $\N_0^d$, where $\N_0=\{0\} \cup \N$, be the set of all multi-indices and for all
$\mathbf{j} \in \N^d$ denote by $D^{\mathbf{j}}h$ the partial derivative of $h$ taken $j_i$-times with respect
to the $x_i$ variable, for all $i=1, \dots,d$. Following the lines of \cite[Ch. 2.2]{Samko}, we introduce the
spaces
\begin{equation*}
\mathcal{S}_0(\R^d):= \left\{h \in \mathcal{S}(\R^d): \ D^{\mathbf{j}}h(0)=0, \ \forall \mathbf{j} \in \N_0^d\right\}
\end{equation*}
and
\begin{equation*}
\mathcal{FS}_0(\R^d):= \left\{h \in \mathcal{S}(\R^d): \ h=\cF[g], \ g \in \mathcal{S}_0\right\},
\end{equation*}
where $\mathcal{F}$ denotes Fourier transform. The set $\mathcal{FS}_0(\R^d)$ is called Lizorkin space. Since $\cF$
maps $\cS(\R^d)$ into itself, it is clear that $\mathcal{FS}_0(\R^d) \subset \cS(\R^d)$.  In particular, it is not
difficult to check that $\mathcal{FS}_0(\R^d)$ is formed by all (complex-valued) functions in $\mathcal{S}(\R^d)$
that are orthogonal to all polynomials. We will make use of the following denseness result, which was shown in
\cite[Th. 2.7]{Samko}.
\begin{theorem}
\label{thm:dense}
The Lizorkin space $\mathcal{FS}_0(\R^d)$ is dense in $L^p(\R^d)$ for all $p \in (1,\infty)$.
\end{theorem}
For our purposes below the following connection between Lizorkin space and the Riesz potential will be important, for
details we refer to \cite[Lem. 2.15, Th. 2.16]{Samko}.
\begin{theorem}
\label{thm:LizRi}
Let $\beta \in (0,d)$. If $h \in \mathcal{FS}_0(\R^d)$, then $\cI_\beta h \in \cF\cS_0(\R^d)$ and
\begin{equation*}
\cF[\cI_\beta h](z)=|z|^{-\beta}\cF[h](z).
\end{equation*}
\end{theorem}

In the special case of radially symmetric functions $f:\R^d \to \mathbb{C}$ written in the form $f(x)=g(|x|)$
with some $g:[0,\infty) \to \mathbb{C}$, it has been shown in \cite[Th. 1]{Thim16} that the Riesz potential
$\mathcal{I}_\beta f$ is radially symmetric and for $d \ge 2$ the expression
\begin{equation}
\label{eq:radsympot}
\mathcal{I}_\beta f(x)=\frac{1}{\gamma_d(\beta)}\int_0^{\infty}r^{\beta-1}F_\beta\left(\frac{|x|}{r}\right)g(r)dr
\end{equation}
holds, where
\begin{equation}
\label{Fbeta}
F_\beta(r)=\sigma_{d-1}\int_{-1}^{1}\frac{(1-t^2)^{\frac{d-3}{2}}}{(1+r^2-2rt)^{\frac{d-\beta}{2}}}\, dt, \quad r>0.
\end{equation}
The asymptotic behaviour of the function $F_\beta$ is known. First, it is continuous for $r \in (0,1) \cup (1,\infty)$,
and for $\beta>1$ it is continuous also in $r=1$ with value
\begin{equation*}
F_\beta(1)=\sigma_{d-1}2^{\beta-1}B\left(\frac{d-1}{2},\frac{\beta-1}{2}\right),
\end{equation*}
where $B$ is the Beta-function. Also, if $\beta=1$, then
\begin{equation*}
\lim_{r \to 1}\frac{F_\beta(r)}{-\sigma_{d-1}\log(|r-1|)}=1
\end{equation*}
while if $\beta \in (0,1)$, then
\begin{equation*}
\lim_{r \to 1}\frac{2F_\beta(r)}{\sigma_{d-1}B\left(\frac{d-1}{2},\frac{1-\beta}{2}\right)|r-1|^{\beta-1}}=1.
\end{equation*}
For its behaviour at zero we have
\begin{equation}\label{eq:to0}
F_\beta(r)=\sigma_d+O(r),\  \mbox{ as } r\to 0
\end{equation}
while at infinity
\begin{equation}\label{eq:toinf}
F_\beta(r)=\sigma_dr^{\beta-d}+O(r^{\beta-d-1}),\  \mbox{ as } r\to \infty.
\end{equation}
We note that in \cite[Eq. (2.10)]{Samko} the same formula has been obtained in terms of hypergeometric functions.
However, the above representation is preferable for our purposes since usually the analytic extension of the
hypergeometric function is considered on the cut complex plane $\mathbb{C} \setminus (1,\infty)$. This is due to
the fact that $z=1$ and $z=\infty$ are branch points of this function.

The $d=1$ case is not covered by \eqref{Fbeta}, however, we can derive a similar type of representation. Let
$f:\R \to \mathbb{C}$ be even. First observe that
\begin{align*}
\mathcal{I}_\beta f(-x)=\frac{1}{\gamma_1(\beta)}\int_{\R}\frac{f(y)}{|x+y|^{1-\beta}}\,dy
=\frac{1}{\gamma_1(\beta)}\int_{\R}\frac{f(y)}{|x-y|^{1-\beta}}\,dy=\mathcal{I}_{\beta}(x),
\end{align*}
thus $\mathcal{I}_\beta f$ is also even. Next, notice that for $x>0$
\begin{align*}
\mathcal{I}_\beta f(x)&=\frac{1}{\gamma_1(\beta)}\int_{\R}\frac{f(y)}{|x-y|^{1-\beta}}\,dy=\frac{1}{\gamma_1(\beta)}
\left(\int_{0}^{\infty}\frac{f(y)}{|x-y|^{1-\beta}}\,dy+\int_{-\infty}^{0}\frac{f(y)}{|x-y|^{1-\beta}}\,dy\right)\\
&=\frac{1}{\gamma_1(\beta)}\left(\int_{0}^{\infty}\frac{f(y)}{|x-y|^{1-\beta}}\,dy
+\int_{0}^{\infty}\frac{f(y)}{|x+y|^{1-\beta}}\,dy\right)\\
&=\frac{1}{\gamma_1(\beta)}\int_{0}^{\infty}y^{\beta-1}f(y)\left(\frac{1}{\left|\frac{x}{y}-1\right|^{1-\beta}}
+\frac{1}{\left|\frac{x}{y}+1\right|^{1-\beta}}\right)\,dy.
\end{align*}
This means that \eqref{eq:radsympot} holds for $d=1$ with
\begin{equation*}
F_\beta(r)=\frac{1}{\left|r-1\right|^{1-\beta}}+\frac{1}{\left|r+1\right|^{1-\beta}}.
\end{equation*}
For its the asymptotic behaviour, we still have \eqref{eq:to0}-\eqref{eq:toinf}, while
\begin{equation*}
\lim_{r \to 1}\frac{F_\beta(r)}{|r-1|^{\beta-1}}=1.
\end{equation*}

Below we will also make 
use of the Hardy-Littlewood-Sobolev inequality, which gives conditions under which $\mathcal{I}_\beta$ is a
bounded operator between specific $L^p$ spaces. Let $\beta \in (0,d)$, $p_\beta:=\frac{dp}{dp-\beta}$ be the
Sobolev exponent, and $p>1$. Then
\begin{equation}
\label{eq:HLS}
\Norm{\cI_\beta f}{p_\beta} \le C_{p,\beta} \Norm{f}{p}
\end{equation}
holds. A direct consequence is the inequality
\begin{equation}
\label{eq:HLS2}
\int_{\R^d}\int_{\R^d}\frac{|f(x)||g(x)|}{|x-y|^{\delta}}\, dx \, dy \le C_{p,\delta,d}\Norm{f}{p}\Norm{g}{q},	
\end{equation}
where $\delta \in (0,d)$, $f \in L^p(\R^d)$, $g \in L^q(\R^d)$ and $\frac{1}{p}+\frac{1}{q}=2-\frac{\delta}{d}$.

\subsection{Resolvent estimates}
\label{sec:res}
For every operator $L_{m,\alpha}$, $m \geq 0$, introduced in Section 2.1 above, the one-parameter operator semigroup
$\{e^{-tL_{m,\alpha}}: t\geq 0\}$ can be defined by functional calculus, each of whose elements is an integral operator
for $t>0$. The integral kernel $p(t,x,y) = e^{-tL_{m,\alpha}}(x,y)$, $t > 0$, which we call heat kernel as in the literature,
satisfies the property $p(t,x,y) = p(t,x-y,0)$ and, for fixed $t>0$, $p(t,\cdot,0)$ is rotationally symmetric. Throughout the paper we
use the simpler conventional notation $p_t(x)$. It is well-known that
for $m=0$ the semigroup is transient whenever $\alpha < d$, and for $m>0$ whenever $d \geq 3$, while in the complementary cases
the semigroups are recurrent. In our considerations below we will occasionally have to make a distinction between transient and
recurrent cases. Results relying on the $0$-resolvent operator will apply only in the transient case, since otherwise
$0$-resolvents are not well-defined. On the other hand, results relying on $\lambda$-resolvents with $\lambda>0$ can be
stated also in the recurrent case.

The following heat kernel estimates are known, for a proof see \cite{BBKRSV09}, \cite[Th. 1.2]{CKK11}, \cite[Th. 4.1]{CKS12}.

\begin{proposition}
\hspace{100cm}
\begin{enumerate}
\item
Let $p_t(x)$ be the integral kernel of the operator $e^{-tL_{0,\alpha}}$. Then
\begin{equation}
\label{heatker}
p_t(x) \,\, \asymp \,\, t|x|^{-d-\alpha} \wedge t^{-d/\alpha}
\end{equation}
holds for all $x \in \R^d$ and $t > 0$.
\item
Let $p^m_t(x)$ be the integral kernel of the operator $e^{-tL_{m,\alpha}}$, $m>0$. Then
\begin{align}
p^m_{t}(x)\asymp
\begin{cases}
\ t^{-d/\alpha}\wedge t|x|^{-(d+\alpha)/2}\kappa_m(|x|) & {\rm if}\quad 0< t \leq \frac{1}{m}\\
\\
\ m^{d/\alpha-d/2}t^{-d/2}e^{-c\left(m^{1/\alpha}|x|\wedge m^{2/\alpha-1}|x|^2/t\right)} &{\rm if}\quad t>\frac{1}{m}
\label{massive_heatker}
\end{cases}
\end{align}
hold with $c>0$ and
\begin{equation}
\label{eq:kappam}
\kappa_m(r)=m^{(d+\alpha)/(2\alpha)}K_{(d+\alpha)/2}(m^{1/\alpha}r).	
\end{equation}
\end{enumerate}
\end{proposition}
For every $\lambda > 0$, let $R^m_\lambda = (\lambda+L_{m,\alpha})^{-1}$, $m\geq 0$, be the resolvent of
the relativistic operator, with integral kernel $R^m_\lambda(x-y)$. We have then
$$
R^m_\lambda(x-y) = \int_0^\infty e^{-\lambda t} p^m_t(x-y)dt, \quad m \geq 0,
\label{R_lambda}
$$
and similarly to the heat kernel write just $R^m_\lambda(x)$. For the massless case we drop the superscripts
for all these objects.

Our results below will largely depend on detailed resolvent estimates for the fractional and relativistic Laplace
operators. We will need  uniform (in $\lambda$), and pointwise upper and lower estimates on the resolvent kernel
$R^m_\lambda$ separately for large $\lambda$ and for $\lambda= 0$, both for the fractional ($m=0$) and the
relativistic ($m>0)$  Laplacians.  These results stand also on their own, and can be applied in independent
directions such as potential theory.

We have the following uniform resolvent kernel estimates in the transient massless ($d > \alpha$) and transient
massive ($d > 2$) cases. For every $\lambda \ge 0$
\begin{align}
\label{forroll}
R^m_\lambda(x)  \leq
\begin{cases}
C_{d,\alpha}  |x|^{\alpha-d} & \mbox{if \, $m=0$}\\
&\\
C_{d,\alpha}\big(|x|^{\alpha-d}+m^{\frac{2}{\alpha}-1}{|x|^{2-d}}\big)  & \mbox{if \, $m>0$}
\end{cases}
\end{align}
holds with appropriate constant $C_{d,\alpha}>0$. In fact, for $\lambda=0$ the expression
\begin{equation*}
\label{eq:equivroll}
R_0(x)=C_{d,\alpha}|x|^{\alpha-d},
\end{equation*}
holds \cite[Prop. 4.293]{LHB}, while for the massive case we have
\begin{equation}
\label{eq:equivrollmassive}
R^m_0(x)\asymp C_{d,\alpha,m}\left(|x|^{\alpha-d}\mathbf{1}_{\{|x| \le 1\}}
+|x|^{2-d}\mathbf{1}_{\{|x| > 1\}}\right).
\end{equation}
Throughout the paper we will also make use of the non-uniform estimate
\begin{align}
\label{forrollext1}
R^m_\lambda(x)  \leq C_{d,\alpha,\lambda,m}H^0_{d,\alpha}(|x|),
\end{align}
which holds for every $\alpha \in (0,2)$, $d \ge 1$, $m \ge 0$ and $\lambda>0$ with
\begin{equation}
\label{eq:Hda0}
H_{d,\alpha}^0(r):=
\begin{cases}
r^{-(d-\alpha)}\mathbf{1}_{\{r \le 1\}}+2^{\frac{d+\alpha}{2}-1}\Gamma\left(\frac{d+\alpha}{2}\right)
r^{-(d+\alpha)} \mathbf{1}_{\{r>1\}} & \mbox{if \, $\alpha<d$} \\
& \\
\log\left(1+\frac{1}{r}\right)\mathbf{1}_{\{r \le 1\}}+2^{\frac{d+\alpha}{2}-1}
\Gamma\left(\frac{d+\alpha}{2}\right)r^{-2}\mathbf{1}_{\{r>1\}}	& \mbox{if \, $\alpha=d=1$} \\
&\\		
\mathbf{1}_{\{r \le 1\}}+2^{\frac{d+\alpha}{2}-1}\Gamma\left(\frac{d+\alpha}{2}\right)r^{-(1+\alpha)}
\mathbf{1}_{\{r>1\}} & \mbox{if \, $\alpha>d=1$}.
\end{cases}
\end{equation}
Again, for $m=0$ we actually have
\begin{align}
\label{eq:equivrollext1}
R_\lambda(x)  \asymp C_{d,\alpha,\lambda}H^0_{d,\alpha}(|x|).
\end{align}
In the massive case we do not have this equivalence. However, for $m>0$ and $\lambda>m$, we prove
that
\begin{align}
\label{eq:equivrollext2}
R^m_\lambda(x)  \asymp C_{d,\alpha,\lambda,m}H^m_{d,\alpha}(|x|),
\end{align}
where
\begin{equation}
\label{eq:Hmda}
H^m_{d,\alpha}(r):=
\begin{cases}
r^{-(d-\alpha)}\mathbf{1}_{\{r \le 1\}}+r^{-\frac{d+\alpha}{2}}\kappa_m(r)\mathbf{1}_{\{r>0\}}
& \mbox{if \, $\alpha<d$} \\
& \\
\log\left(1+\frac{1}{r}\right)\mathbf{1}_{\{r \le 1\}}+r^{-\frac{d+\alpha}{2}}\kappa_m(r)
\mathbf{1}_{\{r>0\}} & \mbox{if \, $\alpha=d=1$} \\
& \\
\mathbf{1}_{\{r \le 1\}}+r^{-\frac{d+\alpha}{2}}\kappa_m(r)\mathbf{1}_{\{r>0\}}
& \mbox{if \, $\alpha>d=1$}.
\end{cases}
\end{equation}
We note that
\begin{align}
\label{eq:lowerrollext2}
R^m_\lambda(x)  \ge C_{d,\alpha,\lambda,m}H^m_{d,\alpha}(|x|)
\end{align}
holds in fact for all $m,\lambda>0$. It is also worth recalling that in \eqref{forrollext1} the constant satisfies 
\begin{equation}\label{eq:asymptoticsC}
	C_{d,\alpha,\lambda,m}\sim C_{d,\alpha}\lambda^{-2}\left(1+e^{-\frac{\lambda}{m}}\right) \quad \mbox{ as }\quad \lambda \downarrow 0.
\end{equation} The proof of these resolvent estimates is given in the Appendix.

\section{Fractional Rollnik class}
\subsection{Definitions and basic properties}
\subsubsection{Fractional Rollnik norm}
Making use of \eqref{forroll} we introduce the following class of potentials for the operators $L_{m,\alpha}$, $m \geq 0$.
\begin{definition}
\label{fRollnik}
{\rm
Let  $d \in \mathbb N$,  $0 < \alpha < 2$, and assume that
\begin{equation}
\label{standassmp}
0 < \alpha < d < 2\alpha.
\end{equation}
\begin{enumerate}
\item[(1)]
Let $m=0$. We call
\begin{equation*}
\mathcal{R}_{d,\alpha}=\left\{V:\mathbb{R}^d\rightarrow\mathbb{R} : \, \|V\|^2_{\mathcal{R}_{d,\alpha}}=
\int_{\mathbb{R}^{d}} \int_{\mathbb{R}^{d}}\frac{|V(x)||V(y)|}{|x-y|^{2(d-\alpha)}}dxdy<\infty \right\}
\end{equation*}
\emph{fractional Rollnik-class of order $\alpha$ for the massless Laplacian $L_{0,\alpha}$}.
\medskip
\item[(2)]
Let $d=3$ and $m>0$. We define the \emph{fractional Rollnik-class of order $\alpha$ for the massive
Laplacian $L_{m,\alpha}$} by $\mathcal{R}^m_{\alpha} = \mathcal{R}_{d,\alpha} \cap \mathcal{R}$, where
$\mathcal{R}$ is the classical Rollnik class.
\end{enumerate}
}
\end{definition}
It is seen that the study of the massive case reduces to the study of the massless case. First therefore
we focus on the class $\mathcal{R}_{d,\alpha}$ and as we progress we will consider extensions, which
will also accommodate cases of recurrent semigroups.

\noindent

We begin by showing that the expression $\|\cdot\|_{\mathcal{R}_{d,\alpha}}$ used in the definition of
$\mathcal{R}_{d,\alpha}$ is indeed a norm and that $\mathcal{R}_{d,\alpha}$ is a Banach space with
respect to this norm. Before that, we note the following two properties which will be used below.
\begin{lemma}
\label{lem:L1loc}
We have that $L^1_{\rm loc}(\R^d) \subset \mathcal{R}_{d,\alpha}$ and
\begin{equation*}
\Norm{V}{\mathcal{R}_{d,\alpha}}^2 \ge \frac{1}{(2r(K))^{2d-2\alpha}}\Norm{V}{L^1(K)}^2,
\end{equation*}
where $K \subset \R^d$ is any compact set and $r(K)=\inf\{r \ge 0: \ K \subset B_r\}$.
\end{lemma}
\begin{proof}
Without loss of generality, it will be sufficient to show this for every ball $B_r \subset \R^d$.
Then we have
\begin{equation*}
\Norm{V}{\mathcal{R}_{d,\alpha}}^2 \ge \int_{B_r}\int_{B_r}|V(x)||V(y)||x-y|^{2\alpha-2d}\, dx\, dy
\ge (2r)^{2\alpha-2d}\Norm{V}{L^1(B_r)}^2.
\end{equation*}
\end{proof}
A straightforward but useful consequence of the Hardy-Littlewood-Sobolev inequality is the following inclusion
property.
\begin{proposition}
\label{prop:prop2anticipated}
$L^{d/\alpha}(\R^d)$ is continuously embedded in $\cR_{d,\alpha}$, i.e.,
\begin{equation*}
\Norm{V}{\cR_{d,\alpha}} \le C_{d,\alpha}\Norm{V}{d/\alpha}.
\end{equation*}	
\end{proposition}
\begin{proof}
It is immediate by \eqref{eq:HLS2} with $\sigma=2(d-\alpha)$ and $p=q=\frac{d}{\alpha}$.
\end{proof}
Apart from $\Norm{\cdot}{\cR_{d,\alpha}}$, we can also define the quantity
\begin{equation*}
\NNorm{V}{\cR_{d,\alpha}}^2=\int_{\R^d}\int_{\R^d}V(x)\overline{V(y)}|x-y|^{2\alpha-2d}\, dx \, dy
\end{equation*}
on $\cR_{d,\alpha}$, where we allow $V$ to be possibly complex-valued. Next we prove that both $\Norm{\cdot}{\cR_{d,\alpha}}$
and $\NNorm{\cdot}{\cR_{d,\alpha}}$ are norms on $\cR_{d,\alpha}$; we take inspiration from a strategy proposed in
\cite[Sect. 4]{PipTaq00} for $d=1$ only, with substantial modifications. Recall the factor $\gamma_d$ from \eqref{gammad}.
\begin{theorem}
\label{norms}
The functionals $\Norm{\cdot}{\cR_{d,\alpha}}$ and $\NNorm{\cdot}{\cR_{d,\alpha}}$
are norms on $\cR_{d,\alpha}$. Furthermore, for all $V \in \cR_{d,\alpha}$ we have
\begin{equation}
\label{eq:noabs}
\NNorm{V}{\cR_{d,\alpha}}^2=\gamma_d(2\alpha-d)\Norm{\cI_{\alpha-\frac{d}{2}}V}{2}^2 \quad \mbox{ and } \quad \Norm{V}{\cR_{d,\alpha}}^2=\gamma_d(2\alpha-d)\Norm{\cI_{\alpha-\frac{d}{2}}|V|}{2}^2.
\end{equation}
\end{theorem}
\begin{proof}
The proof will proceed through several steps.

\vspace{0.1cm}
\noindent
\emph{Step 1:}
First we show \eqref{eq:noabs}, starting with assuming that $V \in \cF\cS_0(\R^d)$. Clearly, by the definition of
Lizorkin space, we know that $V \in \cS(\R^d)$ and then $V \in L^2(\R^d) \cap L^{d/\alpha}(\R^d)$, which in turn
implies $V \in \cR_{d,\alpha}$. Furthermore, also $\cI_{2\alpha-d}V \in \cF\cS_0(\R^d)\subset L^2(\R^d)$ by
Theorem \ref{thm:LizRi}. Hence, by an application of Fubini's theorem and Plancherel's identity we obtain
\begin{align*}
\nonumber
\NNorm{V}{\cR_{d,\alpha}}^2&=\gamma_d(2\alpha-d)\left( V, \cI_{2\alpha-d}V\right)_{L^2(\R^d)}\\
&=\gamma_d(2\alpha-d)
\left( \cF[V], |\cdot|^{d-2\alpha}\cF[V]\right)_{L^2(\R^d)} \nonumber \\
\label{Rieszy}
&=\gamma_d(2\alpha-d)\left( |\cdot|^{\frac{d}{2}-\alpha}\cF[V], |\cdot|^{\frac{d}{2}-\alpha}\cF[V]\right)_{L^2(\R^d)}
=\gamma_d(2\alpha-d)\Norm{\cI_{\alpha-\frac{d}{2}}V}{2}^2.
\end{align*}

\vspace{0.1cm}
\noindent
\emph{Step 2:}
Next assume that $V \in L^{d/\alpha}(\R^d)$. Since $\cF\cS_0(\R^d)$ is dense in $L^{d/\alpha}(\R^d)$ by Theorem~\ref{thm:dense},
we may choose a sequence $\seq V \subset \cF\cS_0(\R^d)$ such that $V_n \to V \in L^{d/\alpha}(\R^d)$. Since $V_n-V \in
L^{d/\alpha}(\R^d)$, by Proposition~\ref{prop:prop2anticipated}
we have
\begin{align*}
&\left|\int_{\R^d}\int_{\R^d}V_n(x)V_n(y)|x-y|^{2\alpha-2d}dx\,dy-\int_{\R^d}\int_{\R^d}V(x)V(y)|x-y|^{2\alpha-2d}dx\,dy\right|\\
&\qquad \le \int_{\R^d}\int_{\R^d}|V_n(x)-V(x)||V_n(y)||x-y|^{2\alpha-2d}dx\,dy\\
&\qquad \qquad +\int_{\R^d}\int_{\R^d}|V(x)||V_n(y)-V(y)||x-y|^{2\alpha-2d}dx\,dy\\
&\qquad \le C_{d,\alpha}\Norm{V_n-V}{d/\alpha} \left(\Norm{V_n}{{d/\alpha}} +\Norm{V}{{d/\alpha}}\right) \to 0 \quad
\mbox{as $n\to\infty$}.
\end{align*}
On the other hand, by the triangle inequality and \eqref{eq:HLS},
\begin{equation*}
\left|\Norm{\cI_{\alpha-\frac{d}{2}}V_n}{2}-\Norm{\cI_{\alpha-\frac{d}{2}}V}{2}\right| \le \Norm{\cI_{\alpha-\frac{d}{2}}(V_n-V)}{2}
\le C_{d,\alpha}\Norm{V_n-V}{d/\alpha} \to 0.
\end{equation*}
This proves the first equality in \eqref{eq:noabs} for $V \in L^{d/\alpha}(\R^d)$. The second equality follows by applying the
first on $|V| \in L^{d/\alpha}(\R^d)$.

\vspace{0.1cm}
\noindent
\emph{Step 3:}
To complete this part of the proof, consider now any $V \in \cR_{d,\alpha}$. Let $M \in \N$ and $\eta_M$ be a
$C^\infty_{\rm c}(\R^d)$ function such
that $\eta_M(x)=1$ for every $x \in B_M$, $\eta_M(x) \in [0,1]$ and ${\rm supp}(\eta_M)\subset B_{M+1}$. Then $V_M
:=\min\{|V|,M\}\eta_M \in L^{{d/\alpha}}(\R^d)$ and \eqref{eq:noabs} holds for this function. Next, since
$V_M \le  V_{M+1}$ for all $M \in \N$, by the monotone convergence theorem we have
\begin{equation*}
\lim_{M \to \infty}\Norm{V_M}{\cR_{d,\alpha}}^2=\Norm{V}{\cR_{d,\alpha}}^2 \quad \mbox{ and } \quad
\lim_{M \to \infty}\Norm{\cI_{\alpha-\frac{d}{2}}V_M}{2}^2=\Norm{\cI_{\alpha-\frac{d}{2}}|V|}{2}^2.
\end{equation*}
This proves the second equality in \eqref{eq:noabs}. To prove the first, notice that we have in fact shown
that if $V \in \cR_{d,\alpha}$, then $\Norm{\cI_{\alpha-\frac{d}{2}}|V|}{2}<\infty$. Hence by setting $V_M
=\max\{\min\{V,M\},-M\}\eta_M$, we can use the dominated convergence theorem to obtain
\begin{equation*}
\lim_{M \to \infty}\NNorm{V_M}{\cR_{d,\alpha}}^2=\NNorm{V}{\cR_{d,\alpha}}^2 \quad \mbox{ and } \quad
\lim_{M \to \infty}\Norm{\cI_{\alpha-\frac{d}{2}}V_M}{2}^2=\Norm{\cI_{\alpha-\frac{d}{2}}V}{2}^2.
\end{equation*}

\vspace{0.1cm}
\noindent
\emph{Step 4:}
Finally we show the norm property.
Identity \eqref{eq:noabs} guarantees that for all $V \in \cR_{d,\alpha}$, both $\Norm{V}{\cR_{d,\alpha}}$
and $\NNorm{V}{\cR_{d,\alpha}}$ are non-negative. Furthermore, note that directly by the definition of the
Riesz potential, the facts $V_1,V_2 \ge 0$ and $V_1 \le V_2$, imply $\Norm{\cI_{\alpha-\frac{d}{2}}V_1}{2}
\le \Norm{\cI_{\alpha-\frac{d}{2}}V_2}{2}$. Hence we have
\begin{eqnarray*}
\Norm{V_1+V_2}{\cR_{d,\alpha}}
&=&
\sqrt{\gamma(2\alpha-d)}\Norm{\cI_{\alpha-\frac{d}{2}}|V_1+V_2|}{2} \le
\sqrt{\gamma(2\alpha-d)}\Norm{\cI_{\alpha-\frac{d}{2}}|V_1|+\cI_{\alpha-\frac{d}{2}}|V_2|}{2} \\
&\le& \sqrt{\gamma(2\alpha-d)}\Norm{\cI_{\alpha-\frac{d}{2}}|V_1|}{2}+\sqrt{\gamma(2\alpha-d)}
\Norm{\cI_{\alpha-\frac{d}{2}}|V_2|}{2} \\
&=&
\Norm{V_1}{\cR_{d,\alpha}}+\Norm{V_2}{\cR_{d,\alpha}}.
\end{eqnarray*}
The fact that $\Norm{V}{\cR_{d,\alpha}}=0$ implies $V=0$ a.e. follows by injectivity of the Riesz potential
$\cI_{2\alpha-d}$. The same argument can be repeated for $\NNorm{\cdot}{\cR_{d,\alpha}}$, which proves that
both
are norms.
\end{proof}

\begin{remark}
{\rm
Notice that $\NNorm{\cdot}{\cR_{d,\alpha}}$ is induced by the scalar product
\begin{equation*}
( V_1,V_2)_{\cR_{d,\alpha}}=\int_{\R^d}\int_{\R^d}V_1(x)\overline{V_2(y)}|x-y|^{2\alpha-2d}\, dx\, dy.
\end{equation*}
}
\end{remark}
We can prove the following completeness result with respect to the norm $\Norm{\cdot}{\cR_{d,\alpha}}$.
\begin{theorem}
\label{banach}
The fractional Rollnik class $\mathcal{R}_{d,\alpha}$ equipped with the norm $\|\cdot\|_{\mathcal{R}_{d,\alpha}}$ is
a Banach space.
\end{theorem}
\begin{proof}
Let $\seq V$ be a Cauchy sequence in $\cR_{d,\alpha}$. By Lemma \ref{lem:L1loc} it is a Cauchy sequence also in
$L^1(B_m)$ for all $m \in \N$. By a simple diagonal argument we know that there exists a subsequence
$(V_{n_k})_{k \in \N}$ and a function $V$ such that $V_{n_k}(x) \to V(x)$ for a.e. $x \in \R^d$. By Fatou's
lemma we have that
\begin{equation*}
\int_{\R^{d}}\int_{\R^d}|V(x)-V_{n_k}(x)||V(y)-V_{n_k}(y)||x-y|^{2\alpha-2d}\, dx \, dy \le \liminf_{j \to \infty}
\Norm{V_{n_j}-V_{n_k}}{\cR_{d,\alpha}}^2.
\end{equation*}
Fix $\varepsilon>0$. Since $\seq V$ is a Cauchy sequence in $\cR_{d,\alpha}$, there exists a label $k_0 \in \N$ such
that $\Norm{V_{n_j}-V_{n_k}}{\cR_{d,\alpha}}<\varepsilon$ holds for all $j,k \ge k_0$. Taking the limit inferior, this
means that for all $k \ge k_0$
\begin{equation*}
\int_{\R^{d}}\int_{\R^d}|V(x)-V_{n_k}(x)||V(y)-V_{n_k}(y)||x-y|^{2\alpha-2d}\, dx \, dy \le \liminf_{j \to \infty}
\Norm{V_{n_j}-V_{n_k}}{\cR_{d,\alpha}}^2<\varepsilon^2.
\end{equation*}
This proves that $V-V_{n_k} \in \cR_{d,\alpha}$, for all $k \ge k_0$, and thus also $V \in \cR_{d,\alpha}$. Furthermore,
for all $k \ge k_0$ we have
$\Norm{V-V_{n_k}}{\cR_{d,\alpha}}<\varepsilon$.
Since $\varepsilon>0$ is arbitrary, we get $V_{n_k} \to V$ in $\cR_{d,\alpha}$. This is sufficient to conclude that $V_n
\to V$ in $\cR_{d,\alpha}$ as $n\to\infty$.
\end{proof}
It is natural to ask whether $\cR_{d,\alpha}$ is a Hilbert space when equipped with the norm $\vertiii{\cdot}
{\cR_{d,\alpha}}$. Adapting an argument in \cite{PipTaq00} originally only for $d=1$, we see that this is not the case.
\begin{theorem}
The fractional Rollnik class $\mathcal{R}_{d,\alpha}$ is not complete with respect to the norm
$\vertiii{\cdot}_{\cR_{d,\alpha}}$.
\end{theorem}
\begin{proof}
Assume, to the contrary, that $\cR_{d,\alpha}$ is complete with respect to $\vertiii{\cdot}_{\cR_{d,\alpha}}$.
Consider any function $\phi \in L^2(\R^d)$ such that $\phi \not \in \Dom((-\Delta)^{(2\alpha-d)/4})$. Since
$\cF\cS_0(\R^d)$ is dense in $L^2(\R^d)$, there exists a sequence $\seq\phi\subset \cF\cS_0(\R^d)$ such that
$\phi_n \to \phi$ in $L^2(\R^d)$. Define the sequence of functions
\begin{equation*}
f_n=\cF^{-1}[|\cdot|^{\frac{d}{2}-\alpha}\cF[\phi_n]] \,\in\, \cF\cS_0(\R^d) \,\subset\, \cR_{d,\alpha}.
\end{equation*}
Since $\phi_n \to \phi$ in $L^2(\R^d)$, also $\cI_{\alpha-\frac{d}{2}}f_n \to \phi$ in $L^2(\R^d)$, hence
$\seq{\cI_{\alpha-\frac{d}{2}}f}$ is a Cauchy sequence in $L^2(\R^d)$. By the first equality in \eqref{eq:noabs},
this implies that $\seq f$ is a Cauchy sequence in $(\cR_{d,\alpha},\vertiii{\cdot}_{\cR_{d,\alpha}})$. Since
by assumption this space is complete, there exists a function $f$ such that $f_n \to f$ in
$(\cR_{d,\alpha},\vertiii{\cdot}_{\cR_{d,\alpha}})$. Again by \eqref{eq:noabs}, we know that
$\cI_{\alpha-\frac{d}{2}}f \in L^2(\R^d)$ and $\cI_{\alpha-\frac{d}{2}}f_n \to \cI_{\alpha-\frac{d}{2}}f$. However,
this is impossible since then we would have $\phi=\cI_{\alpha-\frac{d}{2}}f \in \Dom((-\Delta)^{(2\alpha-d)/4})$.
\end{proof}

\begin{remark}
\label{rem:upinc}
{\rm
We note that if $V_1,V_2$ are two potentials such that $|V_1(x)| \le |V_2(x)|$ a.e. and $V_2 \in \cR_{d,\alpha}$, then
clearly also $V_1 \in \cR_{d,\alpha}$. Furthermore, by Lemma \ref{lem:L1loc}, it is clear that if $V \in \cR_{d,\alpha}$,
then for all $R>0$
\begin{equation}\label{eq:L1bound}
\Norm{V}{L^1(B_R)} \le 2^{d-\alpha}\Norm{V}{\cR_{d,\alpha}}R^{d-\alpha}.
\end{equation}
This means that fractional Rollnik-class potentials tend to concentrate their mass at infinity.
}
\end{remark}

\subsubsection{Restrictions on the fractional exponent}
Next we turn to discussing the role of condition \eqref{standassmp}. A straightforward calculation shows that
$|x|^{-2(d-\alpha)} \in L^1_{\rm loc}(\R^d)$ only when $\alpha < 2d$. This implies that if $d \ge 2\alpha$ there
exist no continuous potentials $V \in \cR_{d,\alpha}$ apart from $V \equiv 0$. In fact, we can prove the following
more general property.
\begin{theorem}
\label{thm:nonL1kernel}
Let $k:[0,\infty) \to [0,\infty)$ be a function such that for all $\delta>0$
\begin{equation}
\label{eq:nonintegrability}
\int_0^{\delta}r^{d-1}k(r)=\infty.
\end{equation}
Suppose that $V:\R^d \to \R$ is a measurable function such that
\begin{equation}
\label{eq:finitecond}
\int_{\R^d}\int_{\R^d}|V(x)||V(y)|k(|x-y|)dxdy<\infty.
\end{equation}
Then for any choice of $\varepsilon, r >0$ and $x \in \R^d$ there is no ball $B_r(x)$ such that $|V(y)|>
\varepsilon$ for a.e. $y \in B_r(x)$. In particular, if $V$ is continuous, then $V \equiv 0$.
\end{theorem}
\begin{proof}
Let $V:\R^d \to \R$ be satisfying \eqref{eq:finitecond} and assume, to the contrary, that there exist
$\varepsilon^{\ast}, r^\ast>0$ and $x^\ast \in \R^{d}$ such that $|V(y)|>\varepsilon^\ast$ for all $y \in
B_{r^\ast}(x^\ast)$. Then
\begin{align*}
\int_{\R^d}\int_{\R^d}|V(x)||V(y)|k(|x-y|)\, dx\, dy
&\ge 	\int_{B_{r^\ast}(x^\ast)}\int_{B_{r^\ast}(x^\ast)}|V(x)||V(y)|k(|x-y|)\, dx \, dy \\
&\ge \varepsilon^2 \int_{B_{r^\ast}(x^\ast)}\int_{B_{r^\ast}(x^\ast)}k(|x-y|)\, dx \, dy,
\end{align*}
which implies $\int_{B_{r^\ast}(x^\ast)}\int_{B_{r^\ast}(x^\ast)}k(|x-y|)\, dx \, dy<\infty$.
However, by \eqref{eq:nonintegrability} we also have
\begin{equation*}
\int_{B_{r^\ast}(x^\ast)}\int_{B_{r^\ast}(x^\ast)}k(|x-y|)\, dx \, dy \ge \int_{B_{r^\ast}(x^\ast)}
\int_{B_{r^\ast-|x-x^\ast|}(x)}k(|x-y|)\, dx \, dy=\infty,
\end{equation*}
which is impossible.
\end{proof}	

In fact, a fractional Rollnik-class cannot be obtained even in the limit $\alpha \uparrow \frac{d}{2}$.
\begin{proposition}
\label{prop:notappr}
Let $d=1,2$ and suppose $V:\R^d \to \R$ is a measurable function such that $\liminf\limits_{\alpha
\downarrow \frac{d}{2}}\Norm{V}{\cR_{d,\alpha}}<\infty$. Then for any choice of $\varepsilon,r>0$
and $x \in \R^d$ there is no ball $B_r(x)$ such that $|V(y)|>\varepsilon$ for a.e. $y \in B_r(x)$.
In particular, if $V$ is continuous, then $V \equiv 0$.
\end{proposition}
\begin{proof}
Let $V:\R^d \to \R$ be such that $\liminf\limits_{\alpha \downarrow \frac{d}{2}}\Norm{V}{\cR_{d,\alpha}}
<\infty$ and assume, to the contrary, that there exists $\varepsilon^*,r^*$ and $x^*\in \R^d$ such that
$|V(y)|>\varepsilon^*$ for almost every $y \in B_{r^*}(x^*)$. Let also $\alpha_k$ be such that $\alpha_k
\downarrow \frac{d}{2}$ and $\lim_{k \to \infty}\Norm{V}{\cR_{d,\alpha_k}}=\liminf\limits_{\alpha
\downarrow \frac{d}{2}}\Norm{V}{\cR_{d,\alpha}}$. Without loss of generality, we assume that $\Norm{V}
{\cR_{d,\alpha_k}}<\infty$ for all $k \in \N$. Observe that for all $k \in \N$
\begin{equation}
\label{eq:boundedliminf}
\infty>\int_{\R^d}\int_{\R^d}|V(x)||V(y)||x-y|^{2\alpha_k-2d}dxdy
\ge \varepsilon^2 \int_{B_{r^\ast}(x^\ast)}\int_{B_{r^\ast}(x^\ast)}|x-y|^{2\alpha_k-2d}dxdy.
\end{equation}
On the other hand,
\begin{equation*}
\int_{B_{r^\ast}(x^\ast)}\int_{B_{r^\ast}(x^\ast)}|x-y|^{2\alpha_k-2d}dxdy
\ge \int_{B_{r^\ast}(x^\ast)}\int_{B_{r^\ast-|x-x^\ast|}(x)}|x-y|^{2\alpha_k-2d}dydx.
\end{equation*}
We can evaluate the inner integral for $x \in B_{r^\ast}(x^\ast)$ as
\begin{equation*}
\int_{B_{r^\ast-|x-x^\ast|}(x)}|x-y|^{2\alpha_k-2d}dy=\sigma_d \int_0^{r^\ast-|x-x^\ast|}
\rho^{2\alpha_k-d-1}d\rho=\frac{\sigma_d}{2\alpha_k-d}(r^\ast-|x-x^\ast|)^{2\alpha_k-d}
\end{equation*}
and thus
\begin{align*}
\int_{B_{r^\ast}(x^\ast)}\int_{B_{r^\ast-|x-x^\ast|}(x)}|x-y|^{2\alpha_k-2d}dydx
&=\frac{\sigma_d^2}{2\alpha_k-d}\int_0^{r^\ast}(r^\ast-\rho)^{2\alpha_k-d}\rho^{d-1}d\rho\\
&=\frac{\sigma_d^2}{2\alpha_k-d}(r^\ast)^{2\alpha_k}B(2\alpha_k-d+1;d),
\end{align*}
where $B(x,y)$ is the Beta function. Taking the limit as $k \to \infty$ we get
\begin{equation*}
\int_{B_{r^\ast}(x^\ast)}\int_{B_{r^\ast-|x-x^\ast|}(x)}|x-y|^{2\alpha_k-2d}dydx=\infty,
\end{equation*}
which is in contradiction with \eqref{eq:boundedliminf}.
\end{proof}

Clearly, by \eqref{standassmp} we have the possible ranges $D_\alpha$ of $\alpha$ as follows:
\begin{equation*}
D_\alpha =
\begin{cases}
(\frac{1}{2},1) &\mbox{if \, $d=1$}\\
(1,2) &\mbox{if \, $d=2$} \\
(\frac{3}{2},2) &\mbox{if \, $d=3$} \\
\emptyset &\mbox{if \, $d\geq 4$}.
\end{cases}
\end{equation*}
This just excludes the case $\alpha = 1$, which will be dealt with below. In Section \ref{extendrollnik}
we introduce an extended  fractional Rollnik-class which is subject only to the constraint $d<2\alpha$, in
particular, the case $\alpha=1$ is covered by it in the $d=1$ case. In $d=2$ we will obtain it in Section
\ref{sectalpha1} below as a limiting case. As Proposition \ref{prop:notappr} suggests, this limit is not
straightforward  and we arrive at defining
\begin{equation*}
\cR_{d,1}^p:= L^p(\R^d) \, \cap \, \liminf_{\alpha \downarrow 1}\cR_{d,\alpha}, \quad p > 1.
\end{equation*}
For later use below, we define more generally for $\alpha=\frac{d}{2}$, $d=1, 2, 3$, the class
\begin{equation}
\label{crunchyrollnik}
\cR_{d,\frac{d}{2}}^p:= L^p(\R^d) \, \cap \, \liminf_{\alpha \downarrow \frac{d}{2}}\cR_{d,\alpha},
\quad p > 1,
\end{equation}
where the limit inferior means $\liminf_{\alpha \downarrow \frac{d}{2}}\cR_{d,\alpha}=
\cup_{\beta \in \left(\frac{d}{2},d\right)} \cap_{\alpha \in \left(\frac{d}{2},\beta\right)}\cR_{d,\alpha}$. The
intersection with the $L^p$ spaces in the definition will be justified by Theorem \ref{alpha1sa} below.

\subsection{Comparison properties}
We have already seen in Proposition \ref{prop:prop2anticipated} that $L^{d/\alpha}(\mathbb{R}^d) \subset
\mathcal{R}_{d,\alpha}$. A corollary to this result is the following.
\begin{proposition}
\label{cor7}
If $p\leq \frac{d}{\alpha}\leq q$, then $L^p(\mathbb{R}^d)\cap L^q(\mathbb{R}^d)
\subset\mathcal{R}_{d,\alpha}$.
\end{proposition}
\begin{proof}
By interpolation it is immediate that $L^p(\R^d)\cap L^q(\R^d)\subset L^{d/\alpha}(\R^d) \subset
\mathcal{R}_{d,\alpha}$.
\end{proof}

\begin{proposition}
\label{pro8}
If $\alpha > \frac{d}{2}$, then $L^1(\mathbb{R}^d)\cap L^2(\mathbb{R}^d)\subset\mathcal{R}_{d,\alpha}$ and the estimate
\begin{equation}
\|V\|_{\mathcal{R}_{d,\alpha}}\leq\left(\frac{\sigma_d}{2(d-\alpha)}\right)^{(d-\alpha)/d}\sqrt{\frac{d}{2\alpha-d}} \,
\|V\|_{1}^{(2\alpha-d)/d}\|V\|_{2}^{2(d-\alpha)/d}
\label{pro8_1}
\end{equation}
holds, where $\sigma_d$ is the surface area of the $d$-dimensional unit sphere.
\end{proposition}

\begin{proof}
Let $V\in L^1(\mathbb{R}^d)\cap L^2(\mathbb{R}^d)$ and $\delta>0$. We compute
\begin{equation*}
\iint_{|x-y|\geq \delta}\frac{|V(x)||V(y)|}{|x-y|^{2(d-\alpha)}}dxdy\leq\frac{1}{\delta^{2(d-\alpha)}}
\iint_{\mathbb{R}^{2d}}|V(x)||V(y)|dxdy = \frac{\|V\|_{1}^2}{\delta^{2(d-\alpha)}}.
\end{equation*}
By the Schwarz inequality, we also have
\begin{eqnarray}
\iint_{|x-y|<\delta}\frac{|V(x)||V(y)|}{|x-y|^{2(d-\alpha)}}dxdy
&\leq&
\iint_{|x-y|<\delta}\frac{|V(x)|^2}{|x-y|^{2(d-\alpha)}}dxdy \nonumber\\
&\leq&
\|V\|_{2}^2\int_{|x|<\delta}\frac{dx}{|x|^{2(d-\alpha)}}
=\frac{\sigma_d \delta^{2\alpha-d}}{2\alpha-d}\|V\|_{2}^2,\label{pro8_2}
\end{eqnarray}
and thus
\begin{equation*}
\|V\|_{\mathcal{R}_{d,\alpha}}^2\leq\frac{\|V\|_{1}^2}{\delta^{2(d-\alpha)}}
+\frac{\sigma_d \delta^{2\alpha-d}}{2\alpha-d}\|V\|_{2}^2.
\end{equation*}
Optimizing over $\delta$ by the choice
\begin{equation*}
\delta=\left(\frac{(2(d-\alpha)\|V\|_{1}^2}{\sigma_d\|V\|_{2}^2}\right)^{1/d}
\end{equation*}
leads to \eqref{pro8_1}.
\end{proof}

\begin{proposition}
\label{cor9}
The following inclusions hold:
\begin{enumerate}
\item
Let $\alpha \leq \alpha_1\leq d$. Then $\mathcal{R}_{d,\alpha}\cap L^1(\mathbb{R}^d)
\subset\mathcal{R}_{d,\alpha_1}$.

\medskip
\item
Let $\alpha_2\leq\alpha$ and $\frac{d}{\alpha_2}< p \wedge q$ for $\frac{1}{p}+\frac{1}{q}=1$. Then
$\mathcal{R}_{d,\alpha} \cap L^p(\mathbb{R}^d)\cap L^q(\mathbb{R}^d)\subset\mathcal{R}_{d,\alpha_2}$. In
particular, $\mathcal{R}_{d,\alpha}\cap L^2(\mathbb{R}^d)\subset\mathcal{R}_{d,\alpha_2}$.
\item
Let $\alpha_1 \le \alpha \le \alpha_2$. Then $\cR_{d,\alpha_1}\cap \cR_{d,\alpha_2} \subset \cR_{d,\alpha}$.
\end{enumerate}
\end{proposition}

\begin{proof}
(1) Let $V\in \mathcal{R}_{d,\alpha}\cap L^1(\mathbb{R}^d)$. As in Proposition \ref{pro8}, we compute
\begin{equation*}
\iint_{|x-y|\geq 1}\frac{|V(x)||V(y)|}{|x-y|^{2d-2\alpha_1}}dxdy\leq\|V\|_{1}^2.
\end{equation*}
We clearly have
\begin{equation*}
\iint_{|x-y|<1}\frac{|V(x)||V(y)|}{|x-y|^{2d-2\alpha_1}}dxdy\leq\iint_{\mathbb{R}^{2d}}
\frac{|V(x)||V(y)|}{|x-y|^{2d-2\alpha}}dxdy=\|V\|_{\mathcal{R}_{d,\alpha}}^2
\end{equation*}
using that $2d-2\alpha_1\leq 2d-2\alpha$.

\medskip
\noindent
(2)
Let $V\in\mathcal{R}_{d,\alpha}\cap L^p(\mathbb{R}^d)\cap L^q(\mathbb{R}^d)$. We  have
\begin{equation*}
\iint_{|x-y|\geq 1}\frac{|V(x)||V(y)|}{|x-y|^{2d-2\alpha_2}}dxdy\leq
\iint_{\mathbb{R}^{2d}}\frac{|V(x)||V(y)|}{|x-y|^{2d-2\alpha}}dxdy=\|V\|_{\mathcal{R}_{d,\alpha}}^2
\end{equation*}
since $2d-2\alpha_2\geq 2d-2\alpha$. On the other hand, as in \eqref{pro8_2}, we compute
\begin{equation*}
\iint_{|x-y|<1}\frac{|V(x)|^p}{|x-y|^{pd-p\alpha_2}}dxdy\leq\frac{\sigma_d}{p\alpha_2-pd+d}
\|V\|_{p}^p
\end{equation*}
noting that $p\alpha_2-pd+d>0$ is equivalent to $d/\alpha_2<q$. Therefore, by the H\"older inequality
we obtain
\begin{equation*}
\iint_{|x-y|<1}\frac{|V(x)||V(y)|}{|x-y|^{2d-2\alpha_2}}dxdy\leq
\left(\frac{\sigma_d}{p\alpha_2-pd+d}\right)^{1/p}
\left(\frac{\sigma_d}{q\alpha_2-qd+d}\right)^{1/q}\|V\|_{p}\|V\|_{q}.
\end{equation*}

\medskip
\noindent
(3) We have
\begin{align*}
\Norm{V}{\cR_{d,\alpha}}^2&=\iint_{|x-y|<1}\frac{|V(x)||V(y)|}{|x-y|^{2d-2\alpha}}dx\, dy
+\iint_{|x-y| \ge 1}\frac{|V(x)||V(y)|}{|x-y|^{2d-2\alpha}}dx\, dy\\
&\le \iint_{|x-y|<1}\frac{|V(x)||V(y)|}{|x-y|^{2d-2\alpha_1}}dx\, dy
+\iint_{|x-y| \ge 1}\frac{|V(x)||V(y)|}{|x-y|^{2d-2\alpha_2}}dx\, dy \le \Norm{V}{\cR_{d,\alpha_1}}^2
+\Norm{V}{\cR_{d,\alpha_2}}^2.
\end{align*}
\end{proof}

By extension of the definition of $\mathcal{R}_{d,\alpha}$ to $\alpha = 2$ we consider the space
$\mathcal{R}_{3,2}$ which coincides with the classical Rollnik-class given by \eqref{class}. A
consequence of part (3) in Proposition \ref{cor9} and the dominated convergence theorem is the following
result.
\begin{corollary}
\label{cor:BBM}
Let $d=3$ and $V \in \cR_{3,\alpha_0}\cap \cR_{3,2}$ for some $\alpha_0 \in \left(\frac{3}{2},2\right)$. Then
the function $\alpha \in [\alpha_0,2] \mapsto \Norm{V}{\cR_{d,\alpha}}$ is continuous.
\end{corollary}

\begin{remark}
{\rm
We see that for $\mathcal{R}_{d,2}$ the only possibility permitted by \eqref{standassmp} is $d=3$ and
$\mathcal{R}_{3,2} = \mathcal{R}$. Corollary \ref{cor:BBM} shows, moreover, continuity as $\alpha \uparrow 2$.
Also, for $d=3$ and $\alpha=2$, estimate \eqref{pro8_1} gives
\begin{equation*}
\|V\|_{\mathcal{R}}\leq\sqrt{3}(2\pi)^{1/3}\|V\|_{1}^{1/3}\|V\|_{2}^{2/3},
\end{equation*}
which coincides with the standard norm estimate for the classical Rollnik-class.
}
\end{remark}
Next we compare the fractional Rollnik-class with the domain of the Riesz potential. Recall the definition in
\eqref{eq:Xbd}.
\begin{proposition}
The following properties hold.
\begin{enumerate}
\item
$\cR_{d,\alpha}\subset {\mathfrak X}^d_{\alpha-\frac{d}{2}}$.
\item
If $V \in \cR_{d,\alpha}$ satisfies $|\{V(x)=0\}|=0$, then $V \in {\mathfrak X}^d_{2\alpha-d}$.
\end{enumerate}
\end{proposition}
\begin{proof}
To prove (1), just observe that by \eqref{eq:noabs} it follows that if $V \in \cR_{d,\alpha}$, then
$\cI_{\alpha-\frac{d}{2}}|V| \in L^2(\R^d)$, in particular it is a.e. finite. Hence $|V| \in
{\mathfrak X}^d_{\alpha-\frac{d}{2}}$ and then also $V \in {\mathfrak X}^d_{\alpha-\frac{d}{2}}$. Next
consider (2) and define the measure $\mu_V(dx)=|V(x)|dx$. Since $|\{V(x)=0\}|=0$, the measure $\mu_V$ is
equivalent to Lebesgue measure. Since $V \in \cR_{d,\alpha}$, we obtain
\begin{equation*}
\int_{\R^d}\cI_{2\alpha-d}|V|(x)\mu_V(dx)=\int_{\R^d}|V(x)|\cI_{2\alpha-d}|V|(x)dx<\infty,
\end{equation*}
hence $\cI_{2\alpha-d}|V|$ is $\mu_V$-a.e. finite. Since $\mu_V$ and Lebesgue measure are equivalent,
$\cI_{2\alpha-d}|V|$ is a.e. finite and then $|V| \in {\mathfrak X}^d_{2\alpha-d}$, which in turn implies
$V \in {\mathfrak X}^d_{2\alpha-d}$.
\end{proof}

Finally, we discuss the relationship between fractional Rollnik-class and relativistic Kato-class by borrowing
some results from Section 3.3.4 below. Recall that a potential $V$ is in relativistic Kato-class
$\mathcal{K}_{d,\alpha}$ if
\begin{equation}
\label{eq:Katoclass}
\lim_{\delta \downarrow 0}\sup_{x \in \R^{d}}\int_{|x-y|<\delta}R_1^m(x-y)|V(y)|dy=0,
\end{equation}
and $V$ is said to be relativistic Kato-decomposable if $V_- \in \mathcal{K}_{d,\alpha}$ and $V_+\mathbf{1}_K
\in \mathcal{K}_{d,\alpha}$ for all compact sets $K \subset \R^d$, where $V^+, V^-$ denote positive and negative
parts, respectively, see \cite[Def. 4.280]{LHB}. Using the resolvent estimates, we can provide an alternative
formulation of \eqref{eq:Katoclass}.
\begin{proposition}
\label{kato}
Let $d \ge 1$, $\alpha \in (0,2)$, and $V:\R^d \to \R$ be a potential. For any $\delta>0$ and $x \in \R^d$ define
\begin{equation*}
{}_\delta \mathcal{I}_\alpha V(x)
=\begin{cases} \displaystyle \int_{|x-y|<\delta}\frac{|V(y)|}{|x-y|^{d-\alpha}}dy & \alpha<d\\
				\displaystyle \int_{|x-y|<\delta}\log\left(1+\frac{1}{|x-y|}\right)|V(y)|dy & \alpha=d=1 \\
				\displaystyle \int_{|x-y|<\delta}|V(y)|dy & \alpha>d=1.
\end{cases}
\end{equation*}
Then $V \in \mathcal{K}_{d,\alpha}$ holds if and only if $\lim\limits_{\delta \downarrow 0}\sup_{x \in \R^d}{}_\delta
\cI_\alpha V(x)=0$.
\end{proposition}
\begin{proof}
The statement easily follows by \eqref{eq:equivrollext2}, which implies that there exists a constant $C_{d,\alpha,m}$
such that for all $\delta<1$ and $x \in \R^d$
$$
\frac{1}{C_{d,\alpha,m}} {}_\delta \mathcal{I}_\alpha V(x) \le \int_{|x-y|<\delta}R_1^m(x-y)V(y)dy
\le C_{d,\alpha,m} {}_\delta \mathcal{I}_\alpha V(x).
$$
\end{proof}
\noindent
Notice that, up to a constant, for $\alpha<d$ the expression ${}_\delta \mathcal{I}_\alpha$ is a truncated Riesz potential.
Furthermore, we note that the above result extends \cite[Prop. 4.295]{LHB} to the massive case, also proving the converse
implication. It is seen that the condition defining relativistic Kato-class is independent of the mass $m$.
	
Using the above characterization we can check that if $\alpha<d$, the Coulomb-type potential $V(x)=|x|^{-\beta}$,
$\beta<\alpha$, belongs to $\cK_{d,\alpha}$ (see \cite[Ex. 4.296]{LHB}). As a consequence, also the local Coulomb-type
potential $V(x)=|x|^{-\beta}\mathbf{1}_{\{|x|<1\}} \in \cK_{d,\alpha}$ whenever $\beta<\alpha$. In Proposition
\ref{prop:Coul1} below we show that this potential is also in fractional Rollnik-class, which then implies that
$\cK_{d,\alpha} \cap \cR_{d,\alpha} \not = \emptyset$. On the other hand, the outer Coulomb-type potential $V(x)=
|x|^{-\beta}\mathbf{1}_{\{|x| \ge 1\}} \in L^\infty(\R^d)\subset \cK_{d,\alpha}$ for all $\beta>0$, while it belongs
to $\cR_{d,\alpha}$ if and only if $\beta>\alpha$, see Proposition~\ref{prop:Coul3} below. Hence $\cK_{d,\alpha}
\setminus \cR_{d,\alpha} \not = \emptyset$. On the other hand, taking a critical local Coulomb-type
potential with logarithmic correction $V(x)=|x|^{-\alpha}(-\log|x|)^{-\gamma}\mathbf{1}_{\{|x|<1/2\}}$, $\gamma
\in (\frac{1}{2},1)$, in Proposition~\ref{prop:logV} below we show that $V \in \cR_{d,\alpha}$, while
\begin{equation*}
{}_\delta \mathcal{I}_\alpha V(0)=\sigma_d\int_0^\delta r^{-1}(-\log r)^{-\gamma}dr=\infty \quad
\mbox{for all $\delta< \frac{1}{2}$},
\end{equation*}
hence $V \not \in \cK_{d,\alpha}$. This shows that $\cR_{d,\alpha}\setminus \cK_{d,\alpha} \not = \emptyset$. Furthermore,
it is clear that $V(x)=|x|^{-\alpha}(-\log|x|)^{-\gamma}\mathbf{1}_{\{|x|<1/2\}}$ is not even Kato-decomposable. Hence
it is seen that the fractional Rollnik-class and the relativistic Kato-class overlap, but neither contains the other.

\subsection{Specific cases}
\subsubsection{Potentials with compact support}
Let $V_K=-\mathbf{1}_{K}$ for a compact set $K \subset \R^d$. Then clearly $V_K \in L^{d/\alpha}(\R^d)
\subset \cR_{d,\alpha}$. We compare the Rollnik norm of $V_K$ with the norm of $V_{B^K}$, where $B^K$ is
a ball centred in the origin having the same volume as $K$. Recall the fractional perimeter of index $s
\in (0,1)$  of a Borel set $E \subset \R^d$ defined by
$$
\Per_s (E) = \frac{s2^{s-1}\Gamma\left(\frac{d+s}{2}\right)}{\pi^{d/2}\Gamma\left(1-\frac{s}{2}\right)}
\int_E \int_{\R^d\setminus E} \frac{dxdy}{|x-y|^{d+s}}.
$$
A sufficient condition for a set $E$ to have a finite fractional perimeter of order $s$ is that its indicator
function $\mathbf{1}_{E}$ belongs to the fractional Sobolev space $W^{\frac{s}{p},p}(\R^d)$. In particular, $E$
has finite fractional perimeter if it has finite classical perimeter $\Per(E)$. The inequality
\begin{equation}
\label{eq:isopineq}
\Per_s(E) \ge \Per_s(B^E), \ s \in (0,1]
\end{equation}
is well-known. Indeed, the case $s=1$ is the classical isoperimetric inequality, while for $s \in (0,1)$ it can
be obtained as a consequence of \cite[Th. 4.1]{FS08}. For quantitative versions of \eqref{eq:isopineq} we
refer to \cite{FMM11,F15}.

A straightforward consequence of the Riesz rearrangement inequality \cite{R30,B96} gives
\begin{equation}
\label{Rieszineq}
\Norm{V_K}{\cR_{d,\alpha}} \le \Norm{V_{B^K}}{\cR_{d,\alpha}}, \quad 0<\alpha <d<2 \alpha.
\end{equation}
Combining the two inequalities above, we get for any compact set $K$ with finite $s$-perimeter
\begin{equation*}
\Norm{V_K}{\cR_{d,\alpha}}^2 \le \Norm{V_{B^K}}{\cR_{d,\alpha}}^2=\frac{\Norm{V_{B^K}}{\cR_{d,\alpha}}^2}
{\Per_s(B^K)}\Per_s(B^K) \le \frac{\Norm{V_{B^K}}{\cR_{d,\alpha}}^2}{\Per_s(B^K)}\Per_s(K),
\end{equation*}
where the constant $\Norm{V_{B^K}}{\cR_{d,\alpha}}/\Per_s(B^K)$ is sharp in the sense that all the
inequalities become equalities if $K=B^K$. The constant can be exactly evaluated. Indeed, by \cite[Prop. 2.3]{FFMMM15}
we have
\begin{equation*}
\Per_s(B^K) =
\begin{cases} \left(\frac{|K|}{\omega_d}\right)^{\frac{d-s}{d}}\frac{\left(\Gamma\left(\frac{d+s}{2}\right)\right)^2}
{2\pi\Gamma\left(1-\frac{s}{2}\right)\Gamma\left(\frac{d-s}{2}+1\right)}\sigma_{d} & s \in (0,1)\\
\left(\frac{|K|}{\omega_d}\right)^{\frac{d-1}{d}}\sigma_d & s=1.
\end{cases}
\end{equation*}
On the other hand, by using \cite[Prop. 7.5]{FFMMM15} we obtain
\begin{equation*}
\Norm{V_{B^K}}{\cR_{d,\alpha}}^2=\left(\frac{|K|}{\omega_d}\right)^{\frac{d+\alpha}{d}}
\frac{\sigma_d}{\alpha(d+\alpha)}\mu_{2\alpha-d},
\end{equation*}
where
\begin{equation*}
\mu_{2\alpha-d} =
\begin{cases} 2^{1+2\alpha-d}\pi^{\frac{d-1}{2}}\frac{\Gamma\left(\frac{1+2\alpha-d}{2}\right)}
{\Gamma(\alpha)} & \alpha \not = \frac{d+1}{2}\\
\frac{16\pi^\frac{d-1}{2}}{\Gamma\left(\frac{d+1}{2}\right)} & \alpha=\frac{d+1}{2}.
\end{cases}
\end{equation*}

\subsubsection{Asymmetric fractional Rollnik potentials}
By \eqref{Rieszineq} symmetric rearrangement increases the fractional Rollnik norm, and since the kernel in the
Rollnik integral is spherically symmetric, the inequality reduces to equality only for spherically symmetric
potentials. We can use the fractional Rollnik norm to estimate the symmetry deficit of a non-radial $L^\infty(\R^d)
\cap L^1(\R^d)$ potential from a rotationally symmetric potential. We discuss this by comparing with a spherical
potential well.

Let $V \in L^1(\R^d)\cap L^\infty(\R^d)$ with $V \le 0$ a.e. Clearly, $V \in L^{d/\alpha}(\R^d)\subset \cR_{d,\alpha}$.
Denote by $B^V$ the ball centered in the origin with volume $\Norm{V}{1}/\Norm{V}{\infty}$, and consider the
family of potential wells $V_{B,a}:=-\Norm{V}{\infty}\mathbf{1}_{B^V+a}$, for $a \in \R^d$. As a direct consequence
of \cite[Th. 4]{FL21} we get
\begin{equation*}
\Norm{V_{B,0}}{\cR_{d,\alpha}}^2-\Norm{V}{\cR_{d,\alpha}}^2 \ge \frac{C_{d,\alpha}}{2}
\left(\frac{\Norm{V}{\infty}}{\Norm{V}{1}}\right)^{\frac{2(d-\alpha)}{d}}\inf_{a \in \R^d}
\Norm{V-V_{B,a}}{1}^2,
\end{equation*}
i.e., the difference of the fractional Rollnik norms of $V$ and the spherical potential well $V_{B}$ controls from
above an $L^1$ distance of $V$ from the family of spherical potential wells with the same $L^\infty$ and $L^1$ norms.

\subsubsection{Rotationally symmetric potentials}
Consider a generic potential $V$ and denote by $V^\ast$ the symmetric decreasing rearrangement of $|V|$. Then the
following property holds.
\begin{proposition}
Let $d \in \N$ and $\alpha \in (0,2)$ such that $\alpha<d<2\alpha$. If $V^\ast \in \cR_{d,\alpha}$, then
$V \in \cR_{d,\alpha}$ and
\begin{equation}
\label{eq:Riesz}
\Norm{V}{\cR_{d,\alpha}} \le \Norm{V^\ast}{\cR_{d,\alpha}}.
\end{equation}
Equality holds if and only if there exists a matrix $A \in \R^{d\times d}$ with $|{\rm det}(A)|=1$ and a vector
$a \in \R^d$ such that
\begin{equation}
\label{eq:translation}
|V(x)|=V^\ast(Ax+a).
\end{equation}
Furthermore, if $V$ is continuous, then equality holds if and only if there exists a matrix $A \in \R^{d\times d}$ with
$|{\rm det}(A)|=1$ and a vector $a \in \R^d$ such that
\begin{equation}
\label{eq:trans1}
\mbox{either} \quad V(x)=V^\ast(Ax+a), \quad \mbox{ or } \quad V(x)=-V^\ast(Ax+a).
\end{equation}
\end{proposition}
\begin{proof}
The inequality \eqref{eq:Riesz} follows by a direct application of Riesz rearrangement inequality \cite{R30}. If equality
holds in \eqref{eq:Riesz}, then we are under the assumptions of \cite[Lem. 3]{L77} and \eqref{eq:translation} holds. If,
moreover, $V$ is continuous, then $V^\ast(x)=0$ for some $x$ with $|x|=R$ implies $V^\ast(x)=0$ for all
$x \in \R^d$ with $|x| \ge R$. Thus either $V^\ast(Ax+a)>0$ for all $x \in \R^d$, or there exists a bounded open set $K
\subset \R^d$ such that $V^\ast(Ax+a)=0$ exactly when $x \not \in K^c$. By \eqref{eq:translation} this means that $V$
cannot change sign and so \eqref{eq:trans1} holds.
\end{proof}

We focus now on the case when $V$ is a non-negative radially symmetric potential, with a non-increasing profile (or,
equivalently, on $V^\ast$ for any potential $V$). Let this profile be $v:[0,\infty) \to [0,\infty]$ such that $v(t)
=V(x)$ for all $x \in \R^d$ with $|x|=t$. Making use of \eqref{eq:L1bound} we show a bound on $v$ near the origin.
\begin{theorem}
\label{nomore}
Let $d \in \N$ and $\alpha \in (0,2)$ with $\alpha<d<2\alpha$. Suppose that $V \in \cR_{d,\alpha}$ is non-negative,
radially symmetric, with non-increasing profile $v$. Then there exists $t_0>0$ such that
\begin{equation*}
v(t) < \frac{2^{d-\alpha}\Norm{V}{\cR_{d,\alpha}}}{\omega_d}t^{-\alpha}, \quad t \in (0,t_0).
\end{equation*}
\end{theorem}
\begin{proof}
Consider the function
\begin{equation*}
U(r)=\int_0^r t^{d-1}v(t)dt \le \frac{2^{d-\alpha}\Norm{V}{\cR_{d,\alpha}}}{\sigma_d}r^{d-\alpha},
\end{equation*}
where the bound follows by \eqref{eq:L1bound}. Assume, to the contrary, that there exists a sequence $\seq t$ such that
$t_n \downarrow 0$ and
$v(t_n) \ge \frac{2^{d-\alpha}\Norm{V}{\cR_{d,\alpha}}}{\omega_d}t_n^{-\alpha}$.	
Then we have
\begin{equation*}
U(t_0)=\sum_{n=0}^{\infty}\int_{t_{n+1}}^{t_n}t^{d-1}v(t)dt
\ge \frac{2^{d-\alpha}\Norm{V}{\cR_{d,\alpha}}}{\sigma_d}\sum_{n=0}^{\infty}t_n^{-\alpha}(t_n^d-t_{n+1}^d)
=\frac{2^{d-\alpha}\Norm{V}{\cR_{d,\alpha}}}{\sigma_d}\sum_{n=0}^{\infty}(t_n^{d-\alpha}-t_n^{-\alpha}t_{n+1}^d).
\end{equation*}
Since $t_n>t_{n+1}$ and thus $t_n^{-\alpha}<t_{n+1}^{-\alpha}$, in the end we have
\begin{equation*}
U(t_0)> \frac{2^{d-\alpha}\Norm{V}{\cR_{d,\alpha}}}{\sigma_d}\sum_{n=0}^{\infty}(t_n^{d-\alpha}-t_{n+1}^{d-\alpha})
=\frac{2^{d-\alpha}\Norm{V}{\cR_{d,\alpha}}}{\sigma_d}t_0^{d-\alpha},
\end{equation*}	
which is impossible.
\end{proof}

\subsubsection{Coulomb-type potentials}
We give an alternative expression of the fractional Rollnik norm for radial potentials,
which will be useful below. Let $V(x)=v(|x|)$ with a function $v:\R_0^+ \to \mathbb{C}$. Then, by making use of
\eqref{eq:radsympot} we have
\begin{align}
\label{eq:radialexpr}
\begin{split}
\Norm{V}{\mathcal{R}_{d,\alpha}}^2&=\gamma_d(2\alpha-d)\int_{\R^d}V(x)\mathcal{I}_{2\alpha-d}V(x)dx\\
&=\int_{\R^d}v(|x|)\int_0^{\infty}r^{2\alpha-d-1}F_{2\alpha-d}\left(\frac{|x|}{r}\right)v(r)dr\, dx\\
&=\sigma_d\int_{0}^{\infty}\int_{0}^{\infty}s^{d-1}r^{2\alpha-d-1}v(s)F_{2\alpha-d}\left(\frac{s}{r}\right)v(r)dr \, ds.
\end{split}
\end{align}

\begin{proposition}
\label{prop:Coul1}
The potential $V(x)=|x|^{-\beta}\mathbf{1}_{\{|x|<1\}}$ belongs to $\cR_{d,\alpha}$ if and only
if $\beta < \alpha$.
\end{proposition}
\begin{proof}
By \eqref{eq:radialexpr} we get
\begin{align*}
\Norm{V}{\cR_{d,\alpha}}^2&=\sigma_d\int_{0}^{1}\int_0^{1}s^{d-1-\beta} r^{2\alpha-d-1-\beta}
F_{2\alpha-d}\left(\frac{s}{r}\right)ds\, dr\\
&=\sigma_d
\int_{0}^{1}\int_0^{1/r}z^{d-\beta-1} r^{2(\alpha-\beta)-1}F_{2\alpha-d}(z)dz\, dr\\
&=\sigma_d
\int_{0}^{1}\int_0^{1}z^{d-\beta-1} r^{2(\alpha-\beta)-1}F_{2\alpha-d}(z)dz\, dr+\sigma_d
\int_{0}^{1}\int_1^{1/r^2}z^{d-\beta-1} r^{2(\alpha-\beta)-1}F_{2\alpha-d}(z)dz\, dr\\
&=\sigma_d
\int_{0}^{1}\int_0^{1}z^{d-\beta-1} r^{2(\alpha-\beta)-1}F_{2\alpha-d}(z)dr\, dz+\sigma_d
\int_{1}^{\infty}\int_{0}^{1/z}z^{d-\beta-1} r^{2(\alpha-\beta)-1}F_{2\alpha-d}(z)dr\, dz\\
&=\frac{\sigma_d}{2(\alpha-\beta)}
\int_{0}^{1}z^{d-\beta-1}
F_{2\alpha-d}(z)dz+\frac{\sigma_d}{2(\alpha-\beta)}\int_{1}^{\infty}z^{d-2\alpha+\beta-1}
F_{2\alpha-d}(z) dz,
\end{align*}
where integration with respect to $r$ was possible since $\beta<\alpha$, while otherwise the integral is
divergent. The first integral is finite since $z^{d-\beta-1}F_{2\alpha-d}(z) \sim \sigma_d z^{d-\beta-1}$
as $z \to 0$, which is integrable near zero, while $F_{2\alpha-d}$ is always integrable near $1$. For the
second integral observe that $z^{d-2\alpha+\beta-1}F_{2\alpha-d}(z)\sim \sigma_dz^{\beta-1-d}$ as $z \to
\infty$, which is integrable at infinity since $\beta<d$.
\end{proof}
In case $\beta=\alpha$, we can make a finer analysis including logarithmic corrections.
\begin{proposition}
\label{prop:logV}
The potential $V(x)=|x|^{-\alpha}(-\log|x|)^{-\gamma}\mathbf{1}_{\{|x|<1/2\}}$, $\alpha < d$, belongs to
$\cR_{d,\alpha}$ if and only if $\gamma>\frac{1}{2}$.
\end{proposition}
\begin{proof}
By \eqref{eq:radialexpr} we obtain
\begin{align*}
\Norm{V}{\cR_{d,\alpha}}^2
&=\sigma_d\int_{0}^{1/2}\int_0^{1/2}s^{d-1-\alpha} r^{\alpha-d-1}(-\log s)^{-\gamma}(-\log r)^{-\gamma}
F_{2\alpha-d}\left(\frac{s}{r}\right)ds\, dr\\
&=\sigma_d
\int_{0}^{1/2}\int_0^{1/2r}z^{d-\alpha-1}(-\log z -\log r)^{-\gamma} r^{-1}(-\log r)^{-\gamma}
F_{2\alpha-d}(z)dz\, dr\\
&=\sigma_d
\int_{0}^{1/2}\int_0^{1}z^{d-\alpha-1}(-\log z-\log r)^{-\gamma} r^{-1}(-\log r)^{-\gamma}F_{2\alpha-d}(z)dz\, dr\\
&\qquad +\sigma_d
\int_{0}^{1/2}\int_{1}^{1/2r}z^{d-\alpha-1}(-\log z-\log r)^{-\gamma} r^{-1}(-\log r)^{-\gamma}F_{2\alpha-d}(z)dz\, dr\\
&=\sigma_d
\int_{0}^{1}\int_0^{1/2}z^{d-\alpha-1}(-\log z-\log r)^{-\gamma} r^{-1}(-\log r)^{-\gamma}F_{2\alpha-d}(z)dr\, dz\\
&\qquad +\sigma_d
\int_{1}^{\infty}\int_{0}^{1/2z}z^{d-\alpha-1}(-\log z-\log r)^{-\gamma} r^{-1}(-\log r)^{-\gamma}F_{2\alpha-d}(z)dr\, dz\\
&=\sigma_d(I_1+I_2).
\end{align*}
If $\gamma>\frac{1}{2}$, then
\begin{align*}
I_1&= \int_0^{1}z^{d-\alpha-1}F_{2\alpha-d}(z)\int_{\log 2}^{\infty}
x^{-\gamma}\left(x-\log z\right)^{-\gamma}dx\, dz\\
&= \int_0^{1}z^{d-\alpha-1}F_{2\alpha-d}(z)\int_{\log 2-\log z}^{\infty}
(x+\log z)^{-\gamma}x^{-\gamma}dx\, dz\\
&= \int_0^{1}z^{d-\alpha-1}F_{2\alpha-d}(z)(-\log z)^{-2\gamma}\int_{\log 2-\log z}^{\infty}
(1-\frac{-\log z}{x})^{-\gamma}\left(\frac{x}{-\log z}\right)^{-2\gamma}dx\, dz\\
&= \int_0^{1}z^{d-\alpha-1}F_{2\alpha-d}(z)(-\log z)^{1-2\gamma}\int_{0}^{\frac{-\log z}{\log 2-\log z}}
(1-x)^{-\gamma}x^{-2\gamma-2}dx\, dz\\
&=\int_0^{1}z^{d-\alpha-1}F_{2\alpha-d}(z)(-\log z)^{1-2\gamma}\int_{0}^{\frac{-\log z}{\log 2-\log z}}
B\left(\frac{-\log z}{\log 2-\log z};2\gamma-1,1-\gamma\right)\, dz,
\end{align*}
with the incomplete Beta-function $B(\cdot; \cdot, \cdot)$. In the limit $z \downarrow 0$ we have
\begin{equation*}	
B\left(\frac{-\log z}{\log 2-\log z};2\gamma-1,1-\gamma\right) \to B(2\gamma-1,1-\gamma),
\end{equation*}
hence
\begin{equation*}
z^{d-\alpha-1}\log^{1-2\gamma}(z)F_{2\alpha-d}(z)B\left(\frac{-\log z}{\log 2-\log z};2\gamma-1,1-\gamma\right)
\sim B(1-\gamma,2\gamma-1)\sigma_d z^{d-\alpha-1}\log^{1-2\gamma}(z),
\end{equation*}
where the right hand side is integrable in a neighbourhood of zero. As to the behaviour at $1$, recall that the
incomplete Beta-function satisfies
\begin{equation*}
B\left(z;2\gamma-1,1-\gamma\right) \sim \frac{z^{2\gamma-1}}{2\gamma-1} \quad \mbox{ as $z \downarrow 0$}.
\end{equation*}
Indeed, by \cite[eq. (8.391)]{GR14}, we have
\begin{equation*}
B\left(z;2\gamma-1,1-\gamma\right)=\frac{z^{2\gamma-1}}{2\gamma-1}{}_2F_1(2\gamma-1, \gamma; 2\gamma;z),
\end{equation*}
where ${}_2F_1(a,b;c,z)$ is Gauss' hypergeometric function, so it is clear that
\begin{equation*}
\lim_{z \to 0^+}\frac{(2\gamma-1)B\left(z;2\gamma-1,1-\gamma\right)}{z^{2\gamma-1}}
={}_2F_1(2\gamma-1, \gamma; 2\gamma;0)=1.
\end{equation*}
Hence we get, in general,
\begin{equation*}
z^{d-\alpha-1}\log^{1-2\gamma}(z)F_{2\alpha-d}(z)B\left(\frac{-\log z}{\log 2-\log z};2\gamma-1,
1-\gamma\right) \sim \frac{F_{2\alpha-d}(z)}{2\gamma-1} \quad \mbox{ as $z \downarrow 1$},
\end{equation*}
where the right-hand side is integrable near $1$. This proves that $I_1<\infty$. Next we estimate $I_2$ and write
\begin{align*}
I_2&=\int_1^{\infty}z^{d-\alpha-1}F_{2\alpha-d}(z)\int_{\log(2z)}^{\infty}x^{-\gamma}(x-\log z)^{-\gamma}dx\, dz\\
&=\int_1^{\infty}z^{d-\alpha-1}\log^{1-2\gamma}(z)F_{2\alpha-d}(z)\int_{0}^{\log z/(\log 2+\log z)}x^{2\gamma-2}
\left(1-x\right)^{-\gamma}dx\, dz\\
&=\int_1^{\infty}z^{d-\alpha-1}\log^{1-2\gamma}(z)F_{2\alpha-d}(z)
B\left(\frac{\log z}{\log 2+\log z};2\gamma-1,1-\gamma\right),
\end{align*}
with the incomplete Beta-function $B(\cdot; \cdot, \cdot)$. We have
\begin{equation*}
B\left(\frac{\log z}{\log 2+\log z};2\gamma-1,1-\gamma\right) \to B(1-\gamma,2\gamma-1)
\end{equation*}
as $z \to \infty$, hence
\begin{equation*}
z^{d-\alpha-1}\log^{1-2\gamma}(z)F_{2\alpha-d}(z)B\left(\frac{\log z}{\log 2+\log z};2\gamma-1,1-\gamma\right) \sim
B(1-\gamma,2\gamma-1)\sigma_d z^{\alpha-d-1}\log^{1-2\gamma}(z),
\end{equation*}
which is integrable near infinity since $\alpha<d$. Again, to handle the behaviour at $1$, note that
\begin{equation*}
z^{d-\alpha-1}\log^{1-2\gamma}(z)F_{2\alpha-d}(z)B\left(\frac{\log z}{\log 2+\log z};2\gamma-1,1-\gamma\right)
\sim \frac{F_{2\alpha-d}(z)}{2\gamma-1} \quad \mbox{ as $z \downarrow 1$},
\end{equation*}
where the right-hand side is integrable near $1$.
	
Finally, consider the case $\gamma \le \frac{1}{2}$. Then
\begin{equation*}
I_1 \ge \sigma_d \int_0^{1}z^{d-\alpha-1}F_{2\alpha-d}(z)\int_1^{\infty}\frac{1}{(-\log z+x)^{2\gamma}}dx\, dz=\infty
\end{equation*}
due to $2\gamma \le 1$.
\end{proof}
\begin{remark}
{\rm	
Notice that for $V$ as in Proposition \ref{prop:logV} we have
\begin{eqnarray*}
\int_{\R^d}|V(x)|^{d/\alpha}dx
&=&
\int_{B_{1/2}}|x|^{-d}(-\log(|x|))^{-\frac{\gamma d}{\alpha}}dx \\
&=&
\sigma_d \int_0^{1/2}r^{-1}(-\log r)^{-\frac{\gamma d}{\alpha}}dr=\sigma_d \int_{\log 2}^{\infty}
x^{-\frac{\gamma d}{\alpha}}dx,
\end{eqnarray*}
which is finite if and only if $\gamma>\frac{\alpha}{d}$. If $d=3$ and $\alpha>\frac{3}{2}$ or if $d=2$ and $\alpha>1$, then
$\frac{\alpha}{d}>\frac{1}{2}$. Hence, under these conditions, whenever $\gamma \in
\left(\frac{1}{2},\frac{\alpha}{d}\right)$, we have $V \in \cR_{d,\alpha}$ but $V \not \in L^{d/\alpha}(\R^d)$.
}
\end{remark}

Next we study the possible decay at infinity of fractional Rollnik potentials.
\begin{proposition}
\label{prop:Coul3}
The potential $V(x)=|x|^{-\beta}\mathbf{1}_{\{|x| \ge 1\}}$ belongs to $\cR_{d,\alpha}$ if and only if
$\beta \in (\alpha,d)$.
\end{proposition}
\begin{proof}
By \eqref{eq:radialexpr} we get
\begin{align*}
	\Norm{V}{\cR_{d,\alpha}}^2&=\sigma_d\int_{1}^{\infty}\int_1^{\infty}s^{d-1-\beta} r^{2\alpha-d-1-\beta}
	F_{2\alpha-d}\left(\frac{s}{r}\right)ds\, dr\\
	&=\sigma_d
	\int_{1}^{\infty}\int_{1/r}^{\infty}z^{d-\beta-1} r^{2(\alpha-\beta)-1}F_{2\alpha-d}(z)dz\, dr\\
	&=\sigma_d
	\int_{1}^{\infty}\int_{1/r}^{1}z^{d-\beta-1} r^{2(\alpha-\beta)-1}F_{2\alpha-d}(z)dz\, dr+\sigma_d
	\int_{1}^{\infty}\int_1^{\infty}z^{d-\beta-1} r^{2(\alpha-\beta)-1}F_{2\alpha-d}(z)dz\, dr\\
	&=\sigma_d
	\int_{1}^{\infty}\int_{1/z}^{\infty}z^{d-\beta-1} r^{2(\alpha-\beta)-1}F_{2\alpha-d}(z)dr\, dz+\sigma_d
	\int_{1}^{\infty}\int_{1}^{\infty}z^{d-\beta-1} r^{2(\alpha-\beta)-1}F_{2\alpha-d}(z)dr\, dz\\
	&=\frac{\sigma_d}{2(\beta-\alpha)}
	\int_{1}^{\infty}z^{d-2\alpha+\beta-1}
	F_{2\alpha-d}(z)dz dz+\frac{\sigma_d}{2(\beta-\alpha)}\int_{1}^{\infty}z^{d-\beta-1}F_{2\alpha-d}(z) dz,
\end{align*}
where again integration on $r$ makes since since $\beta>\alpha$ and else diverge. For the former integral notice
that $F_{2\alpha-d}$ is integrable near $1$, and
\begin{equation*}
z^{d-2\alpha+\beta-1}F_{2\alpha-d}\sim \sigma_d z^{\beta-1-d} \mbox{ as }z \to \infty,
\end{equation*}
which is integrable if and only if $\beta<d$. For the latter integral, observe again that $F_{2\alpha-d}$ is
integrable near $1$, while
\begin{equation*}
z^{d-\beta-1}F_{2\alpha-d}\sim \sigma_d z^{2\alpha-d-\beta-1} \mbox{ as }z \to \infty,
\end{equation*}
which is integrable since $2\alpha-d<\alpha<\beta$.
\end{proof}
\begin{proposition}
The potential $V(x)=|x|^{-\alpha}\log^{-\gamma}(x)\mathbf{1}_{\{|x| \ge 2\}}$ belongs to $\cR_{d,\alpha}$ if
and only if $\gamma>\frac{1}{2}$.
\end{proposition}
\begin{proof}
By \eqref{eq:radialexpr} we get
\begin{align*}
\Norm{V}{\cR_{d,\alpha}}^2
&=
\sigma_d\int_{2}^{\infty}\int_2^{\infty}s^{d-1-\alpha} r^{\alpha-d-1}\log^{-\gamma}(s)\log^{-\gamma}(r)
F_{2\alpha-d}\left(\frac{s}{r}\right)ds\, dr\\
&=
\sigma_d \int_{2}^{\infty}\int_{2/r}^{\infty}z^{d-\alpha-1} r^{-1}(\log(r)+\log z)^{-\gamma}\log^{-\gamma}(r)
F_{2\alpha-d}(z)dz\, dr\\
&=
\sigma_d \int_{2}^{\infty}\int_{2/r}^{1}z^{d-\alpha-1} r^{-1}(\log(r)+\log z)^{-\gamma}\log^{-\gamma}(r)
F_{2\alpha-d}(z)\, dr\\
& \quad +
\sigma_d \int_{2}^{\infty}\int_1^{\infty}z^{d-\alpha-1} r^{-1}(\log(r)+\log z)^{-\gamma}\log^{-\gamma}(r)
F_{2\alpha-d}(z)dz\, dr\\
&=
\sigma_d \int_{0}^{1}\int_{2/z}^{\infty}z^{d-\alpha-1} r^{-1}(\log(r)+\log z)^{-\gamma}\log^{-\gamma}(r)
F_{2\alpha-d}(z)dr\, dz\\
&\quad +
\sigma_d\int_{1}^{\infty}\int_{2}^{\infty}z^{d-\alpha-1} r^{-1}(\log(r)+\log z)^{-\gamma}\log^{-\gamma}(r)
F_{2\alpha-d}(z)dr\, dz\\
&=
\sigma_d(I_1+I_2).
\end{align*}
First consider $I_1$ with $\gamma>\frac{1}{2}$. We have
\begin{align*}
I_1
&=
\int_{0}^{1}z^{d-\alpha-1}F_{2\alpha-d}(z)\int_{\log(2/z)}^{\infty} (x+\log z)^{-\gamma}x^{-\gamma}dx\, dz\\
&=
\int_{0}^{1}z^{d-\alpha-1}F_{2\alpha-d}(z)(-\log z)^{1-2\gamma}\int_{0}^{\frac{-\log z}{\log 2-\log z}}
\left(1-x\right)^{-\gamma}x^{2\gamma-2}dx\, dz\\
&=
\int_{0}^{1}z^{d-\alpha-1}F_{2\alpha-d}(z)(-\log z)^{1-2\gamma}
B\left(\frac{-\log z}{\log 2-\log z};2\gamma-1,1-\gamma\right)dz.
\end{align*}
As in Proposition \ref{prop:logV}, we have
\begin{equation*}
z^{d-\alpha-1}F_{2\alpha-d}(z)(-\log z)^{1-2\gamma}B\left(\frac{-\log z}{\log 2-\log z};2\gamma-1,1-\gamma\right)
\sim \frac{F_{2\alpha-d}(z)}{2\gamma-1} \mbox{ as $z \uparrow 1$},
\end{equation*}
where the right-hand side is integrable near $1$. Concerning the behaviour at $0$, observe that
\begin{multline*}
z^{d-\alpha-1}F_{2\alpha-d}(z)(-\log z)^{1-2\gamma}B\left(\frac{-\log z}{\log 2-\log z};2\gamma-1,1-\gamma\right) \\
\sim \sigma_d z^{d-\alpha-1}(-\log z)^{1-2\gamma}B(2\gamma-1,\gamma-1) \mbox{ as $z \downarrow 0$},
\end{multline*}
which is integrable near $0$ since $\alpha<d$. Hence $I_1<\infty$. Next consider $I_2$. We can rewrite
\begin{align*}
I_2
&=
\int_1^{\infty}z^{d-\alpha-1}F_{2\alpha-d}(z)\int_{\log 2}^{\infty}(x+\log z)^{-\gamma}x^{-\gamma}\, dx \, dz\\
&=
\int_1^{\infty}z^{d-\alpha-1}F_{2\alpha-d}(z)\int_{\log 2+\log z}^{\infty}x^{-\gamma}(x-\log z)^{-\gamma}\, dx \, dz\\
&=
\int_1^{\infty}z^{d-\alpha-1}F_{2\alpha-d}(z)\log^{-2\gamma}(z)
\int_{\log 2+\log z}^{\infty}\left(\frac{x}{\log z}\right)^{-2\gamma}\left(1-\frac{\log z}{x}\right)^{-\gamma}\, dx\, dz\\
&=
\int_1^{\infty}z^{d-\alpha-1}F_{2\alpha-d}(z)\log^{1-2\gamma}(z)\int_0^{\frac{\log z}{\log 2+\log z}}x^{2\gamma-2}
\left(1-x\right)^{-\gamma}\, dx \, dz\\
&=
\int_1^{\infty}z^{d-\alpha-1}F_{2\alpha-d}(z)\log^{1-2\gamma}(z)
B\left(\frac{\log z}{\log 2+\log z};2\gamma-1,1-\gamma\right)\, dz.
\end{align*}
Like before, we thus have
\begin{equation*}
z^{d-\alpha-1}F_{2\alpha-d}(z)\log^{1-2\gamma}(z)B\left(\frac{\log z}{\log 2+\log z};2\gamma-1,1-\gamma\right)
\sim \frac{F_{2\alpha-d}(z)}{2\gamma-1} \mbox{ as $z \downarrow 1$},
\end{equation*}
where the right-hand side is integrable near $1$. Concerning the behaviour at $\infty$, observe that
\begin{equation*}
z^{d-\alpha-1}F_{2\alpha-d}(z)\log^{1-2\gamma}(z)B\left(\frac{-\log z}{\log 2-\log z};2\gamma-1,1-\gamma\right) \sim
\sigma_d z^{\alpha-d-1}\log^{1-2\gamma}(z)B(2\gamma-1,\gamma-1) \mbox{ as $z \uparrow \infty$},
\end{equation*}
which is integrable near $\infty$ since $\alpha<d$. This proves that $I_2<\infty$.
	
For the case $\gamma \le \frac{1}{2}$, note that
\begin{equation*}
I_2 \ge \int_1^{\infty}z^{d-\alpha-1}F_{2\alpha-d}(z)\int_1^{\infty}(x+\log z)^{-2\gamma}dx\, dz =\infty.
\end{equation*}
\end{proof}

\begin{corollary}
\label{coulombsummary}
The potential $V(x) = |x|^{-\beta_1}\mathbf{1}_{\{|x|<1\}}+|x|^{-\beta_2}\mathbf{1}_{\{|x| \ge 1\}}$ belongs to
$\cR_{d,\alpha}$ if and only if $\beta_1<\alpha$ and $\beta_2 \in (\alpha,d)$.
\end{corollary}
\begin{proof}
By Propositions \ref{prop:Coul1} and \ref{prop:Coul3} it follows that if $\beta_1<\alpha$ and $\beta_2 \in (\alpha,d)$,
then $V \in \cR_{d,\alpha}$. Conversely, if $V \in \cR_{d,\alpha}$, then it follows by Remark \ref{rem:upinc} that
both $V_1(x) = |x|^{-\beta_1}\mathbf{1}_{\{|x|<1\}}$ and $V_2(x) = |x|^{-\beta_2}\mathbf{1}_{\{|x| \ge 1\}}$ belong
to $\cR_{d,\alpha}$, which implies $\beta_1<\alpha$ and $\beta_2 \in (\alpha,d)$ again by Propositions \ref{prop:Coul1}
and \ref{prop:Coul3}.
\end{proof}

\subsection{Extended fractional Rollnik-class}
\label{extendrollnik}
Next we introduce another class of potentials  related to the fractional Rollnik-class, constructed on the resolvent
estimates \eqref{eq:Hda0} and \eqref{eq:Hmda}. This class will be justified by the behaviours we will discuss in Section
4 below.
\begin{definition}
\label{extrollnik}
Let $\alpha \in (0,2)$, $d \in \N$, $m \ge 0$, and assume that $d<2\alpha$. We call
\begin{equation*}
\overline{\cR}^m_{d,\alpha}:=\left\{V:\R^d \to \R : [V]_{\overline{\cR}_{d,\alpha}}^2:=\int_{\R^d}\int_{\R^d}
H^m_{d,\alpha}(|x-y|)^2|V(x)||V(y)|dx \, dy<\infty\right\}
\end{equation*}
extended fractional Rollnik-class. If $m=0$, we omit the superscript.
\end{definition}
First we settle the relationship between the fractional Rollnik-classes and their extensions.

\begin{proposition}
\label{includes}
Let $\alpha \in (0,2)$, $d \in \N$ and assume that $d<2\alpha$.
\begin{enumerate}
\item
If $d>\alpha$, we have $\cR_{d,\alpha}\subset \overline{\cR}_{d,\alpha}$. In particular, there exists a constant
$C_{d,\alpha}>0$ such that
\begin{equation*}
[V]_{\overline{\cR}_{d,\alpha}} \le C_{d,\alpha}\Norm{V}{\cR_{d,\alpha}}
\end{equation*}
for every $V \in \cR_{d,\alpha}$.
\item
For every $m>0$ we have $\overline{\cR}_{d,\alpha}\subset \overline{\cR}^m_{d,\alpha}$. In particular,
there exists a constant $C_{d,\alpha}>0$ such that
\begin{equation*}
[V]_{\overline{\cR}^m_{d,\alpha}} \le C_{d,\alpha}[V]_{\overline{\cR}_{d,\alpha}}
\end{equation*}
for every $V \in \cR_{d,\alpha}$.
\end{enumerate}
\end{proposition}
\begin{proof}
(1) easily follows since $r \ge 1$ trivially implies $r^{-2(d+\alpha)} \le r^{-2(d-\alpha)}$, and thus
\begin{align*}
[V]_{\overline{\cR}_{d,\alpha}}^2
&=\iint_{|x-y| \le 1}\frac{|V(x)||V(y)|}{|x-y|^{2(d-\alpha)}}dx\, dy+2^{d+\alpha-2}
\Gamma^2\Big(\frac{d+\alpha}{2}\Big)\iint_{|x-y| \ge 1}\frac{|V(x)||V(y)|}{|x-y|^{2(d+\alpha)}}dx\, dy \\
&\le \iint_{|x-y| \le 1}\frac{|V(x)||V(y)|}{|x-y|^{2(d-\alpha)}}dx\, dy+2^{d+\alpha-2}
\Gamma^2\Big(\frac{d+\alpha}{2}\Big)\iint_{|x-y| \ge 1}\frac{|V(x)||V(y)|}{|x-y|^{2(d-\alpha)}}dx\, dy \\
&\le C_{d,\alpha}\Norm{V}{\cR_{d,\alpha}}^2.
\end{align*}
To obtain (2) note that by \eqref{cor4_1} there exists a constant $C_{d,\alpha}>0$ such that $H^m_{d,\alpha}(r)
\le C_{d,\alpha}H_{d,\alpha}(r)$, for all $r>0$.
\end{proof}
The above definitions yield a consistent behaviour in the massless limit.
\begin{proposition}
Let $V \in \overline{\cR}_{d,\alpha}$. Then
$\lim_{m \downarrow  0}[V]_{\overline{\cR}^m_{d,\alpha}}=[V]_{\overline{\cR}_{d,\alpha}}$.
\end{proposition}
\begin{proof}
By Remark \ref{rmk:consistent} in the Appendix we have $\lim_{m \downarrow 0}H_{d,\alpha}^m(r)=
H_{d,\alpha}(r)$, and then the claim follows by the dominated convergence theorem.
\end{proof}
Concerning the relationship between $L^p$-spaces and $\overline{\cR}_{d,\alpha}^m$, we can prove the following
continuous embedding property.
\begin{proposition}
\label{prop:L2inclext}
Let $m \ge 0$. Then $L^2(\R^d)\subset \overline{\cR}_{d,\alpha}^m$. In particular,
\begin{equation}
\label{eq:inclusionL2}
[V]_{\overline{\cR}_{d,\alpha}^m} \le \Norm{H_{d,\alpha}^m}{2}\Norm{V}{2}.
\end{equation}
\end{proposition}
\begin{proof}
Note that if $d<2\alpha$ then $H_{d,\alpha}^m(|\cdot|) \in L^2(\R^d)$ for all $m \ge 0$ and thus \eqref{eq:inclusionL2}
follows by a direct application of Young's convolution inequality.
\end{proof}
\begin{remark}
\label{genuine}
{\rm
Proposition~\ref{prop:L2inclext} guarantees that for $\beta>\frac{d}{2}$ the potential $V(x)=|x|^{-\beta}
\mathbf{1}_{\{|x| \ge 1\}}$ belongs to $\overline{\cR}_{d,\alpha}^m$ for all $m \ge 0$. However, since
$d<2\alpha$, we can take $\beta \in \left(\frac{d}{2},\alpha\right)$ such that $V \not \in \cR_{d,\alpha}$
due to Proposition~\ref{prop:Coul3}, and thus $V \in \overline{\cR}_{d,\alpha} \setminus \cR_{d,\alpha}$.
}
\end{remark}
\rm{
Similarly to \eqref{crunchyrollnik} we define $\overline{\cR}_{d,\frac{d}{2}}^{m,p}$ for some $p > 1$ through
of a limiting argument by
\begin{equation*}
\overline{\cR}_{d,\frac{d}{2}}^{m,p}:= L^p(\R^d) \, \cap \, \liminf_{\alpha \downarrow \frac{d}{2}}
\overline{\cR}_{d,\alpha}^m.
\end{equation*}
By Proposition~\ref{prop:L2inclext}, for $p=2$ we have $\overline{\cR}_{d,\frac{d}{2}}^{m,2}=L^2(\R^d)$.
Furthermore, Proposition~\ref{includes} implies that for every $p>1$
\begin{equation*}
	\cR_{d,\frac{d}{2}}^{p}\subset \overline{\cR}_{d,\frac{d}{2}}^{m,p}.
\end{equation*}
}

\section{Perturbations by fractional Rollnik class potentials}
\subsection{Self-adjointness}\label{subs:selfadj}
In what follows our main interest is to study the properties of fractional Schr\"odinger operators for $m\geq0$
$$
H=L_{m,\alpha}+ V
$$
with fractional Rollnik-class potentials $V$. First we show that $H$ can be defined as a self-adjoint linear
operator in the form sense for $V \in {\mathcal R}_{d,\alpha}$.

Define
\begin{equation*}
K^m_\lambda= |V|^{1/2}(\lambda+L_{m,\alpha})^{-1}|V|^{1/2},
\end{equation*}
which is an integral operator with kernel
$$K^m_\lambda(x-y) = \sqrt{|V(x)|}R^m_\lambda(x-y)\sqrt{|V(y)|}.
$$

First we give a criterion for self-adjointness of the operator $H$.
\begin{lemma}
\label{lem:selfadj}
Fix $m \ge 0$ and assume that $\lim_{\lambda \to \infty}\Norm{K_\lambda^m}{\cL(L^2(\R^d))}=0$. Then $H$ is
self-adjoint on $L^2(\mathbb{R}^d)$ in the form sense.
\end{lemma}

\begin{proof}
By the assumption we have directly that
\begin{equation*}
\|\sqrt{|V|}(\lambda+L_{m,\alpha})^{-1/2}\|_{\mathscr{B}(L^2)}\rightarrow 0
\end{equation*}
as $\lambda\rightarrow\infty$. Define the function
\begin{equation*}
\widetilde{V}^{1/2}(x)=
\begin{cases}
V(x)/\sqrt{|V(x)|}&\mbox{if}\ V(x)\not=0,\\
0 &\mbox{if}\ V(x)=0.
\end{cases}
\end{equation*}
Then $|\widetilde{V}^{1/2}|=\sqrt{|V|}$ and $\tilde{V}^{1/2}\sqrt{|V|}=V$ hold. For $\phi\in L^2(\mathbb{R}^d)$ we
compute
\begin{eqnarray*}
\lefteqn{
\vspace{-3cm}
\|(\lambda+L_{m,\alpha})^{-1/2}V(\lambda+L_{m,\alpha})^{-1/2}\phi\|^2_{2}\nonumber}\\
&=&
(\tilde{V}^{1/2}(\lambda+L_{m,\alpha})^{-1}V(\lambda+L_{m,\alpha})^{-1/2}\phi,
\sqrt{|V|}(\lambda+(-\Delta)^{\alpha/2})^{-1/2}\phi)_{L^2}\nonumber\\
&\leq&
\|K^m_\lambda\|_{\mathscr{B}(L^2)}\|\sqrt{|V|}(\lambda+L_{m,\alpha})^{-1/2}\|_{\mathscr{B}(L^2)}^2
\|\phi\|_{2}^2.
\end{eqnarray*}
Thus we obtain
\begin{equation}
\|(\lambda+L_{m,\alpha})^{-1/2}V(\lambda+L_{m,\alpha})^{-1/2}\|_{\mathscr{B}(L^2)}\rightarrow 0
\label{the11_1}
\end{equation}
as $\lambda\rightarrow\infty$. An application of the KLMN
theorem \cite[Th. X.17]{RS2} completes the proof.
\end{proof}

Using Lemma \ref{lem:selfadj} it remains to show that under specific assumptions on $V$ it follows that
$\lim_{\lambda \to \infty}\Norm{K_\lambda^m}{\cL(L^2(\R^d))}=0$. This can be guaranteed if for a sufficiently
large $\lambda_0>0$ the operator $K^m_{\lambda_0}$ is a Hilbert-Schmidt operator.

\begin{lemma}
Suppose that there exists $\lambda_0>0$ such that $K^m_{\lambda_0}$ is a Hilbert-Schmidt operator. Then
$K^m_{\lambda}$ is a Hilbert-Schmidt operator for all $\lambda>\lambda_0$ and
\begin{equation*}
\lim_{\lambda \to \infty}\Norm{K_\lambda^m}{\cL(L^2(\R^d))}
= \lim_{\lambda \to \infty}\Norm{K_\lambda^m}{\rm HS} =0,
\end{equation*}
where $\Norm{\cdot}{\rm HS}$ is the Hilbert-Schmidt norm.
\end{lemma}

\begin{proof}
Assume that $K^m_{\lambda_0}$ is a Hilbert-Schmidt operator and let $\lambda>\lambda_0$. Then we have
\begin{equation*}
R_\lambda^m(x)=\int_0^{\infty}e^{-\lambda t}p^m_t(x)\, dt \le \int_0^{\infty}e^{-\lambda_0 t}p^m_t(x)\, dt
=R_{\lambda_0}^m(x).
\end{equation*}
Furthermore, by the dominated convergence theorem, it is clear that $\lim_{\lambda \to \infty}R_\lambda(x)=0$.
We get
\begin{eqnarray*}
\Norm{K_\lambda^m}{\rm HS}^2
&=&
\int_{\R^{2d}}|V(x)|  \, R_\lambda^m(x-y)^2  \, |V(y)|dx\, dy
\le  \int_{\R^{2d}}|V(x)| \,  R_{\lambda_0}^m(x-y)^2  \, |V(y)|dx\, dy \\
&=&
\Norm{K_{\lambda_0}^m}{\rm HS}^2<\infty.
\end{eqnarray*}
Again by dominated convergence, it follows that
$\lim_{\lambda \to \infty}\Norm{K_\lambda^m}{\rm HS}=0$.
\end{proof}

Once this established, our goal is determine conditions under which $K_\lambda^m$ is a Hilbert-Schmidt operator
for some $\lambda>0$. First consider the case of the fractional Rollnik-class.

\begin{theorem}
\label{prop:uniformHS}
Let $\alpha \in (0,2)$, $d \ge 1$, $m \geq 0$ and assume that $\alpha<d<2\alpha$. If $V\in\mathcal{R}_{d,\alpha}$,
then $K^m_\lambda$ is a Hilbert-Schmidt operator for every $\lambda>0$. Furthermore, there exists a constant
$C_{d,\alpha}>0$ such that
\begin{equation*}
\Norm{K^m_{\lambda}}{\rm HS} \le\sqrt{ C_{d,\alpha}}\Big(1+\frac{me^{-\lambda/m}}{\lambda}\Big)\Norm{V}{\cR_{d,\alpha}}.
\end{equation*}
(The bracket equals 1 if $m=0$.) In particular, if $V \in \cR_{d,\alpha}$, then $H$ is a self-adjoint operator in the form sense.
\end{theorem}

\begin{proof}
By Lemmas \ref{lem1} and \ref{lem3}, it is straightforward that
\begin{equation}
\iint_{\mathbb{R}^{2d}} K^m_\lambda(x-y)^2dxdy=\iint_{\mathbb{R}^{2d}}R^m_\lambda(x-y)^2|V(x)||V(y)|dxdy
\leq C_{d,\alpha}^2\left(1+\frac{me^{-\frac{\lambda}{m}}}{\lambda}\right)^2\|V\|_{\mathcal{R}_{d,\alpha}}^2
\label{lem10_1}
\end{equation}
for every $\lambda>0$ and $m\geq0$. Lemma \ref{lem:selfadj} then implies that $H$ is self-adjoint on $L^2(\R^d)$ in the form sense.
\end{proof}

We can, however, provide a finer argument to characterize the class of potentials $V$ such that $K_\lambda^m$ is a
Hilbert-Schmidt operator, at least for $\lambda>m$.
\begin{theorem}
\label{thm:HSKlambda}
Let $\alpha \in (0,2)$, $d \ge 1$, $m \ge 0$ and $\lambda>0$  and assume that $d<2\alpha$.
Then the following properties hold:
\begin{enumerate}
\item
If $V \in \overline{\cR}_{d,\alpha}$, then $K_\lambda^m$ is a Hilbert-Schmidt operator and
there exists a constant $C_{d,\alpha,\lambda,m}$, independent of $V$, such that
\begin{equation*}
\Norm{K_\lambda^m}{\rm HS} \le C_{d,\alpha,\lambda,m}[V]_{\overline{\cR}_{d,\alpha}}.
\end{equation*}

\item
If $K_\lambda^m$ is a Hilbert-Schmidt operator, then $V \in \overline{\cR}_{d,\alpha}^m$ and
there exists a constant $C_{d,\alpha,\lambda,m}$, independent of $V$, such that
\begin{equation*}
[V]_{\overline{\cR}_{d,\alpha}^m} \le C_{d,\alpha,\lambda,m}\Norm{K_\lambda^m}{\rm HS}.
\end{equation*}

\item
If $\lambda>m$, then $K_\lambda^m$ is a Hilbert-Schmidt operator if and only if $V \in
\overline{\cR}_{d,\alpha}^m$, and there exists a constant $C_{d,\alpha,\lambda, m}>1$ independent
of $V$ such that
\begin{equation*}
\frac{1}{C_{d,\alpha,\lambda,m}}[V]_{\overline{\cR}_{d,\alpha}^m} \le \Norm{K_\lambda^m}{\rm HS}
\le C_{d,\alpha,\lambda,m}[V]_{\overline{\cR}_{d,\alpha}^m}.
\end{equation*}
\end{enumerate}
\end{theorem}
\begin{proof}
To obtain (1), let $V \in \overline{\cR}_{d,\alpha}$. Then by \eqref{forrollext1}
 we have
\begin{multline*}
\Norm{K_\lambda^m}{\rm HS}^2=\iint_{\R^{2d}}|V(x)| \, R^m_\lambda(x-y)^2 \,|V(y)|dx \, dy \\
\le
C_{d,\alpha,\lambda,m}\iint_{\R^{2d}}|V(x)| \, H_{d,\alpha}^0(|x-y|)^2 \, |V(y)|dx \, dy
=
C_{d,\alpha,\lambda,m}[V]_{\overline{\cR}_{d,\alpha}}^2.
\end{multline*}
Next we show (2). Assume that $K_\lambda^m$ is a Hilbert-Schmidt operator. By \eqref{eq:lowerrollext2}
we then have
\begin{multline*}
[V]_{\overline{\cR}_{d,\alpha}^m}^2=\iint_{\R^{2d}}|V(x)|\, H_{d,\alpha}^m(|x-y|)^2 \, |V(y)|dx \, dy \\
\le
C_{d,\alpha,\lambda,m}\iint_{\R^{2d}}|V(x)| \, R^m_\lambda(x-y)^2 \, |V(y)|dx \, dy
= C_{d,\alpha,\lambda,m}\Norm{K_\lambda^m}{\rm HS}^2.
\end{multline*}
Finally, to show (3) note that one implication has already been shown in (2). If we assume further that
$\lambda>m$ and $V \in \overline{\cR}_{d,\alpha}^m$, then by \eqref{eq:equivrollext2}
we get
\begin{multline*}
\Norm{K_\lambda^m}{\rm HS}^2=\iint_{\R^{2d}}|V(x)| \, R^m_\lambda(x-y)^2 \, |V(y)|dx \, dy \\
\le
C_{d,\alpha,\lambda,m}\iint_{\R^{2d}}|V(x)| \,  H_{d,\alpha}^m(|x-y|)^2 \, V(y)|dx \, dy
= C_{d,\alpha,\lambda,m}[V]_{\overline{\cR}^m_{d,\alpha}}^2.
\end{multline*}
\end{proof}

\begin{remark}
{\rm
We note that by the Kato-Rellich theorem the relativistic Schr\"odinger operators with Coulomb potentials
$|x|^{-\gamma}$, $\gamma > 0$, are self-adjoint if $\gamma < \frac{d}{2}$ for $d\leq3$ by \cite[Prop. 1.5]{AI}.
On the other hand, a combination of Theorem \ref{thm:HSKlambda} and Corollary \ref{coulombsummary} improve
this to $\gamma < \alpha$.
}
\end{remark}

Finally we note that if $d\ge 2\alpha$, then $K^m_\lambda$ cannot be a Hilbert-Schmidt operator for continuous
potentials $V$.
\begin{proposition}
Let $\alpha \in (0,2)$, $d \ge 1$, $\lambda>0$, $m \ge 0$ and assume $d \ge 2\alpha$. If there exist
$\varepsilon^\ast,r^\ast>0$ and $x^\ast \in \R^d$ such that $|V(y)|>\varepsilon^\ast$ for all $y \in B_{r^\ast}(x^\ast)$,
then $K_\lambda^m$ is not a Hilbert-Schmidt operator. In particular, if $V$ is continuous, then $K_\lambda^m$ is not a
Hilbert-Schmidt operator, unless $V \equiv 0$.
\end{proposition}
\begin{proof}
By Lemma \ref{lem:lowerboundRml} we see that the function $r \in [0,\infty) \mapsto R_\lambda^m(r\mathbf{e}_1)^2 \in
[0,\infty)$ satisfies \eqref{eq:nonintegrability}, since clearly $H^m_{d,\alpha}$ does not belong to $L^2(B_\delta)$
for any ball $B_\delta$ centred in $0$. The statement follows then by Theorem \ref{thm:nonL1kernel}.
\end{proof}

\subsection{Spectral properties}
\subsubsection{Essential spectrum}
In this section we discuss the spectrum of non-local Schr\"odinger operators with fractional Rollnik-class potentials.
In the classical case, for $V \in \mathcal{R}$  Fourier transform gives that ${\Spec}_{\rm ess}(-\Delta + V)=
{\Spec}_{\rm ess}(-\Delta)$ without any further requirements \cite{Si1}. However, if $(d,\alpha)\not=(3,2)$, it
appears to be difficult to construct a Fourier theory of $\mathcal{R}_{d,\alpha}$. Therefore we need other ways to
prove the stability of the essential spectrum.

\begin{theorem}
\label{the12}
Let $\alpha\in(0,2)$, $d\ge1$, $m\ge0$, and assume that $d<2\alpha$. If the potential
$V\in\overline{\mathcal{R}}^m_{d,\alpha}$, then
\begin{equation*}
{\Spec}_{\rm ess}H={\Spec}_{\rm ess}L_{m,\alpha}=[0,\infty)
\end{equation*}
holds for $m\geq0$.
\end{theorem}

\begin{proof}
Due to \eqref{the11_1} we may choose $\lambda\gg1$ such that
\begin{equation}
\|G_{m,\lambda}\|_{\mathscr{B}(L^2)}<1,
\end{equation}
where we wrote
\begin{equation*}
G_{m,\lambda}=(\lambda+L_{m,\alpha})^{-1/2}V(\lambda+L_{m,\alpha})^{-1/2}.
\end{equation*}
We first prove the following Tiktopoulos-type formula (\cite[Lemmas II.10 and II.11]{Si1}),
\begin{equation}
(\lambda+H)^{-1}=(\lambda+L_{m,\alpha})^{-1/2}(1+G_{m,\lambda})^{-1}(\lambda+L_{m,\alpha})^{-1/2}.\label{the12_1}
\end{equation}
Since $H$ is defined as the self-adjoint operator in the form sense on $\Dom(L_{m,\alpha}^{1/2})$, we have to consider the duality argument. We regard the operator $H$ and $V$ as the map from $\Dom(L_{m,\alpha}^{1/2})$ to its dual. Indeed
, we can define $H: \Dom(L_{m,\alpha}^{1/2})\rightarrow \Dom(L_{m,\alpha}^{1/2})^*$ by $\phi\mapsto f_{H\phi}$ with $f_{H\phi}(\psi)=(H\phi,\psi)$ for $\phi,\psi\in\Dom(L_{m,\alpha}^{1/2})$. The resolvent is also defined as the map from $\Dom(L_{m,\alpha}^{1/2})^*$ to $\Dom(L_{m,\alpha}^{1/2})$ in the following sense. For any $f\in\Dom(L_{m,\alpha}^{1/2})^*$, there exists a unique $\tilde{\phi}\in\Dom(L_{m,\alpha}^{1/2})$ 
such that $f=f_{\tilde{\phi}}=(\cdot,\tilde{\phi})$ by the Riesz representation theorem. Putting $\phi=(\lambda+L_{m,\alpha})^{1/2}\tilde{\phi}$, we define $(\lambda+L_{m,\alpha})^{-1}: \Dom(L_{m,\alpha}^{1/2})^*\rightarrow\Dom(L_{m,\alpha}^{1/2})$ by $f\mapsto(\lambda+L_{m,\alpha})^{-1}\phi$. By these definitions, we have
\begin{gather}
(\lambda+L_{m,\alpha})^{-1}(\lambda+H)(\lambda+L_{m,\alpha})^{-1/2}\sum_{j=0}^N(-1)^jG_{m,\lambda}^j(\lambda+L_{m,\alpha})^{-1/2}\nonumber\\
=(\lambda+L_{m,\alpha})^{-1}+(-1)^N(\lambda+L_{m,\alpha})^{-1/2}G_{m,\lambda}^{N+1}(\lambda+L_{m,\alpha})^{-1/2}\label{the12_2}
\end{gather}
as the map from $L^2(\R^d)$ to $\Dom(L_{m,\alpha}^{1/2})\subset L^2(\R^d)$ for $N\in\N$. Denoting the right-hand-sided of \eqref{the12_1} by $\tilde{R}_{m,\lambda}$, and taking the limit $N\rightarrow\infty$ in \eqref{the12_2}, we have
\begin{gather*}
(\lambda+L_{m,\alpha})^{-1}(\lambda+H)\tilde{R}_{m,\lambda}=(\lambda+L_{m,\alpha})^{-1}
\end{gather*}
in the operator norm sense. For $\lambda\gg1$, $(\lambda+H)^{-1}$ exist by \eqref{the11_1}. Once we write
\begin{gather*}
(\lambda+L_{m,\alpha})^{-1}(\lambda+H)\tilde{R}_{m,\lambda}=(\lambda+L_{m,\alpha})^{-1}(\lambda+H)(\lambda+H)^{-1},
\end{gather*}
 we obtain \eqref{the12_1} by injectivity of $\lambda+H$ and $(\lambda+L_{m,\alpha})^{-1}$.
We next define
\begin{equation*}
\tilde{K}_{m,\lambda}=\tilde{V}^{1/2}(\lambda+L_{m,\alpha})^{-1}|V|^{1/2}.
\end{equation*}
The following strategy is based on \cite[Theorem II.34]{Si1}. As in Lemma 4.2, we have $\|\tilde{K}_{m,\lambda}\|_{\mathscr{B}(L^2)}<1$ for $\lambda\gg1$. By the duality argument again, we have
\begin{gather}
(\lambda+L_{m,\alpha})^{-1}|V|^{1/2}\sum_{j=0}^N(-1)^{j+1}\tilde{K}_{m,\lambda}^j\tilde{V}^{1/2}(\lambda+L_{m,\alpha})^{-1}\nonumber\\
=(\lambda+L_{m,\alpha})^{-1/2}\sum_{j=0}^{N+1}(-1)^jG_{m,\lambda}^j(\lambda+L_{m,\alpha})^{-1/2}-(\lambda+L_{m,\alpha})^{-1}\label{the12_3}
\end{gather}
for $N\in\N$. Taking the limit $N\rightarrow\infty$ in \eqref{the12_3}, we have
\begin{equation*}
-(\lambda+L_{m,\alpha})^{-1}|V|^{1/2}(1+\tilde{K}_{m,\lambda})^{-1}\tilde{V}^{1/2}(\lambda+L_{m,\alpha})^{-1}=(\lambda+H)^{-1}-(\lambda+L_{m,\alpha})^{-1}
\end{equation*}
in the operator-norm sense. This says that $(\lambda+H)^{-1}-(\lambda+L_{m,\alpha})^{-1}$ is compact because $\tilde{K}_{m,\lambda}$ is compact and the operator-norm limit of a sequence of compact operators is compact. The Weyl theorem then completes the proof.
\end{proof}

\subsubsection{Discrete spectrum}
Next we consider the discrete spectrum. First we provide a bound on the number of negative eigenvalues, which in
order to keep explicit, we only consider the case $\alpha<d$ and $V \in \mathcal{R}_{d,\alpha}$.

\begin{theorem}
\label{the13}
Let $\alpha\in(0,2)$, $d\ge 1$, $m\ge0$ and assume that $\alpha<d<2\alpha$. Let the potential
$V\in\mathcal{R}_{d,\alpha}$. Denote by $V_-\geq0$ its negative part and by $N_{\lambda_0}(V)$ the number of the eigenvalues of $H$ that are less than $-\lambda_0$ with $\lambda_0>0$. Then
\begin{equation}
N_{\lambda_0}(V)\leq C_{d,\alpha}^2\left(1+\frac{me^{-\lambda_0/m}}{\lambda_0}\right)^2\|V_-\|_{\mathcal{R}_{d,\alpha}}^2\label{the13_1}
\end{equation}
holds (the quantity in the brackets is set to $1$ and is independent of $\lambda_0$ if $m=0$), where $C_{d,\alpha}>0$ is same constant as in Lemma 6.10. In particular, if $m=0$ and we denote the number of the negative eigenvalues of $H$ by $N(V)$, then
\begin{equation}
N(V)\leq C_{d,\alpha}^2\|V_-\|_{\mathcal{R}_{d,\alpha}}^2\label{the13_2}
\end{equation}
holds, that is, the discrete spectrum of $H$ is finite for the massless case.
\end{theorem}

\begin{proof}
The argument is based on the Birman-Schwinger principle, for details see \cite[Section 4.9.7]{LHB} and \cite[Theorem XIII.10]{RS4}. By the limit \eqref{the11_1}, we find that $H$ is bounded below. To apply the min-max principle \cite[Theorem XIII.1]{RS4}, for $n\in\N$, we define
\begin{gather*}
\mu_n(H)=\sup_{\phi_1,\ldots,\phi_{n-1}\in L^2({\R^d})}\inf_{\substack{\phi\in\Dom{H},\|\phi\|=1\\ \phi\in\langle\phi_1,\ldots,\phi_{n-1}\rangle^\perp}}(H\phi,\phi)
\end{gather*}
where $\langle\phi_1,\ldots,\phi_{n-1}\rangle^\perp$ means $\{\psi\in L^2(\R^d):(\psi,\phi_j)=0\ \mbox{for}\ 0\leq j\leq n-1\}$ if $n\geq2$ and $L^2(\R^d)$ if $n=1$. By the min-max principle, we know that either $\mu_n(H)$ denotes the $n$th negative eigenvalue of $H$ counting multiplicity or it coincides with $\inf\Spec_{\rm ess}(H)=0$. Without loss of generality, we can assume that $V \le 0$. Indeed $L_{m,\alpha}-V_- \leq H$, thus the number of negative eigenvalues of $H$ is not large than the one of $L_{m,\alpha}-V_-$.

Let $H(\gamma)=L_{m,\alpha}+\gamma V$ with $\gamma>0$. By Theorem \ref{the12}, we also have ${\Spec}_{\rm ess}(H(\gamma))=[0,\infty)$. If $-\lambda\in{\Spec}_{\rm p}(H(\gamma))$, there exists an eigenfunction $\phi\in\Dom(L_{m,\alpha}^{1/2})$ such that $H(\gamma)\phi=-\lambda\phi$. Passing through the duality argument, we have
\begin{equation*}
(\lambda+L_{m,\alpha})^{-1}(-V)\phi=\frac{1}{\gamma}\phi.
\end{equation*}
This implies that $1/\gamma\in{\Spec}_{\rm p}(K_\lambda^m)$ and $(-V)^{1/2}\phi$ is the corresponding eigenfunction. For $n\in\N$, We write $\mu_n(\gamma)$=$\mu_n(H(\gamma))$ for simplicity. For $\lambda_0>0$, we can write
\begin{equation*}
N_{\lambda_0}(V)=\#\left\{n\in\mathbb{N}: \mu_n(1)<-\lambda_0\right\}.
\end{equation*}
Note that $\mu_n(\gamma)$ is monotone decreasing and continuous in $\gamma$ by \cite[Lemma XIII.3.C]{RS4}. For every $1\leq n\leq N_{\lambda_0}(V)$, the relations $\mu_n(0)=0$ and $\mu_n(1)<-\lambda_0$ imply that there exists a $0<\gamma_n<1$ such that $\mu_n(\gamma_n)=-\lambda_0$. We thus compute
\begin{equation}
N_{\lambda_0}(V)=\#\left\{n\in\mathbb{N}: \mu_n(\gamma_n)=-\lambda_0\right\}\leq\sum_{n=1}^{N_{\lambda_0}(V)}\frac{1}{\gamma_n^2}\leq\sum_{\nu\in{\Spec}_{\rm p}(K_{\lambda_0})}\nu^2=\|K^m_{\lambda_0}\|_{\rm HS}^2.
\label{the13_3}
\end{equation}
By Theorem \ref{prop:uniformHS}
and \eqref{the13_3} we obtain \eqref{the13_1}. If $m=0$, we have \eqref{the13_2} from \eqref{the13_1} with $m=0$ and $N(V)=\sup_{\lambda_0>0}N_{\lambda_0}(V)$.
\end{proof}

By the same proof, with the only difference being the employement of Item $(1)$ in Theorem \ref{thm:HSKlambda} in place of \ref{prop:uniformHS}, we can show the following result for extended fractional Rollnik class.
\begin{theorem}
\label{the13extended}
Let $\alpha \in (0,2)$, $d \ge 1$, $m \ge 0$ and assume that $d<2\alpha$. Let the potential
$V\in\overline{\mathcal{R}}_{d,\alpha}$. Denote by $V_-$
its negative part and by $N_{\lambda_0}(V)$ the number of the eigenvalues of $H$ that are less than $-\lambda_0$ for $m\geq0$. Then 
\begin{equation*}
N_{\lambda_0}(V)\leq C_{d,\alpha,\lambda_0,m}^2[V_-]_{\overline{\mathcal{R}}_{d,\alpha}}^2
\label{the13_2ext}
\end{equation*}
holds, where $C_{d,\alpha,\lambda_0,m}>0$ is the same constant as in \eqref{forrollext1}.

\end{theorem}
Notice that the constant $C_{d,\alpha,\lambda,m}$ in \eqref{forrollext1}
diverges as $\lambda \to 0$ even if $m=0$ by \eqref{eq:asymptoticsC}. Hence, we are not able to provide an exact estimate on the number of negative eigenvalues. Nevertheless, Theorem \ref{the13extended} provides a lower bound on the discrete spectrum: if $C_{d,\alpha,\lambda_0,m}^2[V_-]_{\overline{\mathcal{R}}_{d,\alpha}}^2<1$ then $N_{\lambda_0}(V)=0$. This implies that the discrete spectrum should be contained in the interval $[-\overline{\lambda}(V),0]$, where 
\begin{equation*}
\overline{\lambda}(V)=\inf\left\{\lambda_0>0: \ C_{d,\alpha,\lambda_0,m}^2[V_-]_{\overline{\mathcal{R}}_{d,\alpha}}^2 \ge 1\right\}.
\end{equation*}
In case $m>0$, we cannot directly provide a bound on the number of negative eigenvalues. This is due to the fact that we cannot guarantee that $K_{\lambda}^m$ is a Hilbert-Schmidt operator for any $\lambda>0$. Nevertheless, the previous results allowed us to find a lower bound on the discrete spectrum, provided that either $V \in \cR_{d,\alpha}$ or $V \in \overline{\cR}_{d,\alpha}$. In case $V \in \overline{\cR}_{d,\alpha}^m$, we can still provide a bound on the number of eigenvalues that are less then $-m$.
\begin{theorem}
\label{the13extended2}
Let $\alpha \in (0,2)$, $d \ge 1$, $m > 0$ and assume that $d<2\alpha$. Let $V\in \overline{\mathcal{R}}_{d,\alpha}^m$.
Denote by $V_-$ its negative part and by $N_m(V)$ the number of
eigenvalues of $H$, counting multiplicity, that are less then $-m$. Then there exists $\lambda_0>m$ such that
\begin{equation*}
N_m(V)\leq C_{d,\alpha,\lambda_0,m}^2[V_-]_{\overline{\mathcal{R}}^m_{d,\alpha}}^2
\label{the13_2ext3}
\end{equation*}
holds, where $C_{d,\alpha,\lambda_0,m}>0$ is the same constant as in \eqref{eq:equivrollext2}.
\end{theorem}
\begin{proof}
Arguing exactly as in Theorem \ref{the13}, we use the min-max principle to characterize the negative eigenvalues $\mu_n(H)$ of $H$. Furthermore, the fact that $L_{m,\alpha}-V_- \le H$ still guarantees that the number of eigenvalues of $H$ that are less then $-m$ do not increase if we use $V_-$ in place of $V$. Hence, without loss of generality, we may assume that $V \le 0$. Define 
$H(\gamma)=
L_{m,\alpha}+\gamma V$ with $\gamma>0$, so that, as shown in the proof of Theorem \ref{the13}, if $-\lambda \in {\rm Spec}_p(H(\gamma))$ then $1/\gamma \in \Spec_p
(K_\lambda)$. Since $\gamma V \in \overline{\cR}^m_{d,\alpha}$, by Theorem \ref{the12} we have $\Spec_{\rm ess}
(H(\gamma))=[0,\infty)$. Now denote, for simplicity, $\mu_n(\gamma)=:\mu_n(H(\gamma))$ for $n \in \N$ and notice that these are monotone decreasing and continuous functions of $\gamma$. Now, using the notation in Theorem \ref{the13}, if $N_m(V)=0$ the statement clearly follows.
In the contrary case $N_m(V)>0$ and notice that
\begin{equation*}
N_m(V):=\max\{n \ge 1: \ \mu_n(1)<-m\},
\end{equation*}
since $\mu_n(1)$ is a non-decreasing sequence. Clearly, $N_m(V)<\infty$. Indeed, if this is not the case, then $\mu^\star=\lim_{n \to +\infty}\mu_n(1)$ is an accumulation point of the discrete spectrum of $H$, that is absurd since $\mu^\star \le -m<0$ and thus $\mu^\star \not \in Spec_{\rm ess}(H)$. Furthermore, notice that $\mu_{N_m(V)}(1)<-m$, hence there exists $\lambda_0>m$ such that $\mu_{N_m(V)}(1)<-\lambda_0$. In particular, for all $n \le N_m(V)$ it holds $\mu_{n}(1)<-\lambda_0$. At the same time $\mu_n(0)=0$ and $\mu_n(\gamma)$ is continuous and non-decreasing in $\gamma$, thus there exists $\gamma_n \in (0,1)$ such that $\mu_n(\gamma_n)=-\lambda_0$. Let $(\gamma_n)_{n=1,\dots, N_m(V)} \in (0,1)^{N_m(V)}$ be the sequence of such $\gamma_n$. In particular, $1/\gamma_n \in \Spec_{\rm p}(K_{\lambda_0}^m)$. On the other hand, if $n>N_m(V)$, then $\mu_n(1) \ge -m>-\lambda_0$ and thus $\mu_n(\gamma)\not = -\lambda_0$ for any $\gamma \in [0,1]$. We have then
\begin{equation*}
N_m(V)=\#\{n \in \N: \ \mu_{n+n_\star}(\gamma_{n+n_\star})=-\lambda_0\} \le \sum_{n=1}^{N_m(V)}
\frac{1}{\gamma_{n+n_\star}^2} \le \sum_{\nu \in \Spec_{\rm p}(K_{\lambda_0}^m)}\nu^2=\Norm{K_{\lambda_0}^m}{\rm HS}^2.
\end{equation*}
Since $\lambda_0>m$, by part (3) of Theorem \ref{thm:HSKlambda} it follows that $K_{\lambda_0}^m$ is a
Hilbert-Schmidt operator, and the same result then completes the proof.
\end{proof}
\begin{remark} Notice that under the assumptions of Theorem \ref{the13extended}, if we know that $0 \in  {\rm Spec}_{\rm ess}(H)$ is not an accumulation point of ${\rm Spec}_{\rm d}(H)$, then $N(V)$ is finite and there exists $\lambda_0>0$ such that
\begin{equation}\label{eq:boundonNV}
N(V) \le C^2_{d,\alpha,\lambda_0,m}[V_-]_{\mathcal{R}_{d,\alpha}}^2.	
\end{equation}
Indeed, in such a case, arguing as in Theorem \ref{the13extended2}, we have that
\begin{equation*}
	N(V)=\max\{n \ge 0: \ \mu_n(1)<0\},
\end{equation*}
where the latter is clearly a maximum since otherwise $N(V)=+\infty$ and $0$ would be an accumulation point of ${\rm Spec}_{\rm d}(H)$. Then there exists $\lambda_0>0$ such that $\mu_{N(V)}(1)<-\lambda_0$, hence $N(V)=N_{\lambda_0}(V)$ and \eqref{eq:boundonNV} follows by Theorem \ref{the13extended}. 
\end{remark}

\subsubsection{Non-existence of negative and zero eigenvalues}
It is known, see \cite{Simo}, \cite[Th. 4.308]{LHB},
that whenever the semigroup $\{e^{-tL_{m,\alpha}}: t\geq 0\}$ is recurrent, the operator $H = L_{m,\alpha}
+ V$ has at least one bound state whenever $V$ is a non-positive potential. However, when
$\{e^{-tL_{m,\alpha}}: t\geq 0\}$ is transient, a bound state exists only when the potential is strong enough
in a specific sense. By making use of the extended Birman-Schwinger principle, we can show that for too small
Rollnik-norm there exist no bound states and as a bonus we can also rule out embedded eigenvalues at zero. We
will use the following result, for a proof see \cite[Cor. 4.150, Cor. 4.151, Prop. 4.309]{LHB}.
\begin{theorem}
\label{extBS}
Let $\alpha < d$. Suppose that $V \leq 0$ is form-compact with respect to $L_{m,\alpha}$ and the operator norm
$\|K^m_0\| < 1$. Then $H$ has no non-positive eigenvalues, in particular, it has no bound states at negative
eigenvalues and zero is not an embedded eigenvalue.
\end{theorem}
In what follows, we derive suitable upper estimates on the norm of $K_0^m$ which will permit to use this observation.
We will separately consider the cases $m=0$ and $m>0$, which will require different treatment.

	
First we note that the map $\lambda \in [0,\infty) \mapsto \Norm{K^m_\lambda}{\cL(L^2(\R^d))}$ is monotone and
continuous.
\begin{lemma}
Let $m \geq 0$ and suppose that $K^m_{\lambda_0} \in \cL(L^2(\R^d))$ for some $\lambda_0>0$. Then $K^m_\lambda
\in \cL(L^2(\R^d))$ for all $\lambda \ge \lambda_0$  and $\Norm{K^m_\lambda}{\cL(L^2(\R^d))} \le
\Norm{K^m_{\lambda_0}}{\cL(L^2(\R^d))}$. Furthermore, if $\Norm{K^m_{\lambda_0}}{\cL(L^2(\R^d))}<\infty$,
then $\lim_{\lambda \downarrow \lambda_0}\Norm{K^m_{\lambda}}{\cL(L^2(\R^d))}=\Norm{K^m_{\lambda_0}}{{\mathscr B}
(L^2)}$.
\end{lemma}
\begin{proof}
Since $K_\lambda$ is positivity preserving, we have
\begin{equation}
\label{opnorm}
\Norm{K_\lambda}{\cL(L^2(\R^d))}=\sup_{\substack
{ f \in L^2(\R^d) \\ f \ge 0, \ f \not \equiv 0}}\frac{\Norm{K_\lambda f}{2}}{\Norm{f}{2}}.
\end{equation}
Indeed, choosing $C$ to be the right-hand side of \eqref{opnorm}, for all $f \in L^2(\R^d)$ with positive and negative
parts $f^+, f^-$ respectively, we get
\begin{equation*}
\Norm{K_\lambda f}{2}\le \sqrt{\Norm{K_\lambda f^+}{2}^2+\Norm{K_\lambda f^-}{2}^2} \le C \sqrt{\Norm{f^+}{2}^2
+\Norm{f^-}{2}^2}=C \Norm{f}{2}.
\end{equation*}
Once this is established, the statement clearly follows by the fact that $0 \le R_\lambda(x) \le R_{\lambda_0}(x)$.
By the monotone convergence theorem, the same relation guarantees that $\Norm{K_\lambda f}{2} \to \Norm{K_{\lambda_0} f}{2}$
for any $f \in L^2(\R^d)$ with $f \ge 0$, and so $\lim_{\lambda \to \lambda_0} \Norm{K_\lambda}{\cL(L^2(\R^d))} \le
\Norm{K_{\lambda_0}}{\cL(L^2(\R^d))}$. To prove that in fact equality holds, for every $\varepsilon>0$ consider
$f_\varepsilon \in L^2(\R^d)$ with $f_\varepsilon \ge 0$ such that
\begin{equation*}
\left|\frac{\Norm{K_{\lambda_0} f_\varepsilon}{2}}{\Norm{f_\varepsilon}{2}}-\Norm{K_{\lambda_0}}{\cL(L^2(\R^d))}\right|
<\varepsilon.
\end{equation*}
There also exists $\delta>0$ such that if $0<\lambda-\lambda_0<\delta$ then
\begin{equation*}
\left|\frac{\Norm{K_{\lambda} f_\varepsilon}{2}}{\Norm{f_\varepsilon}{2}}-\frac{\Norm{K_{\lambda_0} f_\varepsilon}{2}}
{\Norm{f_\varepsilon}{2}}\right|<\varepsilon.
\end{equation*}
Then for all $\lambda_0 < \lambda < \lambda_0+\delta$ it follows that
\begin{equation*}
\left|\frac{\Norm{K_{\lambda} f_\varepsilon}{2}}{\Norm{f_\varepsilon}{2}}-\Norm{K_{\lambda_0}}{\cL(L^2(\R^d))}\right|
<2\varepsilon,
\end{equation*}
which in turn implies
\begin{equation*}
\Norm{K_\lambda}{\cL(L^2(\R^d))} \ge \frac{\Norm{K_{\lambda} f_\varepsilon}{2}}{\Norm{f_\varepsilon}{2}}
\ge \Norm{K_{\lambda_0}} {\cL(L^2(\R^d))}-2\varepsilon.
\end{equation*}
Taking $\varepsilon \to 0$ the claim follows.
\end{proof}

Now we focus on estimating $\Norm{K_0^m}{\cL(L^2(\R^d))}$. In fact, we will provide a necessary and sufficient
condition for $K_0^m$ to be a Hilbert-Schmidt operator.
\begin{theorem}
\label{none}
Let $\alpha \in (0,2)$ with $0<\alpha<d<2\alpha$ and $m \ge 0$.
\begin{enumerate}
\item
If $m=0$, then $K_0$ is a Hilbert-Schmidt operator if and only if $V \in \cR_{d,\alpha}$. Furthermore, there exists
a constant $C_{d,\alpha}>0$ such that
\begin{equation}
\label{K00}
\Norm{K_0}{\rm HS} = C_{d,\alpha}\Norm{V}{\cR_{d,\alpha}}
\end{equation}
for all $V \in \cR_{d,\alpha}$. In particular, whenever $V\leq 0$ such that
\begin{equation}
\label{inpart0}
\Norm{V}{\cR_{d,\alpha}} < \frac{1}{C_{d,\alpha}},
\end{equation}
the operator $H = L_{0,\alpha} + V$ has no bound states at
negative eigenvalues and zero is not an embedded eigenvalue.
\item
If $m>0$, let $d=3$. Then $K_0^m$ is a Hilbert-Schmidt operator if and only if $V \in \cR_{3,\alpha}\cap \cR$ and there exist two constants $C^1_{\alpha},C^2_{\alpha,m}>0$ such that
\begin{equation}
\label{K0m}
C^2_{\alpha,m}\sqrt{\Norm{V}{\cR_{3,\alpha}}^2+\Norm{V}{\cR}^2}\le \Norm{K_0^m}{\rm HS} \le C^1_{\alpha}\sqrt{\Norm{V}{\cR_{3,\alpha}}^2+m^{\frac{4}{\alpha}-2}\Norm{V}{\cR}^2}
\end{equation}
for all $V \in \cR_{3,\alpha}\cap \cR$
In particular, whenever $V\leq 0$ is such that
$$
\sqrt{\Norm{V}{\cR_{3,\alpha}}^2+m^{\frac{4}{\alpha}-2}\Norm{V}{\cR}^2} < \frac{1}{C_{\alpha}},
$$
the operator $H = L_{m,\alpha} + V$ has no bound states at
negative eigenvalues and zero is not an embedded eigenvalue.
\end{enumerate}
\end{theorem}
\begin{proof}
To prove (1) observe that by \eqref{eq:equivroll}
\begin{equation*}
\iint_{\mathbb{R}^{2d}}K_0(x-y)^2dxdy=\iint_{\mathbb{R}^{2d}}R_0(x-y)^2|V(x)||V(y)|dxdy
=C_{d,\alpha}^2\|V\|_{\mathcal{R}_{d,\alpha}}^2.
\label{lem10_1bis}
\end{equation*}
\enlargethispage{1cm}
Note that \eqref{K00} and \eqref{inpart0} directly imply that $V$ is form-compact with respect to
$L_{0,\alpha}$ with relative bound strictly less than one, thus by Theorem \ref{extBS} the operator $H =
L_{0,\alpha} + V$ has no bound states at negative eigenvalues and zero is not an embedded eigenvalue.
To get (2) first we show that if $V \in \cR_{3,\alpha}\cap \cR$, then $K_0^m$ is a Hilbert-Schmidt
operator and that the second inequality in \eqref{K0m} holds. By Lemma \ref{lem3nol}
\begin{align*}
&\iint_{\mathbb{R}^{6}}K^m_0(x-y)^2dxdy=\iint_{\mathbb{R}^{6}}R^m_0(x-y)^2|V(x)||V(y)|dxdy\\
&\leq C_{\alpha}\left(\iint_{\mathbb{R}^{6}}|x-y|^{-(6-2\alpha)}|V(x)||V(y)|dxdy+m^{\frac{4}{\alpha}-2}
\iint_{\mathbb{R}^{6}}|x-y|^{-2}|V(x)||V(y)|dxdy\right)\\
&=C_{\alpha}(\Norm{V}{\cR_{3,\alpha}}^2+m^{\frac{4}{\alpha}-2}\Norm{V}{\cR}^2).
\end{align*}
If $V \in \cR_{3,\alpha} \cap \cR$, we restate Lemma \ref{lem:lowboundmassive} for any $\alpha \in (0,2)$,
$m>0$ and $d\ge 3$ as follows
\begin{equation*}
R_0^m(x) \ge C_{d,\alpha,m}\max\{|x|^{-(d-\alpha)},|x|^{-(d-2)}\} \ge
\frac{C_{d,\alpha,m}}{2}(|x|^{-(d-\alpha)}+|x|^{-(d-2)}).
\end{equation*}
Hence, for $d=3$ we get
\begin{align*}
&\iint_{\mathbb{R}^{6}}|K^m_0(x-y)|^2dxdy=\iint_{\mathbb{R}^{6}}R^m_0(x-y)^2|V(x)||V(y)|dxdy\\
&\ge C_{\alpha,m}\left(\iint_{\mathbb{R}^{6}}|x-y|^{-(6-2\alpha)}|V(x)||V(y)|dxdy
+\iint_{\mathbb{R}^{6}}|x-y|^{-2}|V(x)||V(y)|dxdy\right)\\
&=C_{\alpha,m}(\Norm{V}{\cR_{3,\alpha}}^2+\Norm{V}{\cR}^2).
\end{align*}
The remaining part of the claim follows in the same way as in part (1).
\end{proof}

\section{The case $\alpha = 1$}
\label{sectalpha1}
In this section we study the $\alpha \to 1$ limit and, more generally, consider $\alpha \to \frac{d}{2}$ for
$d=1,2,3$, using the spaces defined in \eqref{crunchyrollnik}. As seen in Proposition~\ref{prop:notappr},
this limit taken directly does not lead to a sensible set of potentials. Here we show, however, that the limit
of Schr\"odinger operators with Rollnik-class potentials in $\Gamma$-convergence sense behaves much better,
and as a bonus we even obtain self-adjointness of the limit operator.

Recall the L\'evy jump measures \eqref{levym}-\eqref{levy0} and let
\begin{equation*}
[u]_{m,\alpha}^2:=\frac{1}{2}\int_{\R^d}\int_{\R^d}|u(x+h)-u(x)|^2j_{m,\alpha}(|h|)dxdh
\end{equation*}
be a Gagliardo-type seminorm on $L^2(\R^d)$ and rewrite fractional Sobolev space
\begin{equation*}
H^{\alpha/2}(\R^d):= \Big\{u \in L^2(\R^d): [u]_{0,\alpha}<\infty \Big\}
\end{equation*}
equivalently in its terms. For a given potential $V$, we define the quadratic form associated to the
operator formally written as $H=L_{m,\alpha}+V$ by
\begin{equation*}
\cA_{m,\alpha}^V[u]:=[u]_{m,\alpha}+\int_{\R^d}V(x)u^2(x)dx, \quad u \in L^2(\R^d).
\end{equation*}
We denote $\mathcal{D}(\cA_{m,\alpha}^V)=H^{\alpha/2}(\R^d)\cap L^2(|V|)$ and set
$\cA_{m,\alpha}^V[u]=\infty$ if $u \not \in  \mathcal{D}(\cA_{m,\alpha}^V)$. We will show self-adjointness
of the limit Schr\"odinger operator obtained on varying $\alpha$ by means of the $\Gamma$-convergence of
the associated quadratic forms.

First we recall the concept of $\Gamma$-convergence.
\begin{definition}
\label{def:Gammalim}
Let $\mathscr M$ be a metric space, $(F_n)_{n \in \N}$ a sequence of functionals $F_n:\mathscr M \to
\overline{\R}$, and $F:\mathscr M \to \overline{\R}$ be another functional. We say that the sequence
is \emph{$\Gamma$-convergent} to $F$ as $n \to \infty$, denoted $F=\Gamma-\lim_{n \to \infty}F_n$, if
\begin{enumerate}
\item[(1)]
for all sequences $(x_n)_{n \in \N}\subset \mathscr M$ convergent to $x \in \mathscr M$ the property $F(x)
\le \liminf_{n \to \infty}F_n(x_n)$ holds
\item[(2)]
for all $x \in \mathscr M$ there exists a sequence $\seq x \to x$ in $\mathscr M$ such that $F(x) \ge
\limsup_{n \to \infty} F_n(x_n)$.
\end{enumerate}
Furthermore, if $\mathscr M=L^2(\R^d)$, we say that $\seq F$ is \emph{Mosco-convergent} to $F$ as $n \to
\infty$, denoted $F={\rm M}-\lim_{n \to \infty}F_n$, if property (2) above holds and
\begin{itemize}
\item[$(1_w)$]
for all sequences $(x_n)_{n \in \N}\subset L^2(\R^d)$ weakly convergent to $x \in L^2(\R^d)$ it follows that
$F(x) \le \liminf_{n \to \infty}F_n(x_n)$.
\end{itemize}
\end{definition}
\noindent
It is clear that if $F={\rm M}-\lim_{n \to \infty}F_n$, then also $F=\Gamma-\lim_{n \to \infty}F_n$
in the strong topology of $L^2(\R^d)$.

We will make use of the following decomposition, obtained in \cite[Lem. 2]{R02}, see also \cite[Prop. 2.7]{AL22}.
For every $m>0$, $\alpha \in (0,2)$ and $d \ge 1$ there exists a function $\Sigma_{m,\alpha}: \R^d \to \R^+$ with
the property that
\begin{equation}
\label{finite}
\int_{\R^d}\Sigma_{m,\alpha}(|x|)dx=m,
\end{equation}
such that
\begin{equation}
\label{decomp}
j_{0,\alpha}(r)=j_{m,\alpha}(r)+\Sigma_{m,\alpha}(r).
\end{equation}
This leads to the following domination property.
\begin{lemma}
Let $d \ge 1$, $m>0$ and $\alpha \in (0,2)$. Then
\begin{equation}
\label{equiv}
[u]^2_{m,\alpha} \le [u]^2_{0,\alpha} \le [u]_{m,\alpha}^2+2m\Norm{u}{2}^2
\end{equation}
for every $u \in L^2(\R^d)$. In particular, if $u \in H^{\alpha/2}(\R^d)$, then $[u]_{m,\alpha}<\infty$
for all $m>0$. Conversely, if $[u]_{m,\alpha}<\infty$ for some $m>0$, then $u \in H^{\alpha/2}(\R^d)$.
\end{lemma}
\begin{proof}
The first bound in \eqref{equiv} is a straightforward consequence of \eqref{decomp} and the fact that
$\sigma_{m,\alpha}$ is positive. For the second we have
\begin{align*}
[u]^2_{0,\alpha}&=\frac{1}{2}\int_{\R^d}\int_{\R^d}|u(x+h)-u(x)|^2j_{0,\alpha}(|h|)dxdh\\
&\le \frac{1}{2}\int_{\R^d}\int_{\R^d}|u(x+h)-u(x)|^2j_{m,\alpha}(|h|)dxdh+\frac{1}{2}\int_{\R^d}
\int_{\R^d}|u(x+h)-u(x)|^2\Sigma_{m,\alpha}(|h|)dxdh\\
&=:[u]^2_{m,\alpha}+I.
\end{align*}
Furthermore,
\begin{align*}
I &\le \int_{\R^d}\int_{\R^d}(|u(y)|^2+|u(x)|^2)\Sigma_{m,\alpha}(|x-y|)dxdy\\
&=2\int_{\R^d}\int_{\R^d}|u(y)|^2\Sigma_{m,\alpha}(|x-y|)dxdy
= 2\int_{\R^d}|u(y)|^2\left(\int_{\R^d}\sigma_{m,\alpha}(|x-y|)dx\right)dy
=2m\Norm{u}{2},
\end{align*}
where we used \eqref{finite}.
\end{proof}
Now we can prove our result on $\Gamma$-convergence yielding self-adjointness.

\begin{theorem}
\label{alpha1sa}
Let $d=1,2,3$, $m \ge 0$, and consider a potential $V \in \overline{\cR}^{m,p}_{d,\frac{d}{2}}$ for
some $p\ge 4$ such that $V_- \in L^\infty(\R^d)$. Then the Schr\"odinger operator $H=L_{m,\frac{d}{2}}+V$
is self-adjoint \textcolor{red}{in the form sense}.
\end{theorem}

\begin{proof}
By the definition of $\overline{\cR}_{d,\frac{d}{2}}^{m,p}$ there exists a sequence $(\alpha_k)_{k \in \N}
\subset \left(\frac{d}{2},2\right)$ such that $\alpha_k \downarrow \frac{d}{2}$ and $V \in
\overline{\cR}_{d,\alpha_k}^{m}$. Consider the sequence of Schr\"odinger operators $H_k:=L_{\alpha_k,m}+V$.
By Section~\ref{subs:selfadj} $H_k$ is self-adjoint for all $k \in \N$. Hence we can consider the
lower-semicontinuous (in $L^2(\R^d)$) quadratic forms $\cA_{m,\alpha_k}^V$. Define
\begin{align*}
\cA_{m,\frac{d}{2}}^V[u]:=[u]_{m,\frac{d}{2}}+\int_{\R^d}V(x)u^2(x)dx, \quad u \in L^2(\R^d).
\end{align*}
We show that $\cA_{m,\frac{d}{2}}^V=\Gamma-\lim_{k \to \infty}\cA_{m,\alpha_k}^V$. First note that
$[\cdot]_{m,\frac{d}{2}}={\rm M}-\lim_{k \to \infty}[\cdot]_{m,\alpha_k}$ by \cite[Lem. 3.2]{CS06}. We
proceed in two steps.

\vspace{0.1cm}
\noindent
\emph{Step 1.} We show that condition (1) in Definition \ref{def:Gammalim} holds. Let $u_k \to u$ in
$L^2(\R^d)$. In case that $\liminf_{k \to \infty}\cA_{m,\alpha_k}^V[u_k]=\infty$, the statement is trivial,
thus we assume that $\liminf_{k \to \infty}\cA_{m,\alpha_k}^V[u_k]<\infty$. With no loss of generality we
assume that the sequence $u_k$ satisfies
\begin{eqnarray*}
&& (1) \, \liminf_{k \to \infty}\cA_{m,\alpha_k}^V[u_k]=\lim_{k \to \infty}\cA_{m,\alpha_k}^V[u_k], \\
&& (2) \, \liminf_{k \to \infty} \, [u_k]_{m,\alpha_k}=\lim_{k \to \infty}[u_k]_{m,\alpha_k}, \\
&& (3) \, \liminf_{k \to \infty}\int_{\R^d}(V(x)+C_V)|u_k(x)|^2dx=\lim_{k \to \infty}\int_{\R^d}(V(x)+C_V)|u_k(x)|^2dx, \\
&& (4) \, \lim_{k \to \infty}u_k(x) \to u(x) \ \mbox{ for a.e. } x \in \R^d,
\end{eqnarray*}
	where $C_V:=\Norm{V_-}{L^\infty}$. Note also that
	\begin{equation*}\label{eq:CV}
		\lim_{k \to \infty}C_V\int_{\R^d}|u_k(x)|^2dx=C_V\int_{\R^d}|u(x)|^2dx.
	\end{equation*}
	Then we have
	\begin{equation}\label{eq:pre1Gamma}
	\lim_{k \to \infty}\cA_{m,\alpha_k}^V[u_k]=\lim_{k \to \infty}[u_k]_{m,\alpha_k}
+\lim_{k \to \infty}\int_{\R^d}(V(x)+C_V)|u_k(x)|^2dx-C_V\int_{\R^d}|u(x)|^2dx.
	\end{equation}
	By Mosco-convergence of $[\cdot]_{m,\alpha_k}$ to $[\cdot]_{m,\frac{d}{2}}$ it follows that
	\begin{equation*}
		\lim_{k \to \infty}[u_k]_{m,\alpha_k} \ge [u]_{m,\frac{d}{2}}.
	\end{equation*}
	Furthermore, since $(V(x)+C_V)|u_k(x)|^2 \ge 0$, Fatou's Lemma gives
	\begin{equation*}
		\lim_{k \to \infty}\int_{\R^d}(V(x)+C_V)|u_k(x)|^2dx \ge \int_{\R^d}(V(x)+C_V)|u(x)|^2dx.
	\end{equation*}
	Hence, by \eqref{eq:pre1Gamma} we get
	\begin{align*}
		\lim_{k \to \infty}\cA_{m,\alpha_k}^V[u_k]&=\lim_{k \to \infty}[u_k]_{m,\alpha_k}
+\lim_{k \to \infty}\int_{\R^d}(V(x)+C_V)|u_k(x)|^2dx-C_V\int_{\R^d}|u(x)|^2dx \\
		&\ge [u]_{m,\alpha_k}+\int_{\R^d}(V(x)+C_V)|u(x)|^2dx-C_V\int_{\R^d}|u(x)|^2dx=\cA_{m,\frac{d}{2}}^V[u].
	\end{align*}

\vspace{0.1cm}
\noindent
\emph{Step 2.}
Next we show that condition (2) in Definition \ref{def:Gammalim} also holds. Let $u \in L^2(\R^d)$.
If $u \not \in \cD(\cA_{m,\frac{d}{2}}^{V})$, then the statement is trivial with the sequence
$u_k=u$ for all $k \in \N$. If $u \in \cD(\cA_{m,\frac{d}{2}}^V)$, then consider the sequence
$(u_k)_{k \in \N}$ such that $u_k \to u$ strongly in $L^2(\R^d)$ and $[u]_{m,\frac{d}{2}}=
\lim_{k \to \infty}[u_k]_{m,\alpha_k}$. Without loss of generality, we can assume that $u_k(x)
\to u(x)$ for a.e. $x \in \R^d$ and that there exists a constant $M_1$ (possibly dependent on $m$,
$d$ and $u$) such that $\Norm{u_k}{2}+[u_k]_{m,\alpha_k} \le M_1$ for all $k \in \N$. By
\eqref{equiv} we have $[u_k]_{0,\alpha_k}^2 \le M_2$ for some constant $M_2$. Thus we obtain
\begin{align*}
[u_k]_{0,\frac{d}{2}}^2&=\frac{d2^{\frac{d}{2}-3}}{\pi^{\frac{d}{2}}}\int_{\R^d}
\int_{\R^d}|u_k(x+h)-u_k(x)|^2|h|^{-\frac{3d}{2}} dh dx\\
&=\frac{d2^{\frac{d}{2}-1}}{\pi^{\frac{d}{2}}}\int_{\R^d}
\int_{B_1}|u_k(x+h)-u_k(x)|^2|h|^{-\frac{3d}{2}} dh dx\\
&\qquad +\frac{d2^{\frac{d}{2}-1}}{\pi^{\frac{d}{2}}}\int_{\R^d}
\int_{B_1^c}|u_k(x+h)-u_k(x)|^2|h|^{-\frac{3d}{2}} dh dx=I_1^k+I_2^k.
\end{align*}
For the first integral we have
\begin{equation*}
I_1^k \le \frac{d2^{\frac{d}{2}-1}}{\pi^{\frac{d}{2}}}\int_{\R^d}\int_{B_1}|u_k(x+h)-u_k(x)|^2
|h|^{-d-\alpha_k} dh dx \le C_0[u_k]_{0,\alpha_k}^2,
\end{equation*}
where
\begin{equation*}
C_0:=\max_{\alpha \in \left[\frac{d}{2},\alpha_1\right]}\frac{d2^{\frac{d}{2}-\alpha+1}
\Gamma\left(1-\frac{\alpha}{2}\right)}{\alpha_k\Gamma\left(\frac{d+\alpha}{2}\right)}.
\end{equation*}
For the second,
\begin{equation*}
		I_2^k \le \frac{\sigma_d2^{\frac{d}{2}+1}}{\pi^{\frac{d}{2}}}\Norm{u_k}{2}^2.
	\end{equation*}
	Hence there exists a constant $C_1>0$, independent of $k$, such that
	\begin{equation*}\label{eq:embedding1}
		[u_k]_{0,\frac{d}{2}} \le C_1([u_k]_{0,\alpha_k}+\Norm{u_k}{2}) \le M_3,
	\end{equation*}
	where $M_3>0$ is a suitable constant. By the Sobolev inequality (see \cite[Th. 6.5]{DPV12})
we have
	\begin{equation*}
		\Norm{u_k}{q} \le M_4(q)
	\end{equation*}
	 for every $q \in \left[2,\frac{8}{3}\right]$, where $M_4(q)>0$ is a suitable constant.
Note that $\frac{8}{3}$ is the critical Sobolev exponent in $H^{\frac{d}{4}}(\R^d)$. This
further implies that
	 \begin{equation*}
	 	\Norm{|u_k|^2}{q} \le M_5(q)
	 \end{equation*}
 for a suitable constant $M_5(q)$, for all $q \in \left[1,\frac{4}{3}\right]$. Hence
 $\Norm{|u_k|^2}{p^\prime} \le M_5(p^\prime)$ due to $p \ge 4$. Without loss of generality, we
 can then assume that $|u_k|^2 \rightharpoonup |u|^2$ weakly in $L^{p^\prime}(\R^d)$. Then
	 \begin{equation*}
	 	\lim_{k \to \infty}\int_{\R^d}V(x)|u_k(x)|^2dx=\int_{\R^d}V(x)|u(x)|^2dx
	 \end{equation*}
	 and so
	\begin{equation*}
 		\lim_{k \to \infty}\cA_{m,\alpha_k}^V[u_k]=\lim_{k \to \infty}[u_k]_{m,\alpha_k}+\lim_{k \to \infty}\int_{\R^d}V(x)|u_k(x)|^2dx=[u]_{m,\frac{d}{2}}+\int_{\R^d}V(x)|u(x)|^2dx=\cA_{m,\frac{d}{2}}^V[u].
	\end{equation*}
	Hence $\cA^V_{m,\frac{d}{2}}=\Gamma-\lim_{k \to \infty}\cA_{m,\alpha_k}^V$. By \cite[Prop. 1.28]{B02} this
 implies that $\cA_{m,\frac{d}{2}}^V$ is lower semi-continuous. However, $\cA_{m,\frac{d}{2}}$ is the quadratic
 form associated with the operator $H$, hence by \cite[Th. 12.13]{D12} the operator $H$ is self-adjoint in the form sense.
\end{proof}

\section{Appendix: Proofs of resolvent estimates}
\noindent
In this section we prove the resolvent estimates stated in Section \ref{sec:res}. We treat the massless
($m=0$) and massive ($m>0$) cases separately.
\subsection{Massless case}
\subsubsection{Proof of estimate \eqref{forroll} in the $m=0$ case}
\begin{proposition}\label{lem1}
Let $\alpha \in (0,2)$ and $d \ge 1$. There exists $C_{d,\alpha} > 0$ such that for all $\lambda>0$ and
$x \in \R^d \setminus \{0\}$
\begin{equation*}
\label{resolvest}
R_\lambda(x) \; \leq \; C_{d,\alpha}\begin{cases} |x|^{\alpha-d}, & \alpha<d, \\[10pt]
\log\left(1+\frac{1}{\lambda|x|}\right)+1, & \alpha=d=1,\\[10pt]
|x|^{\alpha-1}\wedge1+\lambda^{\frac{1-\alpha}{\alpha}}, & \alpha>d=1.
\end{cases}
\end{equation*}
Furthermore, if $\alpha<d$, then the statement holds also for $\lambda=0$.
\end{proposition}

\begin{proof}
From (2.11) it follows that there exists a constant $c_1> 0$ such that
\begin{equation}
\label{detailed}
p_t(x) \leq c_1 t|x|^{-d-\alpha} \;\; \mbox{if $t<|x|^\alpha$}
\quad \mbox{and} \quad
p_t(x) \leq c_1 t^{-d/\alpha} \;\; \mbox{if $t>|x|^\alpha$}.
\end{equation}
We write
\begin{equation*}
R_\lambda(x) = \left( \int_0^{|x|^\alpha} + \int_{|x|^\alpha}^\infty \right) e^{-\lambda t}p_t(x)dt
= I_1+I_2.
\end{equation*}
We first estimate $I_1$ as
\begin{equation*}
I_1 \leq c_1|x|^{-d-\alpha} \int_0^{|x|^\alpha} te^{-\lambda t}dt
\leq c_1|x|^{-d-\alpha} \int_0^{|x|^\alpha} tdt = \frac{c_1}{2}|x|^{-d+\alpha}
\end{equation*}
if $\alpha\leq d$. If $\alpha>d=1$ and $|x|\geq1$, by $(\lambda t)^{1-1/\alpha}e^{-\lambda t}\leq(1-1/\alpha)^{1-1/\alpha}e^{-1+1/\alpha}$, we have
\begin{equation*}
I_1 \leq c_1\left(\frac{\alpha-1}{\alpha e}\right)^{\frac{\alpha-1}{\alpha}}|x|^{-1-\alpha}\lambda^{\frac{1-\alpha}{\alpha}}\int_0^{|x|^\alpha} t^{\frac{1}{\alpha}}dt=c_1\left(\frac{\alpha-1}{\alpha e}\right)^{\frac{\alpha-1}{\alpha}}\frac{\alpha}{\alpha+1}\lambda^{\frac{1-\alpha}{\alpha}}.
\end{equation*}
In estimating $I_2$ we have three cases to consider. If $\alpha<d$, we have
$$
I_2 \leq c_1 \int_{|x|^\alpha}^\infty  t^{-d/\alpha} e^{-\lambda t} dt
\leq c_1 \int_{|x|^\alpha}^\infty t^{-d/\alpha} dt = \frac{c_1}{\frac{d}{\alpha}-1}|x|^{-d+\alpha}.
$$
Note that the same estimate holds for $\lambda=0$ in this case. If $d=1$ and $\alpha>d$, we get
\begin{equation*}
\label{eq:I2adeq}
I_2 \leq c_1 \int_{|x|^\alpha}^\infty  t^{-1/\alpha} e^{-\lambda t} dt
\leq c_1 \lambda^{\frac{1-\alpha}{\alpha}}\int_{0}^\infty t^{-1/\alpha}e^{-t} dt
= c_1 \lambda^{\frac{1-\alpha}{\alpha}}\Gamma\Big(1-\frac{1}{\alpha}\Big).
\end{equation*}
In case $\alpha=d=1$ we use the exponential integral function
\begin{equation*}
E_1(x)=\int_x^{\infty}t^{-1}e^{-t}dt, \ x>0
\end{equation*}
and recall the following bracketing property (see \cite[Ch. 5, eq. 5.1.20]{AS})
\begin{equation}
\label{eq:bracket}
\frac{e^{-x}}{2}\log\Big(1+\frac{2}{x}\Big) \le E_1(x) \le e^{-x}\log\Big(1+\frac{1}{x}\Big).
\end{equation}
Hence we obtain
$$
I_2 \leq c_1 \int_{|x|}^\infty  t^{-1} e^{-\lambda t} dt
\leq c_1 \int_{\lambda|x|}^\infty t^{-1}e^{-t} dt
= c_1E_1(\lambda|x|)
\le c_1\log\Big(1+\frac{1}{\lambda|x|}\Big).
$$
\end{proof}

\subsubsection{Proof of estimate \eqref{eq:equivrollext1}}
\noindent
Below we will make use of the inequality
\begin{equation}
\label{eq:ineqexpon}
s^{p}e^{-s} \le p^pe^{-p},
\end{equation}
for $s,p \ge 0$. Beside Proposition \ref{lem1}, we can also prove a further estimate that is sharper for large values of $|x|$.
\begin{lemma}
\label{lem2}
Let $\alpha \in (0,2)$ and $d \ge 1$. Then there exists a positive constant $C_{d,\alpha}$ such that
for all $\lambda>0$ and $x \in \R^d \setminus \{0\}$,
\begin{equation*}
R_\lambda(x)\leq\frac{C_{d,\alpha}}{\lambda^2|x|^{d+\alpha}}.
\end{equation*}
\end{lemma}
\begin{proof}
We give another estimate of $I_1$ and $I_2$ in the proof of Lemma \ref{lem1}. Concerning $I_1$, we have
		\begin{equation*}
			I_1 \le c_1|x|^{-d-\alpha}\int_0^{|x|^\alpha}te^{-\lambda t}dt=c_1|x|^{-d-\alpha}\left(-|x|^\alpha\frac{e^{-\lambda|x|^\alpha}}{\lambda}
			-\frac{e^{-\lambda|x|^\alpha}}{\lambda^2}+\frac{1}{\lambda^2}\right) \le \frac{c_1}{\lambda^2|x|^{d+\alpha}}\label{lem2_1class}
		\end{equation*}
		Next, observe that
		\begin{align*}
			I_2 \le c_1 \int_{|x|^\alpha}^\infty  t^{-d/\alpha}e^{-\lambda t}dt=c_2 \lambda^{\frac{d}{\alpha}-1}\int_{\lambda|x|^\alpha}^\infty  s^{-d/\alpha}e^{-s}ds.
		\end{align*}
		By \eqref{eq:ineqexpon}, we clearly have
		\begin{align*}
			\begin{split}
			I_2 &\le \frac{2 c_1}{e} \lambda^{\frac{d}{\alpha}-1}\int_{\lambda|x|^\alpha}^\infty  s^{-d/\alpha-1}e^{-\frac{s}{2}}ds\le \frac{2 c_1}{e}e^{-\frac{\lambda |x|^{\alpha}}{2}}\lambda^{\frac{d}{\alpha}-1}\int_{\lambda|x|^\alpha}^\infty  s^{-d/\alpha-1}ds\\
			&=\frac{2 \alpha c_1}{d e}e^{-\frac{\lambda |x|^{\alpha}}{2}}\lambda^{-1}|x|^{-d}\le \frac{4\alpha c_1}{d\lambda^2 e^2|x|^{d+\alpha}}.	
			\end{split}
		\end{align*}
\end{proof}
Combining Proposition \ref{lem1} and Lemma \ref{lem2} we get  the following non-uniform upper bound.
\begin{proposition}
\label{cor:comparison1}
Let $\alpha \in (0,2)$ and $d \ge 1$. Then for all $\lambda>0$ there exists a constant $C_{d,\alpha,\lambda}>0$
such that
\begin{equation*}
R_\lambda(x) \le C_{d,\alpha,\lambda}H_{d,\alpha}^0(|x|),
\end{equation*}
for all $x \in \R^d \setminus \{0\}$, where $H_{d,\alpha}^0$ is defined in \eqref{eq:equivrollext1}.
\end{proposition}
Proposition~\ref{cor:comparison1} corresponds to the upper bound in \eqref{eq:equivrollext1}. We only have to prove the lower bound.
\begin{proposition}
\label{lem:lowermasslesssmall}
Let $\alpha \in (0,2)$ and $d \ge 1$. Then for all $\lambda>0$ there exists $C_{d,\alpha,\lambda}>0$ such that for all $x \in
\R^d\setminus \{0\}$
\begin{equation*}
R_\lambda(x) \ge C_{d,\alpha,\lambda}H_{d,\alpha}^0(|x|),
\end{equation*}
where $H_{d,\alpha}^0$ is defined in \eqref{eq:Hda0}.
\end{proposition}
\begin{proof}
First we give an estimate for $|x|>1$. By \eqref{heatker} there exists a constant $c_2$ such that
\begin{equation}
	\label{detailed2}
	p_t(x) \geq c_2 t|x|^{-d-\alpha} \;\; \mbox{if $t<|x|^\alpha$}
	\quad \mbox{and} \quad
	p_t(x) \geq c_2 t^{-d/\alpha} \;\; \mbox{if $t>|x|^\alpha$}.
\end{equation}
Using \eqref{detailed2} we obtain
\begin{align*}
	R_\lambda(x) &\ge \int_0^{|x|^\alpha}e^{-\lambda t}p_t(x)dt \ge c_2|x|^{-d-\alpha}\int_0^{|x|^\alpha}te^{-\lambda t}dt \\
	&\ge c_2|x|^{-d-\alpha}\int_0^{|x|^\alpha}te^{-\lambda t}dt\ge  \frac{c_2}{e^{\lambda}\lambda^2}|x|^{-d-\alpha}\left(e^{\lambda}-1-\lambda\right).
\end{align*}
Now let $|x| \le 1$ and assume $\alpha<d$. Then we have, again by \eqref{detailed},
\begin{align*}
	R_\lambda(x) &\ge c_2|x|^{-d-\alpha}\int_0^{|x|^\alpha}te^{-\lambda t}dt \ge c_2|x|^{-d-\alpha}e^{-\lambda}\int_0^{|x|^\alpha}tdt=\frac{c_2}{2}e^{-\lambda}|x|^{-(d-\alpha)}.
\end{align*}
If $\alpha=d=1$ we get instead
\begin{align*}
	R_\lambda(x) &\ge c_2\int_{|x|}^{\infty}t^{-1}e^{-\lambda t}=c_2E_1(\lambda|x|) \ge \frac{c_2}{2}e^{-\lambda}\log\left(1+\frac{2}{\lambda|x|}\right) \ge C_{1,1,\lambda}\log\left(1+\frac{1}{|x|}\right),
\end{align*}
where we also used \eqref{eq:bracket}. Finally, for $\alpha>d=1$, notice that
\begin{equation*}
	R_\lambda(x) \ge c_2 \int_{|x|^\alpha}^{\infty}t^{-\frac{1}{\alpha}}e^{-\lambda t}dt \ge c_2 \int_{1}^{\infty}t^{-\frac{1}{\alpha}}e^{-\lambda t}dt.
\end{equation*}
\end{proof}

\subsection{Massive case}
\subsubsection{The function $\kappa_m$}
\noindent
Next we turn to the massive case. First we state some basic properties of the function $\kappa_m$ defined in
\eqref{eq:kappam}.
\begin{lemma}
Let $\alpha\in(0,2)$ and $d\geq1$.
\begin{enumerate}
\item
We have
\begin{equation}
\label{masslesskappa}
\lim_{m \downarrow 0}\kappa_m(r)=2^{\frac{d+\alpha}{2}-1}\Gamma\left(\frac{d+\alpha}{2}\right)r^{-\frac{d+\alpha}{2}}.
\end{equation}
\item
There exists a positive constant $C_{1,d,\alpha}$ such that
\begin{equation}\label{cor4_1}
\kappa_m(r)\leq C_{1,d,\alpha} r^{-\frac{d+\alpha}{2}}.
\end{equation}
\item
There exists a positive constant $C_{2,d,\alpha}$ such that
\begin{equation}\label{cor4_2}
\kappa_m(r)\geq C_{2,d,\alpha}r^{-\frac{d+\alpha}{2}}
\end{equation}
for $r\le1$.
\end{enumerate}
\end{lemma}

\begin{proof}
Part (1) directly follows by the second asymptotic relation in \eqref{besselasymp}. Consider part (2) and (3). We note that the modified Bessel function of the second kind $K_{\frac{d+\alpha}{2}}(r)$ is positive, bounded and does not have any zero points for $r>0$. Using \eqref{besselasymp} again, we have
\begin{gather}
\lim_{r\uparrow\infty}\sqrt{r}e^{r}K_{\frac{d+\alpha}{2}}(r)=\sqrt{\frac{\pi}{2}},\label{kappaliminf}\\
\lim_{r\downarrow0}r^{\frac{d+\alpha}{2}}K_{\frac{d+\alpha}{2}}(r)=2^{\frac{d+\alpha}{2}-1}\Gamma\left(\frac{d+\alpha}{2}\right).\label{kappalimzero}
\end{gather}
Hence \eqref{cor4_1} and \eqref{cor4_2} hold for $r\leq1$ by \eqref{kappalimzero}. If $r\geq1$,
\begin{equation*}
c_{1,d,\alpha}:=\sup_{r\geq1}\sqrt{r}e^{r}K_{\frac{d+\alpha}{2}}(r)
\end{equation*}
is finite by \eqref{kappaliminf}. Thus we have
\begin{gather*}
\kappa_m(r)\leq c_{1,d,\alpha}m^{\frac{d+\alpha-1}{2\alpha}}r^{-\frac{1}{2}}e^{-m^{\frac{1}{\alpha}}r}=c_{1,d,\alpha}r^{-\frac{d+\alpha}{2}}\left(m^{\frac{1}{\alpha}}r\right)^{\frac{d+\alpha-1}{2}}e^{-m^{\frac{1}{\alpha}}r}\\
\leq  c_{1,d,\alpha}r^{-\frac{d+\alpha}{2}}\left(\frac{d+\alpha+1}{2}\right)^{\frac{d+\alpha+1}{2}}e^{-\frac{d+\alpha+1}{2}}.
\end{gather*}
\end{proof}

Since the function
\begin{equation*}
\label{eq:Phim}
\Phi^m_{d,\alpha}(r) = \frac{r^{\alpha/2}}{\kappa_m(r)^{\alpha/(d+\alpha)}}, \quad r>0
\end{equation*}
will play a prominent role in what follows, we show some of its elementary properties.
\begin{lemma}
\label{lem:Phim}
The following hold:
\begin{enumerate}
\item
For $r \downarrow 0$ we have
\begin{equation*}
\Phi^m_{d,\alpha}(r)\sim \frac{2^{\frac{\alpha}{d+\alpha}-\frac{\alpha}{2}}}{\Gamma^{\frac{\alpha}{d+\alpha}}
\left(\frac{d+\alpha}{2}\right)}r^{\alpha}.
\end{equation*}
As a consequence, for all $R>0$ there exists a constant $C_{d,\alpha,m,R}$ such that for all $r \in (0,R]$
\begin{equation}\label{eq:controlPhi}
\Phi^m_{d,\alpha}(r) \le C_{d,\alpha,m,R}r^{\alpha}.
\end{equation}
\item
As $r \uparrow \infty$ we have
\begin{equation*}
\Phi^m_{d,\alpha}(r) \sim \frac{2^{\frac{\alpha}{2(d+\alpha)}}}{m^{\frac{1}{2}-\frac{1}{2(d+\alpha)}}
\pi^{\frac{\alpha}{2(d+\alpha)}}} r^{\frac{\alpha}{2}+ \frac{\alpha}{2(d+\alpha)}}
e^{\frac{\alpha}{d+\alpha}m^{1/\alpha}r}.
\end{equation*}
\item
The function $\Phi^m_{d,\alpha}(r)$ is strictly increasing.
\item
For every fixed $r > 0$ we have
\begin{equation*}
\lim_{m \downarrow 0}\Phi^m_{d,\alpha}(r) = \frac{2^{1-\frac{\alpha}{2}}}{\Gamma\Big(\frac{d+\alpha}{2}\Big)^{d/(d+\alpha)}}
r^{\alpha}.
\end{equation*}
\item
There exists a constant $C_{d,\alpha}$ such that for all $r>0$
\begin{equation}
\label{eq:controlPhi2}
\Phi^m_{d,\alpha}(r) \ge C_{d,\alpha}r^\alpha.
\end{equation}
\end{enumerate}
\end{lemma}
\begin{proof}
Using \eqref{besselasymp}, properties (1)-(2) are straightforward. (3) is also immediate since $\Phi^m_{d,\alpha}(r)$ is
the product of the two strictly increasing positive functions $r \in (0,\infty) \mapsto m^{-1/2}r^{\alpha/2}$ and $r \in
(0,\infty) \mapsto \big(K_{\frac{d+\alpha}{2}}(m^{1/\alpha}r)\big)^{-\alpha/(d+\alpha)}$. Finally, (4) and (5) are
consequences of \eqref{masslesskappa} and \eqref{cor4_1}, respectively.
\end{proof}

\subsubsection{Proof of estimate  \eqref{forroll} in the $m>0$ case}
\begin{proposition}
	\label{lem3nol}
	Let $\alpha \in (0,2)$, $d \ge 3$. There exists a constant $C_{d,\alpha}>0$ such that for all $m>0$, $\lambda \ge 0$
and $x \in \R^d \setminus \{0\}$,
	\begin{equation*}
		R^m_\lambda(x)\leq C_{d,\alpha}\left(\frac{1}{|x|^{d-\alpha}}+\frac{m^{\frac{2}{\alpha}-1}}{|x|^{d-2}}\right).
\label{lem3_1nol}
	\end{equation*}
\end{proposition}
\begin{proof}
While in the massless case we split $R_\lambda(x)$ into two integrals, here we have to deal with four different
integrals depending on the range of $|x|$. Indeed, by \eqref{massive_heatker} it follows for $m>0$ that there exists
a constant $c_1>0$ such that
\begin{align}
\label{detailed3}
p^m_{t}(x)\leq
\begin{cases}
c_1 t|x|^{-(d+\alpha)/2}\kappa_m(|x|) & \mbox{if $t<\Phi_{d,\alpha}^m(|x|) \wedge \frac{1}{m}$} \\
\\
c_1 t^{-d/\alpha} & \mbox{if $\Phi_{d,\alpha}^m(|x|)< t \le \frac{1}{m}$}\\
\\
c_1 m^{d/\alpha-d/2}t^{-d/2}e^{-c(m^{1/\alpha}|x|)} & \mbox{if $\frac{1}{m}<t<m^{\frac{1}{\alpha}-1}|x|$}\\
\\
c_1 m^{d/\alpha-d/2}t^{-d/2}e^{-c(m^{2/\alpha-1}|x|^2/t)} & \mbox{if $m^{\frac{1}{\alpha}-1}|x| \vee \frac{1}{m} \le t$}.
\end{cases}
\end{align}
According to this split, we define	
\begin{eqnarray*}
&& I_1^m:=\int_0^{\Phi_{d,\alpha}^m(|x|)\wedge \frac{1}{m}}e^{-\lambda t}p_t^m(x)dt,
\qquad
 I_2^m:=\begin{cases}
			\displaystyle \int_{\Phi_{d,\alpha}^m(|x|)}^{1/m}e^{-\lambda t}p_t^m(x)dt & \mbox{if $\Phi_{d,\alpha}^m(|x|)<\frac{1}{m}$}\\
			0 & \mbox{otherwise}
		\end{cases} \\
&& I_3^m:=\begin{cases}
			\displaystyle \int_{1/m}^{m^{(1-\alpha)/\alpha}|x|} e^{-\lambda t}p_t^m(x)dt & \mbox{if $|x|^\alpha>\frac{1}{m}$}\\
			0 & \mbox{otherwise,}
		\end{cases}
\qquad
I_4^m:=\int_{m^{(1-\alpha)/\alpha}|x| \vee \frac{1}{m} }^\infty e^{-\lambda t}p_t^m(x)dt,
\end{eqnarray*}
so that
\begin{equation*}
R_\lambda(x)=I_1^m+I_2^m+I_3^m+I_4^m.
\end{equation*}
We will use \eqref{detailed3} to estimate each integral from above. To estimate $I_1^m$ we need to distinguish two
cases. If $\Phi^m_{d,\alpha}(|x|) \ge \frac{1}{m}$, we have
\begin{equation*}
I_1^m \le c_1 |x|^{-\frac{d+\alpha}{2}}\kappa_m(|x|)\int_0^{\frac{1}{m}}te^{-\lambda t}dt
\le \frac{c_1}{2m^2}|x|^{-\frac{d+\alpha}{2}}\kappa_m(|x|).
\end{equation*}
However, since $\frac{1}{m^2}\le \left(\Phi^m_{d,\alpha}(|x|)\right)^2$, we get
\begin{equation*}
I_1^m \le \frac{c_1}{2}|x|^{-\frac{d-\alpha}{2}}\kappa^{\frac{d-\alpha}{d+\alpha}}_m(|x|)
\le C_{d,\alpha}\frac{c_1}{2}|x|^{-(d-\alpha)},
\end{equation*}
where we also used \eqref{cor4_1}. In case $\Phi^m_{d,\alpha}(|x|)<\frac{1}{m}$, we get
\begin{equation*}
I_1^m \le c_1 |x|^{-\frac{d+\alpha}{2}}\kappa_m(|x|)\int_0^{\Phi_m^{d,\alpha}(|x|)}te^{-\lambda t}dt
\le
\frac{c_1}{2}|x|^{-\frac{d+\alpha}{2}}\kappa_m(|x|)(\Phi^m_{d,\alpha}(|x|))^2 \le C_{d,\alpha}\frac{c_1}{2}|x|^{-(d-\alpha)}.
\end{equation*}
We only need to estimate $I_2^m$ in case $\Phi^m_{d,\alpha}(|x|)<\frac{1}{m}$. Since $d \ge 3>\alpha$, we have
\begin{equation*}
I_2^m \le c_1\int_{\Phi_{d,\alpha}^m(|x|)}^{\frac{1}{m}}t^{-\frac{d}{\alpha}}e^{-\lambda t}dt
\le c_1\int_{\Phi_{d,\alpha}^m(|x|)}^{\infty}t^{-\frac{d}{\alpha}}dt=\frac{\alpha c_1}{d-\alpha}
\left(\Phi_{d,\alpha}^m(|x|)\right)^{-\frac{d-\alpha}{\alpha}} \le C_{d,\alpha}\frac{\alpha c_1}{d-\alpha}|x|^{-(d-\alpha)}.
\end{equation*}
Next we estimate $I_3^m$. We only need to consider the case $|x|^{\alpha} \ge \frac{1}{m}$ giving
\begin{align*}
I_3^m &\le c_1m^{\frac{d}{\alpha}-\frac{d}{2}}e^{-cm^{\frac{1}{\alpha}}|x|}\int_{\frac{1}{m}}^{m^{\frac{1-\alpha}{\alpha}}|x|}
e^{-\lambda t}t^{-\frac{d}{2}}dt
\le c_1m^{\frac{d}{\alpha}-\frac{d}{2}}e^{-cm^{\frac{1}{\alpha}}|x|}\int_{\frac{1}{m}}^{\infty}t^{-\frac{d}{2}}dt
=\frac{2c_1}{d-2}m^{\frac{d}{\alpha}-1}e^{-cm^{\frac{1}{\alpha}}|x|}\\
&=\frac{2c_1}{(d-2)c^{d-\alpha}|x|^{d-\alpha}}\left(cm^{\frac{1}{\alpha}}|x|\right)^{d-\alpha}e^{-cm^{\frac{1}{\alpha}}|x|}
\le \frac{2c_1(d-\alpha)^{d-\alpha}}{(d-2)e^{d-\alpha}c^{d-\alpha}|x|^{d-\alpha}}.
\end{align*}
Finally, consider $I_4^m$. With $|x|^\alpha \le \frac{1}{m}$ we obtain
\begin{align*}
I_4^m \le c_1 m^{\frac{d}{\alpha}-\frac{d}{2}}\int_{\frac{1}{m}}^{\infty}
e^{-\lambda t}t^{-\frac{d}{2}}e^{-cm^{\frac{2}{\alpha}-1}|x|^2/t}dt \le c_1m^{\frac{d}{\alpha}-\frac{d}{2}}
\int_{\frac{1}{m}}^{\infty}t^{-\frac{d}{2}}e^{-cm^{\frac{2}{\alpha}-1}|x|^2/t}dt.
\end{align*}
The change of variable $s=cm^{\frac{2}{\alpha}-1}|x|^2/t$ gives
\begin{align*}
I_4^m \le \frac{c_1m^{\frac{2}{\alpha}-1}}{c^{\frac{d-2}{2}}|x|^{d-2}}\int_{cm^{\frac{2}{\alpha}}|x|^2}^{\infty}
s^{\frac{d}{2}-2}e^{-s}ds
\le \frac{c_1m^{\frac{2}{\alpha}-1}}{c^{\frac{d-2}{2}}|x|^{d-2}}\Gamma\Big(\frac{d}{2}-1\Big).
\end{align*}
\end{proof}
\subsubsection{Proof of estimate \eqref{eq:equivrollmassive}}
\noindent
Since
\begin{equation*}
|x|^{\alpha-d}+m^{2/\alpha-1}|x|^{2-d} \le
2(1 \vee m^{2/\alpha-1})(|x|^{\alpha-d}\mathbf{1}_{\{|x| \le 1\}}+|x|^{2-d}\mathbf{1}_{\{|x|>1\}}),
\end{equation*}
the above also shows the upper bound in \eqref{eq:equivrollmassive}. The lower bound is shown in the next statement.
\begin{lemma}
\label{lem:lowboundmassive}
Let $\alpha \in (0,2)$ and $d \ge 3$. Then, for all $m>0$ there exists a constant $C_{d,\alpha,m}>0$ such that for all
$x \in \R^d \setminus \{0\}$
\begin{equation*}
R_0^m(x) \ge C_{d,\alpha,m}(|x|^{-(d-\alpha)}\mathbf{1}_{\{|x| \le 1\}}+|x|^{-(d-2)}\mathbf{1}_{\{|x|>1\}}).
\end{equation*}
\end{lemma}
\begin{proof}
Now we use the lower bounds in \eqref{massive_heatker}. Thus there exists a constant $c_2>0$ such that
\begin{align}
\label{detailed4}
p^m_{t}(x)\geq
\begin{cases}
c_2 t|x|^{-(d+\alpha)/2}\kappa_m(|x|) & \mbox{if $t<\Phi_{d,\alpha}^m(|x|) \wedge \frac{1}{m}$} \\
\\
c_2 t^{-d/\alpha} & \mbox{if $\Phi_{d,\alpha}^m(|x|)< t \le \frac{1}{m}$}\\
\\
c_2 m^{d/\alpha-d/2}t^{-d/2}e^{-c(m^{1/\alpha}|x|)} & \mbox{if $\frac{1}{m}<t<m^{\frac{1}{\alpha}-1}|x|$}\\
\\
c_2 m^{d/\alpha-d/2}t^{-d/2}e^{-c(m^{2/\alpha-1}|x|^2/t)} & \mbox{if $m^{\frac{1}{\alpha}-1}|x| \vee\frac{1}{m} \le t$}.
\end{cases}
\end{align}	
Consider the case $|x|>1$ and let first $|x|^\alpha > \frac{1}{m}$. We obtain
\begin{equation*}
I_3^m \ge c_2 m^{\frac{d}{\alpha}-\frac{d}{2}}e^{-cm^{\frac{1}{\alpha}}|x|}
\int_{\frac{1}{m}}^{m^{\frac{1-\alpha}{\alpha}}|x|}t^{-\frac{d}{2}}dt \ge c_2 m^{\frac{d}{\alpha}-\frac{d}{2}} \int_{\frac{1}{m}}^{m^{\frac{1-\alpha}{\alpha}}|x|}t^{-\frac{d}{2}}e^{-cm^{\frac{2}{\alpha}-1}|x|^2/t}dt,
\end{equation*}
since for $t \le m^{\frac{1-\alpha}{\alpha}}|x|$ we clearly have $e^{-cm^{\frac{1}{\alpha}}|x|} \ge
e^{-cm^{\frac{2}{\alpha}-1}|x|^2/t}$. Hence for all $|x|>1$ we get
\begin{equation*}
R_0^m(x)\ge I_3^m+I_4^m \ge c_2m^{\frac{d}{\alpha}-\frac{d}{2}} \int_{\frac{1}{m}}^{\infty}t^{-\frac{d}{2}}
e^{-cm^{\frac{2}{\alpha}-1}|x|^2/t}dt.
\end{equation*}
The substitution $s=\frac{cm^{\frac{2}{\alpha}-1}|x|^2}{t}$ then gives
\begin{equation*}
R_0^m(x)\ge c_2m^{\frac{2}{\alpha}-1} c^{-\frac{d-2}{2}}|x|^{-(d-2)} \int_{0}^{cm^{\frac{2}{\alpha}}
|x|^2}s^{\frac{d}{2}-2}e^{-s}ds \ge c_2m^{\frac{2}{\alpha}-1} c^{-\frac{d-2}{2}}|x|^{-(d-2)}
\int_{0}^{cm^{\frac{2}{\alpha}}}s^{\frac{d}{2}-2}e^{-s}ds
\end{equation*}
where the integral at the right-hand side is finite since $d \ge 3$. Next consider $|x| \le 1$. If
$\Phi_{d,\alpha}^m(|x|) \le \frac{1}{m}$, then
\begin{align*}
R_0^m(x) \ge I_1^m \ge  c_2|x|^{-\frac{d+\alpha}{2}}\kappa_m(|x|)\int_0^{\Phi_{d,\alpha}^m(|x|)}t\,dt
=\frac{c_2}{2}|x|^{-\frac{d+\alpha}{2}}\kappa_m(|x|)(\Phi_{d,\alpha}^m(|x|)^2 \ge C_{d,\alpha}|x|^{-(d-\alpha)},
\end{align*}
where we used \eqref{eq:controlPhi2} and the fact that $\kappa_m(|x|) \ge C_{d,\alpha}|x|^{-\frac{d+\alpha}{2}}$
since $|x| \le 1$. If $\Phi_{d,\alpha}^m(|x|)>\frac{1}{m}$, we have
\begin{align*}
R_0^m(x) &\ge I_1^m \ge  c_2|x|^{-\frac{d+\alpha}{2}}\kappa_m(|x|)\int_0^{\frac{1}{m}}t\,dt
=\frac{c_2}{2m^2}|x|^{-\frac{d+\alpha}{2}}\kappa_m(|x|) \\
&\ge \frac{c_2}{2m^2(\Phi_{d,\alpha}^m(1))^2}|x|^{-\frac{d+\alpha}{2}}\kappa_m(|x|)(\Phi_{d,\alpha}^m(|x|))^2
\ge c_{d,\alpha,m}|x|^{-(d-\alpha)}.
\end{align*}
\end{proof}
\begin{remark}
{\rm
Notice that, up to a multiplicative constant, the term $|x|^{-(d-2)}$ represents the $0$-resolvent kernel of the
Laplacian. Indeed, we could say that the $0$-potential $R_0^m$ interpolates between the massless case $R_0(x)=
C_{d,\alpha}|x|^{-(d-\alpha)}$ and the classical Newtonian case (i.e. $\alpha=2$) that we denote here by
$R_0^{\rm N}(x)=C_d|x|^{-(d-2)}$. To observe this, we introduce a constant $c_\ell$ and write $c_\ell^2m$ in
place of $m$; this constant is physically motivated and it stands for the speed of light (which was set to $1$
in units used in mathematical physics). Taking $m \downarrow 0$ and using \cite[Lem. 3.2]{CS06} it follows that
$[\cdot]_{0,\alpha}={\rm M}-\lim_{m \to 0}[\cdot]_{c_\ell^2m,\alpha}$, which implies by an application of
\cite[Prop. 2.1]{CS06} the strong resolvent convergence of the operators $-L_{c_\ell^2m,\alpha}$ to
$-L_{0,\alpha}$ as $m \downarrow 0$. On the other hand, if we consider the classical limit $c_\ell \to \infty$,
we obtain again by \cite[Lem. 3.2]{CS06} that $[\cdot]_{0,2}={\rm M}-\lim_{c_\ell \to \infty}
[\cdot]_{c_\ell^2m,\alpha}$, where $[u]_{0,2}=\Norm{\nabla u}{2}$ for all $u \in H^1(\R^d)$. Hence
$L_{c_\ell^2m,\alpha}$ converges to $-\Delta$ in strong resolvent sense.
}
\end{remark}

\subsubsection{Proof of estimate \eqref{forrollext1}}
\noindent
We start by the following preliminary estimate.
\begin{lemma}
\label{lem3}
Let $\alpha \in (0,2)$, $d\geq1$ and $m>0$. There exists $C_{d,\alpha}>0$ such that
\begin{equation*}
R^m_\lambda(x)\leq C_{d,\alpha}
\begin{cases}
\Big(1+\frac{me^{-\lambda/m}}{\lambda}\Big)\frac{1}{|x|^{d-\alpha}} & \alpha<d \\
\\
\log\left(1+\frac{1}{\lambda |x|}\right)+1+\frac{m^{1/\alpha}}{\lambda}e^{-\lambda/m} & \alpha=d=1\\
\\
|x|^{\alpha-1}\wedge1+\lambda^{-\frac{\alpha-1}{\alpha}}+\frac{m^{1/\alpha}}{\lambda}e^{-\lambda/m} & \alpha>d=1
\end{cases}
\end{equation*}
holds for all $x\in\mathbb{R}^d$ and $\lambda > 0$.
\end{lemma}

\begin{proof}
As in the proof of Proposition \ref{lem3nol}, we have
\begin{equation*}
I_1^m\le C_{d,\alpha}|x|^{-(d-\alpha)}
\end{equation*}
if $\alpha\leq d$. If $\Phi_{d,\alpha}^m(|x|)\geq{\frac{1}{m}}$ and $\alpha>d=1$, by \eqref{masslesskappa}, there exists a small $m_0>0$ such that
\begin{equation}
\kappa_m(r)\geq2^{\frac{1+\alpha}{2}-2}\Gamma\left(\frac{1+\alpha}{2}\right)r^{-\frac{1+\alpha}{2}}\label{lem3_1}
\end{equation}
for $0<m<m_0$. We thus have
\begin{equation*}
I_1^m\leq\frac{c_1}{2}|x|^{-\frac{1-\alpha}{2}}\kappa_m^{\frac{1-\alpha}{1+\alpha}}(|x|)\leq C_{d,\alpha}|x|^{\alpha-1}
\end{equation*}
for $|x|<1$ or $0<m<m_0$ by \eqref{cor4_2} and \eqref{lem3_1} noting $(1-\alpha)/(1+d)<0$. In the case were $|x|\geq1$ and $m\geq m_0$, since $2-(1+1/\alpha)=1-1/\alpha>0$ and
\begin{equation*}
\frac{1}{m^{1+\frac{1}{\alpha}}}\leq\left(\Phi_{d,\alpha}^m(|x|)\right)^{1+\frac{1}{\alpha}},
\end{equation*}
we estimate
\begin{gather*}
I_1^m\leq\frac{c_1}{2m^2}|x|^{-\frac{1+\alpha}{2}}\kappa_m(|x|)\leq\frac{c_1}{2m_0^{1-\frac{1}{\alpha}}}|x|^{-\frac{1+\alpha}{2}}\kappa_m(|x|)\left(\Phi_{d,\alpha}^m(|x|)\right)^{1+\frac{1}{\alpha}}=\frac{c_1}{2m_0^{1-\frac{1}{\alpha}}}.
\end{gather*}
If $\Phi_{d,\alpha}^m(|x|)<{\frac{1}{m}}$ and $\alpha>d=1$, we have $I_1^m\leq C_{d,\alpha}|x|^{\alpha-1}$ for $|x|<1$ or $0<m<m_0$ in the same way above. For $|x|\geq1$ and $m\geq m_0$, we also have
\begin{gather*}
I_1^m\leq\frac{c_1}{2}|x|^{-\frac{1+\alpha}{2}}\kappa_m(|x|)(\Phi_{d,\alpha}^m(|x|))^2\nonumber\\
\leq\frac{c_1}{2m_0^{1-\frac{1}{\alpha}}}|x|^{-\frac{1+\alpha}{2}}\kappa_m(|x|)\left(\Phi_{d,\alpha}^m(|x|)\right)^{1+\frac{1}{\alpha}}=\frac{c_1}{2m_0^{1-\frac{1}{\alpha}}}
\end{gather*}
since
\begin{equation*}
\left(\Phi_{d,\alpha}^m(|x|)\right)^{1-\frac{1}{\alpha}}<\frac{1}{m^{1-\frac{1}{\alpha}}}.
\end{equation*}
If $d>\alpha$, we obtain
\begin{equation*}
I_2^m\le C_{d,\alpha}|x|^{-(d-\alpha)}.
\end{equation*}
If $d=\alpha=1$, by making use of \eqref{eq:bracket} it follows that
\begin{align*}
I_2^m \le c_1\int_{\Phi_{1,1}^m(|x|)}^{\frac{1}{m}}t^{-1}e^{-\lambda t}dt
&\le c_1\int_{\Phi_{1,1}^m(|x|)}^{\infty}t^{-1}e^{-\lambda t}dt
=c_1\int_{\lambda \Phi_{1,1}^m(|x|)}^{\infty}t^{-1}e^{- t}dt\\
&=c_1E_1(\lambda\Phi_{1,1}^m(|x|)) \le c_1\log\left(1+\frac{1}{C_{d,\alpha}\lambda|x|}\right)
\le C_{d,\alpha}\log\left(1+\frac{1}{\lambda|x|}\right).
\end{align*}
Finally, if $\alpha>d=1$, then
\begin{equation*}
I_2^m \le c_1\int_{\Phi_{1,\alpha}^m(|x|)}^{\frac{1}{m}}t^{-\frac{1}{\alpha}}e^{-\lambda t}dt
\le c_1\int_{0}^{\infty}t^{-\frac{1}{\alpha}}e^{-\lambda t}dt
=c_1\lambda^{-\frac{\alpha-1}{\alpha}}\Gamma\left(\frac{\alpha-1}{\alpha}\right).
\end{equation*}
Next we estimate $I_3^m$. It is necessary to do this only if $|x|^\alpha>\frac{1}{m}$. In this case by
\eqref{detailed3} we get
\begin{equation}
\label{eq:I3m}
I_3^m \le c_1 m^{\frac{d}{\alpha}-\frac{d}{2}}e^{-cm^{\frac{1}{\alpha}}|x|}
\int_{\frac{1}{m}}^{m^{\frac{1-\alpha}{\alpha}}|x|}e^{-\lambda t}t^{-\frac{d}{2}}dt
\le \frac{c_1 m^{\frac{d}{\alpha}}}{\lambda}e^{-cm^{\frac{1}{\alpha}}|x|-\frac{\lambda}{m}}.
\end{equation}
It is clear that if $\alpha \ge d=1$
\begin{equation*}
I_3^m \le C_{1,\alpha}\frac{m^{\frac{1}{\alpha}}}{\lambda}e^{-\frac{\lambda}{m}}.
\end{equation*}
For $\alpha<d$, by \eqref{eq:ineqexpon} we have
\begin{equation*}
I_3^m \le \frac{c_1 m}{\lambda c^{d-\alpha}|x|^{d-\alpha}}\left(cm^{\frac{1}{\alpha}}|x|\right)^{d-\alpha}
e^{-cm^{\frac{1}{\alpha}}|x|}e^{-\frac{\lambda}{m}}
\le \frac{c_1 m (d-\alpha)^{d-\alpha}}{\lambda c^{d-\alpha}e^{d-\alpha}|x|^{d-\alpha}}e^{-\frac{\lambda}{m}}.
\end{equation*}
Finally, for $I_4^m$ we have two cases. If $|x|^\alpha \le \frac{1}{m}$, then for $d>\alpha$ we obtain
\begin{align*}
I_4^m &\le c_1 m^{\frac{d}{\alpha}-\frac{d}{2}}\int_{\frac{1}{m}}^{\infty}e^{-\lambda t}t^{-\frac{d}{2}}
e^{-cm^{\frac{2}{\alpha}-1}|x|^2/t}dt\\
&\le
\frac{c_1 m^{1-\frac{\alpha}{2}}}{c^{\frac{d-\alpha}{2}}|x|^{d-\alpha}}\int_{\frac{1}{m}}^{\infty}
e^{-\lambda t}t^{-\frac{\alpha}{2}} \left(cm^{\frac{2}{\alpha}-1}|x|^2/t\right)^{\frac{d-\alpha}{2}}
e^{-cm^{\frac{2}{\alpha}-1}|x|^2/t}dt \\
&\le
\frac{c_1 m^{1-\frac{\alpha}{2}}}{e^{\frac{d-\alpha}{2}}c^{\frac{d-\alpha}{2}}|x|^{d-\alpha}}
\left(\frac{d-\alpha}{2}\right)^{\frac{d-\alpha}{2}}\int_{\frac{1}{m}}^{\infty}e^{-\lambda t}
t^{-\frac{\alpha}{2}} dt \le \frac{c_2 m e^{-\frac{\lambda}{m}}}{e^{\frac{d-\alpha}{2}}
c^{\frac{d-\alpha}{2}}\lambda |x|^{d-\alpha}}\left(\frac{d-\alpha}{2}\right)^{\frac{d-\alpha}{2}}.
\end{align*}
For $\alpha \ge d=1$ we estimate like
\begin{align*}
I_4^m \le c_1 m^{\frac{1}{\alpha}-\frac{1}{2}}\int_{\frac{1}{m}}^{\infty}e^{-\lambda t}t^{-\frac{1}{2}}
e^{-cm^{\frac{2}{\alpha}-1}|x|^2/t}dt \le c_1 \frac{m^{\frac{1}{\alpha}}}{\lambda}e^{-\frac{\lambda}{m}}.
\end{align*}
In case $|x|^{\alpha}>\frac{1}{m}$, we obtain
\begin{align*}
I_4^m &\le c_1 m^{\frac{d}{\alpha}-\frac{d}{2}}\int_{m^{\frac{1-\alpha}{\alpha}}|x|}^{\infty}e^{-\lambda t}
t^{-\frac{d}{2}} e^{-cm^{\frac{2}{\alpha}-1}|x|^2/t}dt\le c_1 m^{\frac{d}{\alpha}-\frac{d}{2}}
\int_{\frac{1}{m}}^{\infty}e^{-\lambda t}t^{-\frac{d}{2}} e^{-cm^{\frac{2}{\alpha}-1}|x|^2/t}dt
\end{align*}
and then we proceed as in the previous case.
\end{proof}

The following estimate is sharper for large values of $|x|$.
\begin{lemma}
\label{lem5}
Let $\alpha \in (0,2)$, $d \geq 1$ and $m>0$. Then there exists $C_{d,\alpha}>0$ such that for all
$x \in \R^d\setminus \{0\}$ and $\lambda>0$ we have
\begin{equation*}
R^m_\lambda(x)\leq C_{d,\alpha}\left(\frac{1}{\lambda^2}+\frac{e^{-\frac{\lambda}{m}}} {m\lambda}
+\frac{e^{-\frac{\lambda}{m}}}{\lambda^2}\right)\frac{1}{|x|^{d+\alpha}}.\label{lem5_1}
\end{equation*}
\end{lemma}

\begin{proof}
We estimate $I^m_j$, $j=1,2,3,4$, in the proof of Lemma \ref{lem3} in different ways. First consider $I_1^m$. By \eqref{cor4_1} we get
	\begin{equation*}
	I_1^m \le c_1|x|^{-\frac{d+\alpha}{2}}\kappa_m(|x|)\int_0^{\min\left\{\Phi_{d,\alpha}^m(|x|),\frac{1}{m}\right\}}te^{-\lambda t}dt \le c_1|x|^{-\frac{d+\alpha}{2}}\kappa_m(|x|)\int_0^{\infty}te^{-\lambda t}dt \le \frac{C_{d,\alpha}}{\lambda^2 |x|^{d+\alpha}}.
	\end{equation*}
	Next consider $I_2^m$. We need to obtain an estimate only for the case $\Phi_{d,\alpha}^m(|x|)<\frac{1}{m}$. We have
	\begin{align*}
		I_2^m &\le c_1\int_{\Phi_{d,\alpha}^m(|x|)}^{\frac{1}{m}}t^{-\frac{d}{\alpha}}e^{-\lambda t}dt \le \frac{2 c_1}{\lambda e}\int_{\Phi_{d,\alpha}^m(|x|)}^{\infty}t^{-\frac{d}{\alpha}-1}e^{-\frac{\lambda t}{2}}dt \le \frac{2 \alpha  c_1}{d \lambda e}e^{-\frac{\lambda \Phi_{d,\alpha}^m(|x|)}{2}}(\Phi_{d,\alpha}^m(|x|))^{-\frac{d}{\alpha}}\\
		&=\frac{4 \alpha  c_1}{d \lambda^2 e}\left(\frac{\lambda \Phi_{d,\alpha}^m(|x|)}{2}\right)
e^{-\frac{\lambda \Phi_{d,\alpha}^m(|x|)}{2}}(\Phi_{d,\alpha}^m(|x|))^{-\frac{d}{\alpha}-1} \le \frac{4 \alpha  c_1}{d \lambda^2 e^2}
(\Phi_{d,\alpha}^m(|x|))^{-\frac{d+\alpha}{\alpha}} \le \frac{C_{d,\alpha}}{\lambda^2|x|^{d+\alpha}},
	\end{align*}
	where we also used \eqref{eq:controlPhi2}. Next we need to control $I_3^m$. We only need to consider the case $|x|^\alpha>\frac{1}{m}$.
Starting from \eqref{eq:I3m} and using \eqref{eq:ineqexpon} we obtain
	\begin{align*}
I_3^m \le \frac{c_1 m^{\frac{d}{\alpha}}}{\lambda}e^{-cm^{\frac{1}{\alpha}}|x|-\frac{\lambda}{m}}&=\frac{c_1}{\lambda c^{d+\alpha}m|x|^{d+\alpha}}\left(cm^{\frac{1}{\alpha}}|x|\right)^{d+\alpha}e^{-cm^{\frac{1}{\alpha}}|x|-\frac{\lambda}{m}} \\
&\le \frac{(d+\alpha)^{d+\alpha}c_1}{\lambda e^{d+\alpha}c^{d+\alpha}m|x|^{d+\alpha}}e^{-\frac{\lambda}{m}}.
	\end{align*}
Finally, we consider $I_4^m$. If $|x|^\alpha\leq\frac{1}{m}$, we have
\begin{align*}
I_4^m&\leq c_1 m^{\frac{d}{\alpha}-\frac{d}{2}}\int_{\frac{1}{m}}^{\infty}e^{-\lambda t}t^{-\frac{d}{2}} e^{-cm^{\frac{2}{\alpha}-1}|x|^2/t}dt\\
&\leq\frac{c_1 m^{\frac{\alpha}{2}-1}}{c^{\frac{d+\alpha}{2}}|x|^{d+\alpha}}\int_{\frac{1}{m}}^{\infty}e^{-\lambda t}t^{\frac{\alpha}{2}} \left(cm^{\frac{2}{\alpha}-1}|x|^2/t\right)^{\frac{d+\alpha}{2}} e^{-cm^{\frac{2}{\alpha}-1}|x|^2/t}dt\\
&\leq\frac{c_1 m^{\frac{\alpha}{2}-1}}{e^{\frac{d+\alpha}{2}}c^{\frac{d+\alpha}{2}}|x|^{d+\alpha}}\left(\frac{d+\alpha}{2}\right)^{\frac{d+\alpha}{2}}\int_{\frac{1}{m}}^{\infty}e^{-\lambda t}t^{\frac{\alpha}{2}} dt\\
&\leq\frac{c_1}{e^{\frac{d+\alpha}{2}}c^{\frac{d+\alpha}{2}}|x|^{d+\alpha}}\left(\frac{d+\alpha}{2}\right)^{\frac{d+\alpha}{2}}\left(\frac{e^{-\frac{\lambda}{m}}}{m\lambda}+\frac{\alpha e^{-\frac{\lambda}{m}}}{2\lambda^2}\right)
\end{align*}
where we computed
\begin{gather*}
\int_{\frac{1}{m}}^{\infty}e^{-\lambda t}t^{\frac{\alpha}{2}}dt=\frac{e^{-\frac{\lambda}{m}}}{\lambda m^{\frac{\alpha}{2}}}+\frac{\alpha}{2\lambda}\int_{\frac{1}{m}}^{\infty}e^{-\lambda t}t^{\frac{\alpha}{2}-1}dt\leq\frac{e^{-\frac{\lambda}{m}}}{\lambda m^{\frac{\alpha}{2}}}+\frac{\alpha}{2\lambda m^{\frac{\alpha}{2}-1}}\int_{\frac{1}{m}}^{\infty}e^{-\lambda t}dt
\end{gather*}
integrating by parts and noting $\alpha/2-1<0$. If $|x|^\alpha>\frac{1}{m}$, we have instead
	\begin{align*}
		I_4^m &\le c_1 m^{\frac{d}{\alpha}-\frac{d}{2}}\int_{m^{\frac{1-\alpha}{\alpha}}|x|}^{\infty}e^{-\lambda t}t^{-\frac{d}{2}} e^{-cm^{\frac{2}{\alpha}-1}|x|^2/t}dt \le c_1 m^{\frac{d}{\alpha}-\frac{d}{2}}\int_{\frac{1}{m}}^{\infty}e^{-\lambda t}t^{-\frac{d}{2}} e^{-cm^{\frac{2}{\alpha}-1}|x|^2/t}dt
	\end{align*}
	and we proceed as in the previous case.
\end{proof}


Combining Lemmas \ref{lem3} and \ref{lem5}´we get the following result.
\begin{proposition}
\label{cor:comparison2}
Let $\alpha \in (0,2)$, $d \ge 1$. Then for all $\lambda>0$ and $m > 0$ there exists a constant
$C_{d,\alpha,\lambda,m}$ such that
\begin{equation*}
R_\lambda^m(x) \le C_{d,\alpha,\lambda,m}H^0_{d,\alpha}(|x|),
\end{equation*}
for all $x \in \R^d\setminus \{0\}$, where $H^0_{d,\alpha}$ is defined in \eqref{eq:Hda0}.
\end{proposition}
This shows \eqref{forrollext1} for $m>0$.

\subsubsection{Proof of estimate \eqref{eq:equivrollext2}}
\noindent
To provide the upper bound we improve the bound given in Lemma \ref{lem5} for large values of $\lambda$.
\begin{lemma}
\label{lem:Rlambdam}
Let $\alpha \in (0,2)$, $d \geq 1$ and $m>0$. Then there exists a constant $C_{d,\alpha}>0$ such that for
all $\lambda>m$ and $x \in \R^d \setminus \{0\}$
\begin{equation*}
R^m_\lambda(x) \le C_{d,\alpha} \frac{|x|^{-\frac{d+\alpha}{2}}\kappa_m(|x|)}{(\lambda-m)^2}.
\end{equation*}
\end{lemma}
\begin{proof}
By \cite[eq. (2.16)]{CKS12} there exists a constant $c_1$ dependent on $\alpha$ and $d$ such that for
all $t>0$ and $x \in \R^d \setminus \{0\}$
\begin{equation*}
p_t(x) \le c_1 t e^{mt}|x|^{-\frac{d+\alpha}{2}}\kappa_m(|x|)
\end{equation*}
If $\lambda>m$, we have
\begin{equation*}
R^m_\lambda(x) \le c_1 |x|^{-\frac{d+\alpha}{2}}\kappa_m(|x|)\int_0^{\infty}te^{-(\lambda-m)t}dt
\le c_1 \frac{|x|^{-\frac{d+\alpha}{2}}\kappa_m(|x|)}{(\lambda-m)^2}.
\end{equation*}
\end{proof}
Combining the previous estimate with Lemma \ref{lem3} gives the following result.
\begin{proposition}
\label{cor:comparison3}
Let $\alpha \in (0,2)$, $d \ge 1$. Then for all $m > 0$ and $\lambda>m$, there exists a constant
$C_{d,\alpha,\lambda,m}$ such that for all $x \in \R^d\setminus \{0\}$
\begin{equation*}
R_\lambda^m(x) \le C_{d,\alpha,\lambda,m}H^m_{d,\alpha}(|x|)
\end{equation*}
where $H^m_{d,\alpha}$ is defined in \eqref{eq:Hmda}
\end{proposition}

\begin{remark}
\label{rmk:consistent}
{\rm
We note that \eqref{masslesskappa} gives $\lim_{m \downarrow 0}H^m_{d,\alpha}(r)=H^0_{d,\alpha}(r)$.
}
\end{remark}

\begin{remark}\label{rmk:classical}
		We note that in the classical case ($\alpha=2$, while the value of $m$ is not influent in such a case) for $d \ge 2$ the resolvent kernel is given by
		\begin{equation}\label{eq:reskernclass}
			R_\lambda(x)=\frac{\lambda^{\frac{d-2}{2}}}{(2\pi)^{\frac{d}{2}}}|x|^{-\frac{d-2}{2}}K_{\frac{d-2}{2}}(\sqrt{\lambda}|x|),
		\end{equation}
		as shown in \cite[Equation (4.2.20)]{LHB}. For $d \ge 3$, we can control
		\begin{equation*}
			R_\lambda(x) \le \int_0^{+\infty}p_t(x)\, dt=\frac{\Gamma\left(\frac{d}{2}-1\right)}{4\pi^{\frac{d}{2}}}\frac{1}{|x|^{d-2}},
		\end{equation*}
		where $p_t(x)=\frac{1}{(4\pi t)^{\frac{d}{2}}}e^{-\frac{|x|^2}{4t}}$ is the classical heat kernel. This corresponds to Proposition \ref{lem1} in case $2=\alpha<d$. However, notice that this upper bound is not sharp, since for $\sqrt{\lambda}|x| \ge 1$ we have, by \cite[Lemma 4.97]{LHB},
		\begin{equation*}
			R_\lambda(x) \le C_d \frac{\lambda^{\frac{d-3}{4}}}{|x|^{\frac{d-1}{2}}}e^{-\sqrt{\lambda}|x|}.
		\end{equation*} 
		Thus the behaviour of $R_\lambda$ for large values of $|x|$ is analogous to the one presented in Proposition \ref{cor:comparison3}, where we recall that for $m>0$ and $|x| \ge 1$ we have
		\begin{equation*}
			|x|^{-\frac{d+\alpha}{2}}\kappa_m(|x|) \le C_{d,\alpha,m} |x|^{-\frac{d+\alpha+1}{2}}e^{-m^{\frac{1}{\alpha}}|x|}.
		\end{equation*}
		Concerning the case $\alpha=d=2$, we recall that, by \cite[Ch. 9, eqs. 9.6.8, 9.7.2]{AS},
		\begin{equation}
			\label{besselasymp0}
			K_0 (z) \sim \sqrt{\frac{\pi}{2z}}\, e^{-z} \quad \mbox{as $z\to\infty$}
			\qquad \mbox{and} \qquad
			K_0(z)\sim -\log(z) \quad \mbox{as $z \downarrow 0$}.
		\end{equation}
		Furthermore, by \cite[Lemmas 4.97, 4.98]{LHB}, for $\lambda>0$,
		\begin{equation}\label{eq:da20}
			R_\lambda(x) \le C_d\left(\log\left(1+\frac{1}{\sqrt{\lambda}|x|}\right)\mathbf{1}_{\sqrt{\lambda}|x| \le 1}+\frac{1}{\lambda^{\frac{1}{4}}|x|^{\frac{1}{2}}}e^{-\sqrt{\lambda}|x|}\mathbf{1}_{\sqrt{\lambda}|x|>1}\right).
		\end{equation}
		Once we note that for $\sqrt{\lambda}|x|>1$ it holds
		\begin{equation*}
			\frac{1}{\lambda^{\frac{1}{4}}|x|^{\frac{1}{2}}}e^{-\sqrt{\lambda}|x|} \le e^{-1}
		\end{equation*}		
		then we get
		\begin{equation}\label{eq:da2}
			R_\lambda(x) \le C_d\left(\log\left(1+\frac{1}{\sqrt{\lambda}|x|}\right)+1\right).
		\end{equation}
		The latter corresponds to the case $\alpha=d$ in Proposition \ref{lem1}. Again, for fixed $\lambda>0$, \eqref{eq:da20} is sharper than \eqref{eq:da2} and it is analogous to the case $\alpha=d$ in Proposition \ref{cor:comparison3}. Finally, in case $d=1$, by \cite[Equation (4.2.18)]{LHB}, we have the explicit expression
		\begin{equation}\label{eq:res1clas}
			R_\lambda(x)=\frac{e^{-\sqrt{\lambda}|x|}}{\sqrt{\lambda}}.
		\end{equation}
		that clearly satisfies
		\begin{equation*}
			R_\lambda(x) \le (|x| \wedge 1+\lambda^{-\frac{1}{2}}),
		\end{equation*}
		that is the case $\alpha>d$ in Proposition \ref{lem1}. Finally, it is clear by \eqref{eq:res1clas} that for fixed $\lambda>0$ we have an exponential decay as $|x| \to \infty$, analogously to the case $\alpha>d$ in Proposition \ref{cor:comparison3}.
\end{remark}
\noindent
This proves the upper bound in \eqref{eq:equivrollext2}. Finally, we prove the lower bound in
\eqref{eq:equivrollext2},  together with \eqref{eq:lowerrollext2}.
\begin{lemma}
\label{lem:lowerboundRml}
Let $\alpha \in (0,2)$ and $d\geq 1$. Then for all $\lambda>0$ and $m > 0$ there exists a constant
$C_{d,\alpha,\lambda,m}>0$ such that
\begin{equation*}
R_\lambda^m(x) \ge C_{d,\alpha,\lambda,m}H^m_{d,\alpha}(|x|),
\end{equation*}
 for all $x \in \R^d \setminus \{0\}$, where $H^m_{d,\alpha}$ is defined in \eqref{eq:Hmda}.
\end{lemma}
\begin{proof}
Throughout this proof we will use the lower bounds provided in \eqref{detailed4}. First we give
an estimate for $|x|>1$. We need to distinguish among two cases. If $\Phi_{d,\alpha}^m(|x|) \ge
\frac{1}{m}$, we have
\begin{align*}
		R_\lambda^m(x) &\ge \int_0^{\frac{1}{m}}e^{-\lambda t}p_t^m(x)dt
		\ge c_2|x|^{-\frac{d+\alpha}{2}} \kappa_m(|x|)\int_0^{\frac{1}{m}}te^{-\lambda t}dt \\
		&= \frac{c_2|x|^{-\frac{d+\alpha}{2}}\kappa_m(|x|)}{\lambda^2}\left(1-e^{-\frac{\lambda}{m}}
		-\frac{\lambda}{m}e^{-\frac{\lambda}{m}}\right)\ge \frac{c_2}{2}m^{-2}|x|^{-\frac{d+\alpha}{2}}
		\kappa_m(|x|)e^{-\lambda/m}.
\end{align*}
If $\Phi_{d,\alpha}^m(|x|) \le \frac{1}{m}$, we have
\begin{align*}
		R_\lambda^m(x) &\ge \int_0^{\Phi_{d,\alpha}^m(|x|)}e^{-\lambda t}p_t^m(x)dt
		\ge c_2|x|^{-\frac{d+\alpha}{2}} \kappa_m(|x|)\int_0^{\Phi_{d,\alpha}^m(|x|)}te^{-\lambda t}dt \\
		&\ge c_2|x|^{-\frac{d+\alpha}{2}} \kappa_m(|x|)\int_0^{\Phi_{d,\alpha}^m(1)}te^{-\lambda t}dt\\
		&\ge  \frac{c_2|x|^{-\frac{d+\alpha}{2}}\kappa_m(|x|)}{\lambda^2}\left(1-e^{-\lambda\Phi_{d,\alpha}^m(1)}
		-\lambda \Phi_{d,\alpha}^m(1) e^{-\lambda\Phi_{d,\alpha}^m(1)}\right),
\end{align*}
where we used that $\Phi_{d,\alpha}^m$ is strictly increasing. Next let $|x| \le 1$ and consider the following
cases. If $d>\alpha$ and $\Phi_{d,\alpha}^m(|x|)\le \frac{1}{m}$, we get
\begin{align*}
R_\lambda^m(x) &\ge \int_0^{\Phi_{d,\alpha}^m(|x|)}e^{-\lambda t}p_t^m(x)dt
\ge c_2 |x|^{-\frac{d+\alpha}{2}}\kappa_m(|x|)\int_0^{\Phi_{d,\alpha}^m(|x|)}te^{-\lambda t}dt\\
&=\frac{c_2 |x|^{-\frac{d+\alpha}{2}}\kappa_m(|x|)}{\lambda^2}e^{-\lambda \Phi_{d,\alpha}^m(|x|)}
\left(e^{\lambda \Phi_{d,\alpha}^m(|x|)}-1-\lambda \Phi_{d,\alpha}^m(|x|)\right)\\
&\ge c_2 |x|^{-\frac{d+\alpha}{2}}\kappa_m(|x|)(\Phi_{d,\alpha}^m(|x|))^2e^{-\lambda \Phi_{d,\alpha}^m(|x|)}.
\end{align*}
Now observe that for $|x| \le 1$ the second asymptotic relation in \eqref{besselasymp} implies that there
exists a constant $C_{d,\alpha}$ such that $\kappa_m(|x|) \ge C_{d,\alpha}|x|^{-\frac{d+\alpha}{2}}$. Hence,
by also using \eqref{eq:controlPhi2} and the fact that $\Phi_{d,\alpha}^m$ is increasing, we obtain
\begin{align*}
R_\lambda^m(x) \ge C_{d,\alpha} |x|^{-(d-\alpha)}e^{-\lambda \Phi_{d,\alpha}^m(1)}.
\end{align*}
Next, if $d>\alpha$ and $\Phi_{d,\alpha}^m(|x|) > \frac{1}{m}$, then we have
\begin{align*}
R_\lambda^m(x) &\ge \int_0^{\frac{1}{m}}e^{-\lambda t}p_t^m(x)dt \ge c_2 |x|^{-\frac{d+\alpha}{2}}
\kappa_m(|x|)\int_0^{\frac{1}{m}}te^{-\lambda t}dt\\
&=\frac{c_2 |x|^{-\frac{d+\alpha}{2}}\kappa_m(|x|)}{\lambda^2}e^{-\frac{\lambda}{m}}
\left(e^{\frac{\lambda}{m} }-1-\frac{\lambda}{m}\right)\\
&\ge \frac{c_2 |x|^{-\frac{d+\alpha}{2}}\kappa_m(|x|)}{m^{2}} e^{-\frac{\lambda}{m}}\ge \frac{c_2 |x|^{-\frac{d+\alpha}{2}}\kappa_m(|x|)(\Phi_{d,\alpha}^m(|x|))^2}{(\Phi_{d,\alpha}^m(1))^2m^{2}} e^{-\frac{\lambda}{m}}
\ge C_{d,\alpha}\frac{|x|^{-(d-\alpha)}}{(\Phi_{d,\alpha}^m(|1|))^2m^{2}} e^{-\frac{\lambda}{m}}.
\end{align*}
Next assume that $d=\alpha=1$ and let $\Phi_{1,1}^m(|x|) \le \frac{1}{m}$. Then we start from
\begin{align*}
R_\lambda^m(x) \ge I_2^m+I_3^m+I_4^m
\end{align*}
and we first search for some lower bounds of $I_3^m$ and $I_4^m$. Again, we have to distinguish two cases.
If $|x| \le \frac{1}{m}$, then $I_3^m=0$ and we have
\begin{equation*}
I_4^m \ge c_2 m^{\frac{1}{2}}\int_{\frac{1}{m}}^{\infty}e^{-\lambda t}t^{-\frac{1}{2}}e^{-cm|x|^2/t}dt.
\end{equation*}
Now observe that, for fixed $|x|$, the function $f(r)=r^{-\frac{1}{2}}e^{-cm|x|^2r}$ is decreasing, hence, for
$t>\frac{1}{m}$ and $|x| \le 1$, we have
\begin{equation}
\label{eq:controlexp2}
t^{-\frac{1}{2}}e^{-cm|x|^2/t} \ge m^{-\frac{1}{2}}e^{-cm^{2}|x|^2}t^{-1} \ge m^{-\frac{1}{2}}e^{-cm^{2}}t^{-1}.
\end{equation}
	Thus
	\begin{equation*}
		I_4^m \ge c_2 e^{-cm^{2}} \int_{\frac{1}{m}}^{\infty}t^{-1}e^{-\lambda t}dt.
	\end{equation*}
	If $|x|>\frac{1}{m}$, we have
	\begin{equation*}
		I_3^m \ge c_2 m^{\frac{1}{2}}e^{-cm|x|}\int_{\frac{1}{m}}^{|x|}t^{-\frac{1}{2}}e^{-\lambda t}dt \ge e^{-cm}\int_{\frac{1}{m}}^{|x|}t^{-1}e^{-\lambda t}dt.
	\end{equation*}	
	Furthermore
	\begin{equation*}
		I_4^m \ge c_2 m^{\frac{1}{2}}\int_{|x|}^{\infty}e^{-\lambda t}t^{-\frac{1}{2}}e^{-cm^{\frac{2}{\alpha}-1}|x|^2/t}dt \ge c_2 e^{-cm^{2}} \int_{|x|}^{\infty}t^{-1}e^{-\lambda t}dt,
	\end{equation*}
	where we used again \eqref{eq:controlexp2}, since $t>|x|>\frac{1}{m}$ and $|x| \le 1$. In general, we have
	\begin{equation*}
		I_3^m+I_4^m \ge C_{1,1,m}\int_{\frac{1}{m}}^{\infty}t^{-1}e^{-\lambda t}dt
	\end{equation*}
	and then
	\begin{align*}
		R_\lambda^m(x) \ge I_2^m+I_3^m+I_4^m &\ge C_{1,1,m}\int_{\Phi_{1,1}^m(|x|)}^{\infty}t^{-1}e^{-\lambda t}dt=C_{1,1,m}E_1(\lambda\Phi_{1,1}^m(|x|)) \\
		&\ge \frac{C_{1,1,m}e^{-\lambda\Phi_{1,1}^m(|x|)}}{2}\log\left(1+\frac{2}{\lambda\Phi_{1,1}^m(|x|)}\right),
	\end{align*}
	where we used \eqref{eq:bracket}. Since $|x| \le 1$, we can use \eqref{eq:controlPhi} to finally obtain
	\begin{align*}
		R_\lambda^m(x) \ge C_{1,1,m,\lambda}\log\left(1+\frac{1}{|x|}\right).
	\end{align*}
	Now consider the case $d=\alpha=1$ and $\Phi^m_{1,1}(|x|)>\frac{1}{m}$. Arguing as before, we have
	\begin{align*}
		R_\lambda^m(x)&\ge I_3^m+I_4^m \ge C_{1,1,m}\int_{\frac{1}{m}}^{\infty}t^{-1}e^{-\lambda t}dt \ge C_{1,1,m}\int_{\Phi_{1,1}^m(|x|)}^{\infty}t^{-1}e^{-\lambda t}dt\\
		&=C_{1,1,m}E_1(\lambda\Phi_{1,1}^m(|x|)) \ge \frac{C_{1,1,m}e^{-\lambda\Phi_{1,1}^m(|x|)}}{2}\log\left(1+\frac{2}{\lambda\Phi_{1,1}^m(|x|)}\right) \ge C_{1,1,m,\lambda}\log\left(1+\frac{1}{|x|}\right).
	\end{align*}
	Finally, let $\alpha>d=1$ and make the split
	\begin{equation*}
		R_\lambda^m(x) \ge I_3^m+I_4^m.
	\end{equation*}
	If $|x| \le \frac{1}{m}$, then $I_3^m=0$ and
	\begin{equation*}
		I_4^m \ge c_2 m^{\frac{1}{\alpha}-\frac{1}{2}}\int_{\frac{1}{m}}^{\infty}e^{-\lambda t}t^{-\frac{1}{2}}e^{-cm^{\frac{2}{\alpha}-1}|x|^2/t}dt.
	\end{equation*}
	Since $f_\alpha(r)=r^{\frac{1}{2}-\frac{1}{\alpha}}e^{-cm^{\frac{2}{\alpha}-1}|x|^2r}$ is decreasing, for $t>\frac{1}{m}$ and $|x| \le 1$ we get
	\begin{equation}\label{eq:controlexp3}
		t^{-\frac{1}{2}}e^{-cm^{\frac{2}{\alpha}-1}|x|^2/t}dt \ge m^{\frac{1}{2}-\frac{1}{\alpha}}e^{-cm^{\frac{2}{\alpha}}}t^{-\frac{1}{\alpha}}
	\end{equation}
	and then
	\begin{equation*}
		I_4^m \ge c_2 e^{-cm^{\frac{2}{\alpha}}}\int_{\frac{1}{m}}^{\infty}e^{-\lambda t}t^{-\frac{1}{\alpha}}dt.
	\end{equation*}
	If  $|x|>\frac{1}{m}$, we have
	\begin{equation*}
		I_3^m \ge c_2 m^{\frac{1}{\alpha}-\frac{1}{2}}e^{-cm^{\frac{1}{\alpha}}|x|}\int_{\frac{1}{m}}^{m^{\frac{1-\alpha}{\alpha}}|x|}
t^{-\frac{1}{2}}e^{-\lambda t}dt \ge e^{-cm^\frac{1}{\alpha}}\int_{\frac{1}{m}}^{m^{\frac{1-\alpha}{\alpha}}|x|}t^{-\frac{1}{\alpha}}
e^{-\lambda t}dt
	\end{equation*}	
	and
	\begin{equation*}
		I_4^m \ge c_2 e^{-cm^{\frac{2}{\alpha}}}\int_{m^{\frac{1-\alpha}{\alpha}}|x|}^{\infty}
e^{-\lambda t}t^{-\frac{1}{\alpha}}dt,
	\end{equation*}
	by \eqref{eq:controlexp3} since $t>m^{\frac{1-\alpha}{\alpha}}|x|>\frac{1}{m}$. Thus we finally obtain
	\begin{equation*}
		R_\lambda^m(x) \ge I_3^m+I_4^m \ge C_{1,\alpha,\lambda,m}\int_{\frac{1}{m}}^{\infty}e^{-\lambda t}t^{-\frac{1}{\alpha}}dt=C_{1,\alpha,\lambda,m}\Gamma\left(1-\frac{1}{\alpha},\frac{\lambda}{m}\right).
	\end{equation*}
\end{proof}
\begin{remark}
We can actually get similar same lower bounds as in Lemma \ref{lem:lowerboundRml} in the classical case. Indeed, if $\alpha=2$ and $d \ge 3$ the asymptotic behaviour in \eqref{besselasymp} tells us that
\begin{equation*}
	R_\lambda(x) \ge C_{d}\left(\lambda^{-\frac{1}{2}}|x|^{-(d-2)}\mathbf{1}_{\{\sqrt{\lambda}|x| \le 1\}}+\lambda^{\frac{d-5}{4}}|x|^{-\frac{1}{2}}e^{-\sqrt{\lambda}|x|}\mathbf{1}_{\{\sqrt{\lambda}|x|>1\}}\right),
\end{equation*}
while for $d=2$ we get from \eqref{besselasymp0}
\begin{equation*}
	R_\lambda(x) \ge C\left(\lambda^{-\frac{1}{2}}\log\left(1+\frac{1}{\sqrt{\lambda}|x|}\right)\mathbf{1}_{\{\sqrt{\lambda}|x| \le 1\}}+\lambda^{\frac{d-5}{4}}|x|^{-\frac{1}{2}}e^{-\sqrt{\lambda}|x|}\mathbf{1}_{\{\sqrt{\lambda}|x|>1\}}\right).
\end{equation*}
In case $d=1$, we have from \eqref{eq:res1clas}
\begin{equation*}
	R_\lambda(x) \ge C(\lambda^{-\frac{1}{2}}\mathbf{1}_{\{\sqrt{\lambda}|x| \le 1 \}}+\lambda^{-\frac{1}{2}}e^{-\sqrt{\lambda}|x|}\mathbf{1}_{\{\sqrt{\lambda}|x|>1\}}).
\end{equation*}
This provides a corresponding lower bound to the classical resolvent estimate in Remark \ref{rmk:classical}.
\end{remark}

\section*{Acknowledgments}
\noindent
G.A. is supported by the PRIN 2022XZSAFN Project ``Anomalous Phenomena on Regular and Irregular Domains:
Approximating Complexity for the Applied Sciences'' CUP$\_$E53D23005970006  and by the GNAMPA-INdAM group.
A.I. thanks JSPS KAKENHI Grant Numbers 20K03625, 21K03279 and 21KK0245. J.L. thanks the hospitality of
Tokyo University of Science on a research visit during which part of this paper has been written.

\end{document}